\documentclass[11pt]{article}

\usepackage{graphicx}
\graphicspath{ {./images/} }

\usepackage[dvipsnames]{xcolor}
\definecolor{b}{HTML}{4472c4}
\definecolor{o}{HTML}{ED7D31}
\definecolor{g}{HTML}{70ad47}

\definecolor{t}{RGB}{40,154,150}

\usepackage{fullpage}
\usepackage{amsmath,amsthm,amssymb,latexsym,color}
\usepackage[hidelinks,unicode]{hyperref}
\usepackage{lipsum}
\usepackage{enumitem,bbm}
\usepackage{hyperref}

\usepackage{tikz}
\usepackage{subfigure}

\usepackage{authblk}

\usepackage{multirow}

\usepackage{algorithmicx,algorithm,algpseudocode}
\algnewcommand{\Input}{\item[\textbf{Input:}]}
\algnewcommand{\Output}{\item[\textbf{Output:}]}

\newtheorem{theorem}{Theorem}[section]
\newtheorem*{theorem*}{Theorem}

\newtheorem{lemma}[theorem]{Lemma}
\newtheorem{cor}[theorem]{Corollary}
\newtheorem{defi}[theorem]{Definition}
\newtheorem{prop}[theorem]{Proposition}

\theoremstyle{definition}

\theoremstyle{plain}

\newtheorem{theoremA}{Theorem}

\newcommand{\N}{\mathbb{N}}

\newcommand{\R}{\mathbb{R}}

\newcommand{\Exp}{\mathbb{E}}

\newcommand{\Event}{\mathcal{E}}
\def\Prob{{\mathbb P}}

\def\Id{{\rm Id}}

\def\spn{{\rm span\,}}
\def\dist{{\rm dist}}

\def\col{{\rm col}}

\def\row{{\rm row}}
\def\dist{{\rm dist}}

\definecolor{color2}{RGB}{118,128,97} 

\def\prA{\mathcal M}

\def\izero{i_{th}}
\def\ione{i_{\max}}

\title{Average-case analysis\\
of the Gaussian Elimination with Partial Pivoting}
\author{Han Huang\footnote{School of Mathematics, University of Missouri, Columbia. Email: hhuang@missouri.edu}\quad and\quad Konstantin Tikhomirov\footnote{Department of Mathematical Sciences, Carnegie Mellon University. Email: ktikhomi@andrew.cmu.edu. K.T.\ was partially supported by the Sloan Research Fellowship and by the NSF grant DMS 2054666}}
\begin{document}
\maketitle

\begin{abstract}
The Gaussian Elimination with Partial Pivoting (GEPP) is a classical algorithm for solving systems of linear equations. Although in specific cases
the loss of precision in GEPP due to roundoff errors can be very significant, empirical evidence strongly suggests that for a {\it typical} square coefficient matrix, GEPP is numerically stable. We obtain a (partial)
theoretical justification of this phenomenon
by showing that, given the random $n\times n$ standard Gaussian coefficient matrix $A$, the {\it growth factor} of the Gaussian Elimination with Partial Pivoting is at most polynomially large in $n$ with probability close to one.
This implies that with probability close to one the number of bits of precision sufficient to solve $Ax = b$ to $m$ bits of accuracy using GEPP is
$m+O(\log n)$, which improves an earlier estimate $m+O(\log^2 n)$ of Sankar, and which we conjecture to be optimal by the order of magnitude. We further provide
tail estimates of the growth factor which can be used to support
the empirical observation that GEPP is more stable than the Gaussian Elimination with no pivoting.
\end{abstract}

\tableofcontents

\section{Introduction}
The Gaussian Elimination
is a classical algorithm
for solving systems of linear equations \cite[Chapter~3]{Golub}, \cite[Chapter~9]{Higham}.
The simplest form of the algorithm --- the Gaussian Elimination with no pivoting --- solves a linear system (SLE) $Ax=b$ with a square coefficient matrix $A$ 
by performing the {\it $LU$--factorization}: $A$
is represented as the product $LU$ where $L$ and $U$ are lower and upper triangular matrix, respectively, and $x$ is obtained by a combination of forward and back substitutions $y:=L^{-1}b$, $x:=U^{-1}y$. A possible algorithmic representation of this well known process is given below in Algorithm~\ref{GENP}.
The procedure produces a sequence of matrices $A^{(0)}:=A,A^{(1)},\dots,A^{(n-1)}=:U$, where for every $k\leq n-1$, the $k\times k$ top left submatrix of $A^{(k-1)}$ is upper triangular.
The elimination process with no pivoting fails if at any step $k=1,\dots,n-1$, the $k$--th diagonal element of $A^{(k-1)}$ is zero.

The computation of $A^{(k)}$ from $A^{(k-1)}$ can be represented in matrix form as
$$
A^{(k)}:= (\Id_n-\tau^{(k)}e_k^\top)A^{(k-1)},
$$
where
$$
\tau^{(k)}:= \sum_{j=k+1}^n \tau^{(k)}_j e_j,
$$
and where $e_1,\dots,e_n$ are standard unit basis vectors in $\R^n$.
We also note that with this notation
$$
L=\Id_n+\sum_{k=1}^{n-1} \tau^{(k)}e_k^\top.
$$
The matrices $\Id_n-\tau^{(k)}e_k^\top$, $k=1,\dots,n-1$, are called {\it the Gauss transformations}.

\begin{algorithm}[ht]
\normalsize
\caption{{\tt $LU$-factorization}}
\label{GENP}
\begin{algorithmic}[1]
\Input An $n\times n$ matrix $A$
\Output $LU$--factorization $A=LU$ if the algorithm succeeds, or \textbf{Error}
if the algorithm fails
\State{$A^{(0)}:= A$}
\For{$k = 1, \dots, n-1$}
\State{\textbf{if } $A_{k,k}^{(k-1)}\neq 0$}
\State{\textbf{then}} 
\State{Initialize $A^{(k)}$ as zero matrix}
\For{$i = k+1, \dots, n$}
\State{$\tau^{(k)}_i:= A_{i,k}^{(k-1)}/A_{k,k}^{(k-1)}$}
\For{$j= k+1,\dots,n$}
\State{$A^{(k)}_{i,j}:=A^{(k-1)}_{i,j}-A^{(k-1)}_{k,j}\,\tau_i^{(k)}$}
\EndFor
\EndFor
\State{\textbf{else}} 
\State{Output \textbf{Error}} 
\State{\textbf{end if}}
\EndFor
\State{$U:= A^{(n-1)}$}
\State{Set $L$ to be the unit lower triangular matrix with
$$L_{i,k}:=\tau^{(k)}_i, \quad i = k+1, \dots, n, \quad k=1,\dots,n-1.$$}
\end{algorithmic}
\end{algorithm}

When considering an implementation using the floating point arithmetic, a well known issue of the Gaussian Elimination is its {\it numerical instability}.
Recall that the {\it condition number} $\kappa(A)$ of a square  matrix $A$ defined as the ratio of the largest and smallest singular value of $A$.
Even for some well-conditioned matrices (i.e having a small condition number), solving SLE with help of the Gaussian Elimination with no pivoting results in large relative errors of the computed solution vectors \cite[Section~3.3]{Golub}.

\medskip

Several modifications of the elimination procedure are commonly used in matrix computations to address the instability issue \cite[Chapter~3]{Golub}, \cite[Chapter~9]{Higham}.
In particular, the {\it Gaussian Elimination with Partial Pivoting (GEPP)}
looks for a representation $PA=LU$ (called {\it $PLU$--factorization}), where, as before, $L$ and $U$ are lower and upper triangular matrices, while $P$ is a specially constructed permutation matrix.
The solution of a corresponding SLE can then be obtained by a combination of forward and back substitutions, and a permutation of vector's components (see Algorithm~\ref{GEPP1}; for better readability, we represent the formula for $L$ in matrix rather than entry-wise form there).
The GEPP succeeds in exact arithmetic whenever $A$ is non-singular (although it may fail in floating point arithmetic). 

\begin{algorithm}[ht]
\normalsize
\caption{{\tt $PLU$-factorization}}
\label{GEPP1}
\begin{algorithmic}[1]
\Input An $n\times n$ matrix $A$
\Output Matrices $P,L,U$ satisfying $PA=LU$, or \textbf{Error}
if the algorithm fails
\State{$A^{(0)}:=A$}
\For{$k = 1, \dots, n-1$}
\State{Set $y_k$ to be an index in $\{k,\dots,n\}$ corresponding
to the largest (by the absolute value) element among $A^{(k-1)}_{k,k},A^{(k-1)}_{k+1,k},\dots,A^{(k-1)}_{n,k}$\quad}
\Comment{{\it \small We can assume any tie-breaking rule here
}}
\State{\textbf{if } $A^{(k-1)}_{y_k,k}=0$}
\State{\textbf{then}} 
\State{Output \textbf{Error}} 
\State{\textbf{else}} 
\State{Set $P^{(k)}$ to be the permutation matrix that
swaps $k$--th and $y_k$--th components:
$$P^{(k)}_{k,y_k}=P^{(k)}_{y_k,k}:=1,\quad P^{(k)}_{j,j}:=1, \;\; j\in[n]\setminus\{y_k,k\}$$}
\State{Initialize $A^{(k)}$ as zero matrix}
\For{$i = k+1, \dots, n$}
\State{$\tau^{(k)}_i:= (P^{(k)}A^{(k-1)})_{i,k}/(P^{(k)}A^{(k-1)})_{k,k}$}
\For{$j= k+1,\dots,n$}
\State{$A^{(k)}_{i,j}:=(P^{(k)}A^{(k-1)})_{i,j}-(P^{(k)}A^{(k-1)})_{k,j}\,\tau_i^{(k)}$}
\EndFor
\EndFor
\Comment{{\it \small In matrix form,
$A^{(k)}:=(\Id_n-\tau^{(k)}e_k^\top)P^{(k)}A^{(k-1)}$, where $\tau^{(k)}:= \sum_{j=k+1}^n \tau^{(k)}_j e_j$}}
\State{\textbf{end if}}
\EndFor
\State{$U:=A^{(n-1)}$}
\State{$P:=\prod_{k=1}^{n-1}P^{(k)}$}
\State{$L:=\prod_{k=1}^{n-1}\big(\Id_n+\big(P^{(n-1)}\dots P^{(k+1)}\tau^{(k)}\big)\,e_k^\top\big)
=\Id_n+\sum_{k=1}^{n-1}\big(P^{(n-1)}\dots P^{(k+1)}\tau^{(k)}\big)\,e_k^\top$}
\end{algorithmic}
\end{algorithm}

A seminal result of Wilkinson \cite{Wilkinson}
gives an upper bound on the backward error
during the Gaussian Elimination when the floating point computations are performed.
Define the {\it unit roundoff}
$$
{\bf u}:=\frac{1}{2}\big(\mbox{the gap between $1$ and the next floating point number}\big),
$$
so that for every real number $x$, its floating point representation ${\rm fl}(x)$
satisfies $|x-{\rm fl}(x)|\leq {\bf u}|x|$ as long as {\it no underflow or overflow exception arise} \cite[Section~2.7]{Golub}. 
Let $A$ be an invertible $n\times n$ matrix, assume that GEPP in floating point arithmetic with no underflow and overflow exceptions is performed on the matrix ${\rm fl}(A)$,
and assume that no error occurs during the computation (i.e no zero pivots are encountered).
Let $\hat P,\hat L,\hat U$ be the computed matrices from the floating point GEPP of ${\rm fl}(A)$, with $\hat U_{n,n}\neq 0$, and let $P,L,U$ be the $PLU$--factorization of $A$ in exact arithmetic. Assume that $\hat P=P$.
Further, let $\hat x$ denote the computed solution corresponding to the exact solution of the SLE $Ax=b$.
Then
$$
PA=\hat L \hat U+H,\quad (A+E)\hat x=b \quad \mbox{(equalities hold in exact arithmetic)},
$$
where
$$
\|H\|= O\big(n^2\,{\bf u}\,(\|A\|+n\|\hat U\|)\big),
$$
and
$$
\|E\|=O\big(n^2\,{\bf u}\,(\|A\|+n\|\hat U\|)\big)
$$
(see, in particular, \cite[Theorem~3.3.1 and Theorem~3.3.2]{Golub}).
Define the {\it growth factor}
${\bf g_{\rm GEPP}}$ as
\begin{align}\label{dfpajnefpiufnwpfijnwqpkwn}
		{\bf g_{\rm GEPP}}(A)
		:=\frac{ \max_{k,i,j \in [n]} |\hat A^{(k-1)}_{i,j}|}{\max_{i,j\in[n]} |A_{i,j}|},
\end{align}
where $\hat A^{(1)},\dots,\hat A^{(n-1)}$ denote the computed (in the floating point arithmetic)
matrices $A^{(1)},\dots$, $A^{(n-1)}$. Then, under the above assumptions, the backward error estimate can be written as
$$
\|E\|=O\big(n^4 \,{\bf u}\,g_{\rm GEPP}(A)\,\max_{i,j\in[n]} |A_{i,j}|\big),
$$
implying, under the additional assumption $s_{\min}(A)\geq 2\|E\|$, the forward error bound for the computed solution
\begin{equation}\label{eq: wilk123}
\frac{\| \hat x - x\|_2}{\|\hat x\|_2} = O\big(n^4 \,{\bf u}\,\kappa(A)\, {\bf g_{\rm GEPP}}(A)\big),
\end{equation}
where $\kappa(A)=\|A\|\,\|A^{-1}\|$ is the condition number of $A$.
Similar error bounds are available for other versions of the Gaussian Elimination
(with no pivoting, with complete or with rook pivoting).
We refer, in particular, to Wilkinson's paper \cite{Wilkinson} and to modern accounts
of the backward error analysis of the different forms of the Gaussian Elimination
in \cite[Chapter~3]{Golub}, \cite[Chapter~9]{Higham}, as well as \cite{HighamPaper}.


\medskip

It can be checked that ${\bf g}_{\rm GEPP}(A)=O(2^n)$ for any $n\times n$ invertible matrix $A$, and that this bound is attained.
Thus, \eqref{eq: wilk123} provides a satisfactory worst-case estimate 
only under the assumption ${\bf u}\ll 2^{-n}$, i.e when the unit roundoff is exponentially small in the matrix dimension. At the same time, the accumulated empirical evidence suggests that for a ``typical'' coefficient matrix the loss of precision is much smaller than the worst-case prediction. Let us quote \cite[p.~131]{Golub}: ``Although there is still more to understand about [the growth factor], the consensus is that serious
element growth in Gaussian Elimination with Partial Pivoting is {\it extremely} rare. {\it The method can be used with confidence.}''

In \cite{Trefethen} Trefethen and Schreiber carried out an empirical study of the Gaussian Elimination with Partial and with Complete Pivoting in the setting when the input coefficient matrix $A$ is {\it random}, having i.i.d standard Gaussian entries. Their experiments showed that with high probability the growth factor in GEPP is only polynomially large in $n$. Further numerical studies by Edelman suggest that ${\bf g}_{\rm GEPP}(A)$ of an $n\times n$ standard Gaussian matrix $A$ is of order $O(n^{1/2+o(1)})$ with probability close to one (see a remark in \cite[p.~182]{EM}).

An important step in improving theoretical understanding of numerical stability of the Gaussian Elimination was made by Yeung and Chan in \cite{Yeung}. Their result implies (although that is not explicitly stated in the paper) that for the Gaussian Elimination with {\it no pivoting} applied to the standard $n\times n$ Gaussian matrix, the relative error of the solution vector can be bounded above by ${\bf u}\,n^{O(1)}$ with probability close to one. A vast generalization of their estimate was obtained by Sankar, Spielman, and Teng in \cite{SST} in the context of the {\it smoothed analysis of algorithms}. Let $M$ be any non-random $n\times n$ matrix, and let $G$ be an $n\times n$ matrix with i.i.d $N(0,\sigma^2)$ Gaussian entries. The main result of \cite{SST} asserts that the expected number of bits of precision sufficient to solve $(M+G)x = b$ to $m$ bits of accuracy using Gaussian elimination without pivoting is at most
$m+O\big(\log\big(n+\frac{\|M\|}{\sigma}\big)\big)$. This provides a theoretical justification for the observed performance of the GE with no pivoting for {\it structured} dense coefficient matrices.

The no-pivoting strategy is crucial for the proofs in \cite{Yeung} or \cite{SST}. With partial pivoting, the permutations of the rows after each elimination step introduce complex dependencies to the model which require other arguments to handle. In the PhD thesis \cite{Sankar}, Sankar carried out smoothed analysis of GEPP based on certain recursive matrix formula (to be discussed in some detail in the next section).
Let $A=M+G$, where $G$ is the Gaussian random matrix with i.i.d $N(0,\sigma^2)$
entries, and $M$ is a deterministic matrix of spectral norm at most one.
One of main results of \cite{Sankar} states that, with the above notation,
$$
\Prob\bigg\{
\frac{ \max_{k,i,j \in [n]} |A^{(k-1)}_{i,j}|}{\max_{i,j\in[n]} |A_{i,j}|}\geq t\bigg\}
\leq \frac{\big(O\big(n\sigma^{-1}+n^{3/2}\big)\big)^{12\log n}}{t^{(\log n)/21}},\quad t>0,
$$
so that in the mean zero setting $M=0$, with high probability
$\frac{\max_{k,i,j \in [n]} |A^{(k-1)}_{i,j}|}{\max_{i,j\in[n]} |A_{i,j}|}=
n^{O(\log n)}$.
Note that the quantity considered in \cite{Sankar}
is not a growth factor as was defined above but its ``exact arithmetic'' counterpart.
The relation between matrices $A^{(k-1)}$ and the corresponding computed matrices $\hat
A^{(k-1)}$ is not trivial and will be discussed later; at this point we note that
assuming that the magnitudes of the ratio $\frac{\max_{k,i,j \in [n]} |A^{(k-1)}_{i,j}|}{\max_{i,j\in[n]} |A_{i,j}|}$ and the growth factor ${\bf g}_{\rm GEPP}(A)$ match and in view of \eqref{eq: wilk123}, the result of Sankar implies that with high probability GEPP results in at most $O(\log^2 n)$ lost bits of precision in the obtained solution vector.
This bound is worse than the $O(\log n)$ estimate for GE with no pivoting implied by \cite{Yeung}.

\medskip

To summarize, whereas strong results on average-case stability of GE with {\it no pivoting} has been obtained in the literature, the Gaussian Elimination with Partial Pivoting lacked {\it matching} theoretical guarantees, let alone justifying the common belief that GEPP tends to be {\it more stable} than GE with no pivoting.
In this work, we make progress on this problem. To avoid any ambiguity, we
recall all the imposed assumptions and notation:
\begin{theoremA} \label{thm: main}
There are universal constants $C,\tilde C>1$ and a function
$\tilde n:[1,\infty)\to\N$ with the following property. 
Let $p\geq 1$, and let $n\geq \tilde n(p)$.
\begin{itemize}
\item Assume that the floating point computations with no underflow and overflow exceptions
and with a unit roundoff ${\bf u}$ are being performed.
\item Let $A$ be the random $n\times n$ matrix with i.i.d standard Gaussian
entries, (the real Ginibre Ensemble). Assume that the floating point GEPP is performed on the matrix ${\rm fl}(A)$.
\end{itemize}
Then with probability at least $1-{\bf u}^{1/8}\,n^{\tilde C}$, the GEPP for ${\rm fl}(A)$
succeeds in floating point arithmetic and the computed permutation matrix $\hat P$
agrees with the matrix $P$ from the $PLU$--factorization of $A$ in exact arithmetic. Furthermore, assuming ${\bf u}^{1/8}\,n^{\tilde C}\leq 1/2$,
$$
\Prob\big\{{\bf g}_{\rm GEPP}(A)
\geq n^{t}\;\big|\;\mbox{GEPP succeeds in f.p.\ arithmetic and }\hat P=P \big\} \leq
40^p\,n^{-p t},\quad t\geq Cp^2.
$$
\end{theoremA}  
We do not attempt to compute the constant $C$ in the above theorem
explicitly and leave the problem of finding an optimal (up to $n^{o_n(1)}$
multiple) estimate of the growth factor ${\bf g}_{\rm GEPP}(A)$ for future research
(see Section~\ref{fepfojnfpwieufnqwpifjqn}). Further, we expect
a much stronger bound on the probability that GEPP succeeds in the floating point arithmetic
and that $\hat P=P$.

In view of the aforementioned work of Wilkinson and well known estimates for the condition number of the Gaussian matrix \cite{Edelman,Szarek}, the theorem implies that
with probability close to one the number of bits of precision sufficient to solve $Ax = b$ to $m$ bits of accuracy using GEPP is
$m+O(\log n)$. We conjecture that this bound is optimal in the sense that in the same setting $m+\Omega(\log n)$ bits of precision are {\it necessary} with probability close to one.

Let us further apply Theorem~\ref{thm: main} to compare numerical stability of GEPP with that of GE with no pivoting. As we mentioned at the beginning of the introduction, the Gaussian Elimination with no pivoting can produce arbitrarily large relative error in the floating point arithmetic even for well-conditioned coefficient matrices. As an illustration, consider a $2\times 2$ standard Gaussian matrix in floating point arithmetic,
$$
M=\begin{pmatrix}
{\rm fl}(g_{11}) & {\rm fl}(g_{12})\\
{\rm fl}(g_{21}) & {\rm fl}(g_{22})
\end{pmatrix}.
$$
The Gaussian Elimination with no pivoting yields the computed $LU$--factorization of $M$,
$$
\hat L=\begin{pmatrix}
1 & 0\\
{\rm fl}\big({\rm fl}(g_{21})/{\rm fl}(g_{11})\big) & 1
\end{pmatrix};\quad 
\hat U=\begin{pmatrix}
{\rm fl}(g_{11}) & {\rm fl}(g_{12})\\
0 & {\rm fl}\big({\rm fl}(g_{22})-{\rm fl}(g_{12})\cdot {\rm fl}(g_{21})/{\rm fl}(g_{11})\big)
\end{pmatrix}.
$$
It can be checked that for every $\varepsilon\in({\bf u},1)$, with probability
$\Theta(\varepsilon)$ all of the following holds:
\begin{itemize}
    \item The matrix $M$ is well-conditioned, say, $\kappa(M)\leq 100$;
    \item $|{\rm fl}(g_{11})|\leq \varepsilon$, $|{\rm fl}(g_{12})|,|{\rm fl}(g_{21})|,|{\rm fl}(g_{22})|\in[1/2,2]$;
    \item
    $\big|{\rm fl}\big({\rm fl}(g_{22})-{\rm fl}(g_{12})\cdot {\rm fl}(g_{21})/{\rm fl}(g_{11})\big)
    -\big({\rm fl}(g_{22})-{\rm fl}(g_{12})\cdot {\rm fl}(g_{21})/{\rm fl}(g_{11})\big)
    \big|\geq \Omega({\bf u})\,\big({\rm fl}(g_{22})-{\rm fl}(g_{12})\cdot {\rm fl}(g_{21})/{\rm fl}(g_{11})\big)$.
\end{itemize}
With the above conditions,
the bottom right element of the product $\hat L\hat U$
differs from ${\rm fl}(g_{22})$ by a quantity of order
$\Omega({\bf u})\,\big({\rm fl}(g_{22})-{\rm fl}(g_{12})\cdot {\rm fl}(g_{21})/{\rm fl}(g_{11})\big)$, that is, the normwise backward error satisfies
$$
\Prob\big\{\|\hat L\hat U-M\|> c{\bf u}\,\|M\|/\varepsilon\;\big|\;
\kappa(M)\leq 100\big\}\geq c\varepsilon,\quad \varepsilon\in({\bf u},1),
$$
for some universal constant $c>0$ (one may safely take $c=1/100$, say).


In sharp contrast with the above observation, in the case of GEPP the probability
of large deviations for the backward
error is much smaller as Theorem~\ref{thm: main}
shows.
Indeed, with the notation from the theorem and in view of Wilkinson's bound, {\it arbitrary} $p\geq 1$ and assuming $n$ is sufficiently large, we have
$$
\Prob\big\{
\|\hat L\hat U-\hat P A\|> C'{\bf u}\,n^{4}\|A\|/\varepsilon 
\big\}\leq 40^p\varepsilon^p+{\bf u}\,n^C,\; \varepsilon\in(0, n^{-Cp^2}],
$$
for a universal constant $C'>0$. Thus, the tail of the
distribution of the backward error of GEPP decays
superpolynomially.
Informally, {\it the ``proportion'' of well-conditioned coefficient matrices yielding large backward errors
is much smaller for GEPP than for the Gaussian Elimination with no pivoting.}

\bigskip

We provide a detailed outline of the argument, as well as a comparison of our techniques with the earlier approach of Sankar, in the next section.

\bigskip

{\it{}The following notation will be used throughout the paper:}

\medskip

For positive integers $m\le n$, 

\begin{tabular}{cp{0.8\textwidth}}
	$[n]$ & is the set $\{1,2,3,\dots, n\}$	\\
	$[m,n]$ & is the set $\{m,m+1,\dots, n\}$
\end{tabular}

\medskip

For a $n\times m$ matrix $M$, indices $i\in[n]$, $j\in [m]$, and non-empty
subsets $I \subset [n]$ and $J \subset [m]$, 

\begin{tabular}{cp{0.8\textwidth}}
$M_{I,J}$ & is the submatrix of $M$ formed by taking rows indexed over
$I$ and columns indexed over $J$. When $I =\{i\}$ or $J=\{j\}$, we will use lighter notations $M_{i,J}$ and $M_{I,j}$ in place of $M_{\{i\},J}$ and $M_{I,\{j\}}$ \\
$M_{i,j}$ & is the $(i,j)$--th entry of $M$ \\ 
$s_j(M)$ & is the $j$--th largest singular value of $M$ \\
$\R^I$ & The $|I|$-dimensional Euclidean space with components indexed over $I$ \\
$\dist(\cdot,\cdot)$ & The Euclidean distance
\end{tabular}

\section{Outline of the proof}\label{peofjnwqepifunfpijn}

Let $A$ be an $n\times n$ standard Gaussian matrix, let $A^{(0)}:=A,A^{(1)},\dots,A^{(n-1)}$
be the sequence of matrices generated by GEPP process,
and let $P^{(1)},\dots,P^{(n-1)}$ be the corresponding permutation matrices
(see Algorithm~\ref{GEPP1}).
It turns out that in our probabilistic model, estimating the growth factor ${\bf g}_{\rm GEPP}(A)$ can be reduced to bounding the exact arithmetic counterpart of the quantity,
$$
\frac{\max_{k,i,j \in [n]} |A^{(k-1)}_{i,j}|}{\max_{i,j\in[n]} |A_{i,j}|}.
$$
Our main focus is to derive Proposition~\ref{npafjnepfijnfpijqnpijn}, which is
the exact arithmetic counterpart of the main theorem, and
then reduce
the setting of floating-point arithmetic to exact arithmetic.
We provide a rigorous account of the reduction procedure in Section~\ref{faonfpafjnpifjnpijn}, and prefer to
avoid discussing this technical matter here. We only note that
comparison of the matrices $A^{(k-1)}$ and
$\hat A^{(k-1)}$, $1\leq k\leq n-1$, is based on a well established inductive argument somewhat similar to the one used to prove Wilkinson's backward error bound.
From now on and till 
Section~\ref{faonfpafjnpifjnpijn} we work in exact arithmetic unless explicitly stated otherwise.

\medskip

Define ``unpermuted'' matrices $\prA^{(k)}$ obtained at the $k$--th elimination step,
i.e $\prA^{(0)}:=A$ and
\begin{equation}\label{nalfkjnfo34iuhoqfjj}
\prA^{(k)} = \big(P^{(1)}\big)^{-1} \big(P^{(2)}\big)^{-1}
\cdots \big(P^{(k)}\big)^{-1}\, A^{(k)},\quad 1\leq k\leq n-1.
\end{equation}
Let $I_0:=\emptyset$, and
for each $1\leq k\leq n-1$ let $I_k=I_k(A)$ be the (random) subset of $[n]$ of row indices of $A$ corresponding to the pivot elements used in the first $k$ steps of the ``permutation-free'' elimination process.
Notice that within the $k$--th column of ${\prA}^{(k)}$, the components except those
in $I_{k-1}$ and the $k$--th pivot element, are all zeros. Therefore, the set $I_k$ can be defined as 
$$
I_k:=I_{k-1}\cup\big\{i\in[n]\setminus I_{k-1}:\;\prA^{(k)}_{ik}\neq 0\big\},\;\;1\leq k\leq n-1,
$$
where $\big\{i\in[n]\setminus I_{k-1}:\;\prA^{(k)}_{ik}\neq 0\big\}$ is a singleton.
We will further denote by $i_k=i_k(A)$, $1\leq k\leq n-1$, the elements in the singletons $I_k\setminus I_{k-1}$, so that $I_k = \{i_1,i_2,\dots, i_k \}$, $1\leq k\leq n-1$.

For $1\leq k\leq n-1$ and $t \in [k]$, the first $t-1$ components of $\row_{i_t}(\prA^{(k)})$ are zeros, and $\prA^{(k)}_{ [n]\backslash I_k, [k] }$ is the zero matrix; more specifically, for each $1 \le k \le n-1$, and $j \in [n]\backslash I_k$,
\begin{align}\nonumber
    \row_j\big( \prA^{(k)}_{[n],[k]} \big) &= 0 \quad\mbox{  and  } \\
    \row_j\big(\prA^{(k)}_{[n],[k+1,n]}\big)
&=
    \row_j \big( A_{[n],[k+1,n]} \big)  - \row_j(A_{[n],[k]}) \big(A_{I_k,[k]} \big)^{-1} A_{I_k, [k+1,n]}\label{kejfnpfiunf9un}
\end{align}
(see, in particular, \cite[Formula~4.1]{SST} for GE with no pivoting, which can be adapted to our setting).
Thus, for $0\leq k<n-1$, the index $i_{k+1}$ is defined as the one corresponding to the largest number among
$$
A_{j,k+1} - A_{j,[k]} \big(A_{I_k,[k]}\big)^{-1} A_{I_k, k+1}, \, \, j \in [n]\backslash I_k.
$$

\medskip

Due to strong concentration of Gaussian variables, the operator norms of matrices $A_{I,J}$, $I,J\subset[n]$, can be {\it uniformly} bounded from above by a polynomial in $n$.
Thus, the principal difficulty in obtaining satisfactory upper bounds on the growth factor ${\bf g}_{\rm GEPP}(A)$ is in estimating the norm of vectors
$\row_j(A_{[n],[k]}) \big(A_{I_k,[k]} \big)^{-1}$, $j \in [n]\backslash I_k$.
The sets $I_k$ are random and depend on $A$ in a rather complicated way. At the same time, the trivial upper bound
$$
\max\limits_{j\in [n]\backslash I_k}
\big\|\row_j(A_{[n],[k]}) \big(A_{I_k,[k]} \big)^{-1}\big\|_2
\leq \max\limits_{J\subset[n],\,|J|=k;\,j\in [n]\setminus J}
\big\|\row_j(A_{[n],[k]}) \big(A_{J,[k]} \big)^{-1}\big\|_2
$$
which completely eliminates the randomness of $I_k$ from consideration,
is vastly suboptimal\footnote{It can be checked, for example, that with a constant probability
$$
\max\limits_{J\subset[n],\,|J|=n/2;\,j\in [n]\setminus J}
\big\|\row_j(A_{[n],[n/2]}) \big(A_{J,[n/2]} \big)^{-1}\big\|_2
$$
is greater than any predefined constant power of $n$. We expect that a much stronger lower bound can be established.}.

The first part of this section is devoted to the argument of Sankar from \cite{Sankar}
which yields 
a bound ${\bf g}_{\rm GEPP}(A)=O(n^{C\log n})$ with high probability using certain recursive matrix formula. In the second part, we discuss our approach.

\subsection{Sankar's argument}

Consider a block matrix 
$$
\begin{bmatrix} B \\ X \end{bmatrix}= \begin{bmatrix} B_{\rm u} \\ B_{\rm \ell} \\ X \end{bmatrix} =
\begin{bmatrix} B_{ \rm u\ell} & B_{ \rm ur} \\ B_{\rm \ell\ell} & B_{ \rm \ell r} \\ X_{\rm \ell} & X_{\rm r} \end{bmatrix},
$$
where \(B_{\rm u\ell}\) and \(B_{\rm \ell r}\) are square non-singular matrices and \(X\) is a row vector. Then, denoting \(B' := B_{\rm \ell r} - B_{\rm \ell\ell}B_{\rm u\ell}^{-1} B_{\rm ur}\) and \(X' := X_{\rm r} - X_{\rm \ell} B_{\rm u\ell}^{-1} B_{\rm ur}\),
\begin{align}\label{p4pi3hjr4p9ruh34pr9u834h}
 \begin{bmatrix}  -XB^{-1} & 1 \end{bmatrix}  =& \begin{bmatrix}  -X'(B')^{-1} & 1 \end{bmatrix} \cdot \begin{bmatrix} -\begin{bmatrix} B_{\rm \ell} \\ X \end{bmatrix} \cdot B_{\rm u}^\dagger & \Id \end{bmatrix},
\end{align}
where $B_{\rm u}^\dagger$ is the right pseudoinverse of $B_{\rm u}$
(see \cite[Chapter~3]{Sankar}).

The above formula is applied in \cite{Sankar} in a recursive manner. Assume for simplicity of exposition that we are interested in bounding the Euclidean norm of the vector
$A _{j,[n/2]} (A_{I_{n/2},[n/2]})^{-1}$ for some $j\in [n]\setminus I_{n/2}$
(recall that, in view of \eqref{kejfnpfiunf9un} and standard concentration estimates for the spectral norm of Gaussian matrices, this would immediately imply an estimate
on the components of $\prA^{(n/2)}_{j,[n/2+1,n]}$).
Fix for a moment any \(0\le v < m < n/2\), and let
\(B := \prA^{(v)}_{ I_{n/2}\setminus I_v, [v+1,n/2]}\) and \(X:= \prA^{(v)}_{j, [v+1,n/2]}\). We write
\begin{align*}
\begin{bmatrix} B \\ X \end{bmatrix}= \begin{bmatrix} B_{\rm u} \\ B_{\rm \ell} \\ X \end{bmatrix} =
\begin{bmatrix} B_{ \rm u\ell} & B_{ \rm ur} \\ B_{\rm \ell\ell} & B_{ \rm \ell r} \\ X_{\rm \ell} & X_{\rm r} \end{bmatrix}  &= \begin{bmatrix} \prA^{(v)}_{I_{m}\setminus I_v,[v+1,m]} & \prA^{(v)}_{ I_{m}\setminus I_v,[m+1,n/2]} \\ \prA^{(v)}_{I_{n/2}\setminus I_m, [v+1,m]} & \prA^{(v)}_{I_{n/2}\setminus I_m,[m+1,n/2]} \\ \prA^{(v)}_{j,[v+1,m] } & \prA^{(v)}_{j, [m+1,n/2]}
\end{bmatrix}.
\end{align*}
It can be checked that with the above notation,
\(B'= \prA^{(m)}_{I_{n/2}\setminus I_m,[m+1,n/2]}\) and \(X' = \prA^{(m)}_{j,[m+1,n/2]}\) \cite{Sankar}. Relation \eqref{p4pi3hjr4p9ruh34pr9u834h} then implies
\begin{align*}
 &\left\| \begin{bmatrix} \prA^{(v)}
 _{j,[v+1,n/2]} (\prA^{(v)}_{I_{n/2}\setminus I_v,[v+1,n/2]})^{-1} & 1 \end{bmatrix}\right\|
\leq
\left\| \begin{bmatrix} \prA^{(m)}_{j,[m+1,n/2]} (\prA^{(m)}_{I_{n/2}\setminus I_m,[m+1,n/2]})^{-1} & 1 \end{bmatrix}\right\|\cdot\\
&\hspace{2cm}\cdot  \left\|
\begin{bmatrix}\prA^{(v)}_{ (I_{n/2}\setminus I_m)\cup\{j\}, [v+1,n/2]} \big( \prA^{(v)}_{I_{m}\setminus I_v, [v+1,n/2]} \big)^\dagger & \Id_{n/2-m+1} \end{bmatrix}\right\|.
\end{align*}
Now, assume that we have constructed a sequence of indices
\(0=k_0 <k_1<k_2<\cdots k_s <n/2\), with $k_s=n/2-1$ and $k_1\geq n/4$.
Applying the last relation recursively $s$ times, we obtain
\begin{align}\nonumber
&\left\| \begin{bmatrix} A_{j,[n/2]} (A_{I_{n/2},[n/2]})^{-1} & 1 \end{bmatrix}\right\|
 \leq 
 \left\| \begin{bmatrix} \prA^{(n/2-1)}_{j,n/2} (\prA^{(n/2-1)}_{i_{n/2},n/2})^{-1} & 1 \end{bmatrix}\right\|\cdot\\
&\hspace{2cm}\cdot\prod_{\ell=0}^{s-1} \left\|
\begin{bmatrix}\prA^{(k_\ell)}_{ (I_{n/2}\setminus I_{k_{\ell+1}})\cup\{j\}, [k_\ell+1,n/2]} \big( \prA^{(k_\ell)}_{I_{k_{\ell+1}}\setminus I_{k_\ell}, [k_\ell+1,n/2]} \big)^\dagger & \Id_{n/2-k_{\ell+1}+1} \end{bmatrix}\right\|,\label{qpinfopfiu3nf08ufn08un}
\end{align}
where, by the definition of the partial pivoting, $\big|\prA^{(n/2-1)}_{j,n/2} (\prA^{(n/2-1)}_{i_{n/2},n/2})^{-1}\big|\leq 1$.
Therefore, the problem reduces to estimating the spectral norms of matrices
\begin{equation}\label{fpqiufn40r8083unf083ufn}
\prA^{(k_\ell)}_{ (I_{n/2}\setminus I_{k_{\ell+1}})\cup\{j\}, [k_\ell+1,n/2]} \big( \prA^{(k_\ell)}_{I_{k_{\ell+1}}\setminus I_{k_\ell}, [k_\ell+1,n/2]} \big)^\dagger,\;\;0\leq \ell<s.
\end{equation}

Sankar shows that as long as $n/2-k_\ell$ ($\ell=s,s-1,\dots,0$) grow as a geometric sequence (in which case $s$ should be of order logarithmic in $n$), the norm of each matrix can be bounded by a constant power of $n$ with a large probability. We only sketch this part of the argument.
Fix any $0\leq \ell<s$,
and define \(Z:= (A_{I_{k_\ell},[k_\ell+1,n/2]})^\dagger A_{I_{k_\ell},[k_\ell]}\), so that \(ZZ^\dagger = \Id\), and
\begin{align*}
&\prA^{(k_\ell)}_{ (I_{n/2}\setminus I_{k_{\ell+1}})\cup\{j\}, [k_\ell+1,n/2]} \big( \prA^{(k_\ell)}_{I_{k_{\ell+1}}\setminus I_{k_\ell}, [k_\ell+1,n/2]} \big)^\dagger\\
&\hspace{2cm}= \prA^{(k_\ell)}_{ (I_{n/2}\setminus I_{k_{\ell+1}})\cup\{j\}, [k_\ell+1,n/2]}\, Z\big( \prA^{(k_\ell)}_{I_{k_{\ell+1}}\setminus I_{k_\ell}, [k_\ell+1,n/2]} Z\big)^\dagger,
\end{align*}
where, in view of \eqref{kejfnpfiunf9un},
\begin{align*}
&\prA^{(k_\ell)}_{ (I_{n/2}\setminus I_{k_{\ell+1}})\cup\{j\}, [k_\ell+1,n/2]}\, Z\\
&\hspace{2cm}=  A_{(I_{n/2}\setminus I_{k_{\ell+1}})\cup\{j\}, [k_\ell+1,n/2]}\, (A_{I_{k_\ell},[k_\ell+1,n/2]})^\dagger\, A_{I_{k_\ell},[k_\ell]} - A_{(I_{n/2}\setminus I_{k_{\ell+1}})\cup\{j\}, [k_\ell]}\, Z^\dagger Z.
\end{align*}
Since \(Z^{\dagger}Z\) is a projection and has unit norm,
an upper bound on $\big\|(A_{I_{k_\ell},[k_\ell+1,n/2]})^\dagger\big\|$
would provide a bound on $\big\|
\prA^{(k_\ell)}_{ (I_{n/2}\setminus I_{k_{\ell+1}})\cup\{j\}, [k_\ell+1,n/2]} Z
\big\|$. The key observation here is that $A_{I_{k_\ell},[k_\ell+1,n/2]}$ is equidistributed with the standard $k_\ell\times (n/2-k_\ell)$ Gaussian matrix, so that a satisfactory estimate on the norm of the pseudoinverse follows.

Bounding the operator norm of $\big( \prA^{(k_\ell)}_{I_{k_{\ell+1}}\setminus I_{k_\ell}, [k_\ell+1,n/2]} Z\big)^\dagger$ is more involved. Note that, equivalently, it is sufficient to provide a good lower bound on the smallest singular value of the matrix
$$
\big(\prA^{(k_\ell)}_{I_{k_{\ell+1}}\setminus I_{k_\ell}, [k_\ell+1,n/2]} Z\big)^\top.
$$
We have
\begin{align}\label{falkfnofiuyfb038fu3p8fun}
s_{\min}\big(\big(\prA^{(k_\ell)}_{I_{k_{\ell+1}}\setminus I_{k_\ell}, [k_\ell+1,n/2]} Z\big)^\top\big) \ge 
\min_{J \subset [n] \backslash I_{k_\ell}, |J|=k_{\ell+1}-k_{\ell}} s_{\min} \big((\prA_{J,[k_\ell+1,n/2]}^{(k_\ell)}Z)^\top\big),
\end{align}
where, again in view of \eqref{kejfnpfiunf9un}, for each admissible \(J\), 
\begin{align}\label{piaungpoigunpiu3np}
 (\prA_{J,[k_\ell+1,n/2]}^{(k_\ell)}Z)^\top
 =  Z^\top (A_{J,[k_\ell+1,n/2]})^\top 
 - Z^{\dagger}Z (A_{J,[k_\ell]})^\top,
\end{align}
and where $Z^{\dagger}Z$ is a $k_\ell\times k_\ell$ orthogonal projection matrix of rank $n/2-k_\ell$.
Although $(A_{J,[k_\ell]})^\top$ is dependent on $Z$, it can be shown that $Z^{\dagger}Z (A_{J,[k_\ell]})^\top$ behaves ``almost'' like $Z^{\dagger}Z$ applied to an independent tall rectangular $k_\ell\times (k_{\ell+1}-k_\ell)$
Gaussian matrix (see \cite[Section~3.7]{Sankar}).
This allows to obtain probabilistic estimates on the smallest singular value of the matrix in \eqref{piaungpoigunpiu3np} which, under the assumption that the sequence $n/2-k_\ell$ ($\ell=s,s-1,\dots,0$) does not grow too fast, turn out to be strong enough to survive the union bound in \eqref{falkfnofiuyfb038fu3p8fun}.

\medskip

To summarize, the above argument gives a polynomial in $n$ estimate for 
matrices in \eqref{fpqiufn40r8083unf083ufn}, where $s$ is logarithmic in $n$.
Thus, \eqref{qpinfopfiu3nf08ufn08un} implies a bound $\|\prA^{(n/2)}_{j,[n/2+1,n]}\|_2=n^{O(\log n)}$, $j\in [n]\setminus I_{n/2}$,
with high probability. An extension of this argument to all $\prA^{(k)}$, $1\leq k\leq n-1$, yields ${\rm g}_{\rm GEPP}(A)=n^{O(\log n)}$.
As Sankar notes in \cite{Sankar}, a different choice of $s$ and of the sequence $k_0,k_1,\dots,k_s$, and a refined analysis for the operator norms of matrices \eqref{fpqiufn40r8083unf083ufn} may improve the upper estimate on the growth factor, but cannot achieve a polynomial bound.

\subsection{High-level structure of the proof of the main theorem}

Returning to relation \eqref{kejfnpfiunf9un}, a polynomial
bound on the growth factor ${\rm g}_{\rm GEPP}(A)$ will follow
as long as the norm $(A_{I_r,[r]})^{-1}$ is bounded by $n^{O(1)}$ for every $1\leq r\leq n-1$
with high probability.
We obtain this estimate via analysis of the {\it entire singular spectrum}
of $A_{I_r,[r]}$ rather than attempting to directly bound the smallest singular value of the matrix.

The strategy of the proof can be itemized as follows:
\begin{itemize}

\item Obtaining estimates on the singular values of {\it partially random} block matrices. 
More specifically, we consider matrices of the form 
\begin{equation}\label{eq: fepiufno8f03ybof83u}
B=
\begin{bmatrix}
  F & M \\
  W & Q
\end{bmatrix},
\end{equation}
where $F$ is a fixed square matrix with prescribed singular spectrum, and $M,W,Q$ are independent Gaussian random matrices of compatible dimensions. Our goal here is to derive lower bounds on the intermediate singular values of $B$
in terms of singular values of $F$.

\item Applying the estimates on the intermediate singular values of partially random block matrices in a recursive manner together with a union bound argument, derive lower bounds on the ``smallish'' singular values of matrices $A_{I_r,[r]}$. Our argument at this step only allows to bound first to $(r-C)$--th singular value of the matrix for some large constant $C$.

\item Use the bound on $s_{r-C}(A_{I_r,[r]})$ together with the information on the Euclidean distances
from $\row_{i_\ell}(A_{I_r,[r]})$ to $\spn\{\row_{i_j}(A_{I_r,[r]}),\;1\leq j<\ell\}$, $\ell=1,\dots,r$ that can be extracted from the partial pivoting rule, to obtain polynomial in $n$ lower bounds on $s_{\min}(A_{I_r,[r]})$.
\end{itemize}
Below, we discuss each component in more detail.

\medskip

{\bf Singular spectrum of partially random block matrices.}
The partially random block matrices are treated in Section~\ref{s: poienfpq49f3p9-} of the paper.
Consider a block matrix $B$ of type \eqref{eq: fepiufno8f03ybof83u}, where
$F$ is a fixed $r\times r$ matrix,
$M$ is $r\times x$, $W$ is $x\times r$, $Q$ is $x\times x$ (with $x\leq r$), and the entries of $M$, $W$, $Q$ are mutually independent standard Gaussian variables.
In view of rotational invariance of the Gaussian distribution, we can ``replace''
$F$ with a diagonal matrix $D$
with the same singular spectrum, and with its singular values on the main diagonal arranged in a non-decreasing order. We fix a small positive parameter $\tilde\varepsilon>0$ and an integer $i\geq 1$ such that $\tilde\varepsilon (1+\tilde\varepsilon)^{-i}r\geq 2$.
Our goal at this point is to estimate from below the singular value
$$
s_{\lfloor(1-(1+\tilde\varepsilon)^{-i-1})(r+x)\rfloor}(B).
$$
Having chosen a certain small threshold $\tau>0$ (which is defined as a function of $i,r,\tilde\varepsilon$, the singular spectrum of $D$, and some other parameters which we are not discussing here),
our estimation strategy splits into two cases depending on whether the number $\ell_{i+1}'$ of the singular values of $D$ less than $\tau$ is ``small'' or ``large''.
In the former case, the matrix $D$ has a well controlled singular spectrum, and our goal is to show that attaching to it $x$ rows and columns of standard Gaussians cannot deteriorate the singular values estimates. In the latter case, we show that by adding the Gaussian rows and columns we actually improve the control of the singular values, using that
the top left $\ell_{i+1}'\times \ell_{i+1}'$ corner of $B$ is essentially a zero matrix.
The main result of Section~\ref{s: poienfpq49f3p9-} --- Proposition~\ref{prop: sec4Main}
--- provides a probability estimate on the event that the ratio
$$
\frac{s_{\lfloor(1-(1+\tilde\varepsilon)^{-i-1})(r+x)\rfloor}(B)}{\tau}
$$
is small assuming certain additional relations between the parameters $x,r,\tilde\varepsilon$.

\medskip

{\bf A recursive argument to bound $s_{r-C}(A_{I_r,[r]})$.}
The treatment of the partially non-random block matrices
allows us to solve the principal problem with estimating the singular spectrum of $A_{I_r,[r]}$, namely, the complicated dependencies between $A$ and the index set $I_r$. As we mentioned before, simply bounding the $k$--th smallest singular value of $A_{I_r,[r]}$ by $\min\limits_{I\subset[n],\,|I|=r}s_{r-k}(A_{I_r,[r]})$ produces an unsatisfactory estimate for small $k$. On the other hand, in view of strong concentration of intermediate singular values, already for $k\gg\sqrt{n}\,{\rm polylog}(n)$ (see Proposition~\ref{p: small sing shifted}) this straightforward union bound argument does work.
In order to boost the union bound argument to smaller $k$, we avoid taking the union bound over all $I\subset[n],\,|I|=r$ and instead condition on a realization of $I_{r'}$ for certain $r'<r$, so that the union bound over all $I_{r'}\subset I\subset[n]$ of cardinality $r$ runs over only ${n-r'\choose r-r'}$ admissible subsets rather than ${n\choose r}$ subsets. The two main issues with this approach are
\begin{itemize}
\item first, we must have estimates for the singular spectrum of $A_{I_{r'},[r']}$ in order to apply the results of Section~\ref{s: poienfpq49f3p9-} to
obtain bounds for the singular values of $A_{I_r,[r]}$, and,
\item second, conditioning on a realization of $I_{r'}$ inevitably destroys Gaussianity and mutual independence of the entries of $A_{[n]\setminus I_{r'},[r']}$.
\end{itemize}

The first issue is resolved through the inductive argument, when estimates
on the spectrum of
$A_{I_{r'},[r']}$
obtained at the last induction step are used to control the singular spectrum of 
$A_{I_r,[r]}$ at the next step.
Of course, in this argument we must make sure that the total error accumulated throughout the induction process stays bounded by a constant power of $n$.

The second issue with probabilistic dependencies is resolved by observing that the partial pivoting ``cuts'' a not too large set of admissible values for the elements in $A_{[n]\setminus I_{r'},[r']}$ i.e we {\it can} continue treating them as independent Gaussians up to a manageable loss in the resulting probability estimate after conditioning on a certain event of not--too--small probability.
This problem is formally treated by studying the random polytopes $K_{r'}(A)\subset\R^n$
defined in Section~\ref{fepijnfeofiniqwjnpij} as
$$
K_{r'}(A):= \big\{ x \in \R^n \, :\, \forall s \in [r'],\, | 
		\langle v_s(A) ,\, x \rangle |  \le | \langle v_s(A),\, (A_{i_s,[n]})^\top \rangle|
		\big\},
$$
where
$$
v_s(A) := \big( ((A_{I_{s-1},[s-1]})^{-1} A_{I_{s-1},s})^\top ,\, 1 ,\, 
		\underbrace{0, \dots, 0}_{ n-s \mbox{ components } }
	\big)^\top,\quad s=1,2,\dots,r'.
$$
By the nature of the partial pivoting process,
any row of the submatrix $A_{[n]\setminus I_r',[n]}$ necessarily lies within the polytope $K_{r'}(A)$, and its distribution is a restriction of the standard Gaussian measure in $\R^n$ to $K_{r'}(A)$ (see Section~\ref{fepijnfeofiniqwjnpij}
for a rigorous description). After showing that the Gaussian measure of $K_{r'}(A)$
is typically ``not very small'', we can work with the rows
of $A_{[n]\setminus I_r',[n]}$ as if they were standard Gaussian vectors,
up to conditioning on an event of a not very small probability.
We remark here that Sankar's work \cite{Sankar}
uses random polytopes related to our construction.

\medskip

{\bf Estimating the smallest singular value of $A_{I_r,[r]}$.}
To simplify the discussion, we will only describe the idea of showing that
with a ``sufficiently high'' probability,
$\big(s_{\min}(A_{I_r,[r]})\big)^{-1}=n^{O(1)}$,
without considering computation of the moments of $\big(s_{\min}(A_{I_r,[r]})\big)^{-1}$.
As a corollary of the lower bound on $s_{r'-C}(A_{I_{r'},[r']})$ obtained
via the recursive argument,
we get that with high probability, the inverse of the smallest singular value
of the rectangular matrix $A_{I_{r'},[r'+2\tilde C]}$ satisfies
$\big(s_{\min}(A_{I_{r'},[r'+2\tilde C]})\big)^{-1}=n^{O(1)}$, $r'\in[\tilde C+1,
n-2\tilde C]$, for some integer constant $\tilde C>0$
(see Corollary~\ref{cor: sminTall}).
This corollary is a quantitative version of 
a rather general observation that, by adding
at least $\ell+1$ independent Gaussian rows to a fixed square matrix with at most
$\ell$ zero singular values, we get a rectangular matrix with a strictly
positive $s_{\min}$
almost surely.

Once a satisfactory bound on $s_{\min}(A_{I_{r'},[r'+2\tilde C]})$, $r'\in[\tilde C+1,
n-2\tilde C]$, is obtained,
we rely on the simple deterministic relation between the smallest singular
value and distances to rowspaces:
for every $m\times k$ matrix $Q$, 
$$
\min_{i\in[m]} \dist(H_i(Q),Q_{i,[m]}) \ge  s_{\min}(Q^\top) \ge
m^{-1/2} \min_{i\in[m]} \dist(H_i(Q),Q_{i,[m]})
$$
where $H_i(Q)$ denotes the subspace spanned by row vectors $Q_{j,[m]}$ for $j\neq i$.
In our context, a strong probabilistic lower bound on
$\dist( \spn\{A_{i_t,[r]},\,1\leq t<s\}, A_{i_s,[r]})$ ($s\leq r$)
guaranteed by the partial pivoting strategy, implies 
that with high probability for every $t\in [r-2\tilde C]$,
$$
\dist\big(\spn\{A_{i_s,[r]},\,s\in[r-2\tilde C]\setminus\{t\}\}, A_{i_t,[r]}\big)
\leq n^{O(1)}\dist\big(\spn\{A_{i_s,[r]},\,s\in[r]\setminus\{t\}\}, A_{i_t,[r]}\big)
$$
(see proof of Proposition~\ref{smin prop efwf}),
and via the above deterministic relation to the singular values,
$$
s_{\min}(A_{I_{r-2\tilde C},[r]})\leq n^{O(1)}\dist\big(\spn\{A_{i_s,[r]},\,s\in[r]\setminus\{t\}\}, A_{i_t,[r]}\big).
$$
This, combined with some auxiliary arguments, implies the lower bound on $s_{\min}(A_{I_r,[r]})$.

\section{Intermediate singular values of partially random block matrices}\label{s: poienfpq49f3p9-}

We start with a preparatory material to deal with norms and
intermediate singular values of random matrices.
We first consider a standard deviation estimates for the Hilbert--Schmidt norm of a Gaussian random matrix; see, for example, \cite{BLM}:
\begin{theorem}\label{stdev HS}
Let $G$ be an $u\times t$ random matrix with i.i.d standard Gaussian entries.
Then
$$
\Prob\big\{\|G\|_{HS}\geq \sqrt{ut}+s\big\}\leq 2\exp(-cs^2),\quad s>0,
$$
where $c>0$ is a universal constant.
\end{theorem}

The next proposition was proved in the special case of {\it square} random Gaussian
matrices by Szarek in \cite{Szarek}. In a much more general setting, similar results were obtained earlier by Nguyen \cite{Nguyen}; his argument was later reused in \cite{LTT} to get sharp small ball probability estimates for the condition number of a random square matrix.
\begin{prop}[Singular values of random matrices with continuous distributions]\label{p: small sing shifted}
Let $M$ be an $u\times t$ ($t\geq u$) 
random matrix with i.i.d. standard Gaussian entries.
Then
$$
\Prob\bigg\{s_{u-i}(M)\leq \frac{c'i\,s}{\sqrt{u}}\bigg\}\leq u^{i/2}\,s^{i^2/32},\quad 4\leq i\leq u-1,\quad s\in(0,1],
$$
where $c'\in(0,1]$ is a universal constant.
\end{prop}
We provide a proof of the above proposition in the Appendix.

\bigskip

This section deals with a large number of parameters satisfying multiple constraints;
we group those constraints into blocks for better readability.
We have four ``section-wide'' scalar parameters:
\begin{equation}\label{repsilonxassump}
r\in\N,\;\;\tilde\varepsilon\in(0,1],\;\;i\in\N,\;\;\mbox{ such that }
\tilde\varepsilon (1+\tilde\varepsilon)^{-i}r\geq 2;\;\;\;x\in\N.
\end{equation}
The objective of the section is to study singular values of a block matrix of the form
$$
B=
\begin{bmatrix}
  F & M \\
  W & Q
\end{bmatrix},
$$
where $F$ is a fixed $r\times r$ matrix with prescribed singular values, $M$ is $r\times x$, $W$ is $x\times r$, $Q$ is $x\times x$, and the entries of $M$, $W$, $Q$ are mutually independent standard Gaussians. 
Let   
\begin{align*}
    \varrho_j 
: =  
    \lfloor 
        (1-(1+\tilde \varepsilon)^{-j}) r
    \rfloor, \quad j=0,1,\dots,i.
\end{align*}
Observe that the relation $\tilde\varepsilon (1+\tilde\varepsilon)^{-i}r\geq 2$ from \eqref{repsilonxassump} yields, for $j \in [0,i-1]$,
\begin{align}
\label{eq: deltar}
    \varrho_{j+1}-\varrho_j
>&
    (1-(1+\tilde \varepsilon)^{-j-1}) r -1
-
    (1-(1+\tilde \varepsilon)^{-j}) r 
=
    \tilde\varepsilon(1 + \tilde \varepsilon)^{-j-1}r -1 
> 
    0,
\end{align}
which in turn implies that the sequence
$(\varrho_j)_{j=0}^i$ is strictly increasing. 

Next, let $g(\cdot):(r_j)_{j \in [i]}\to(0,\infty)$ be a strictly positive {\it growth function} satisfying  
\begin{equation}\label{p4t8yfip-98y-8}
\begin{split}
& g(\varrho_j)
\geq 16\,g(\varrho_{j+1}), \quad j=1,2,\dots,i-1;\\
& g(\varrho_i) \leq \sqrt{x}.
\end{split}
\end{equation}
Now, we assume the matrix $F$ satisfies 
\begin{equation}\label{p498h3-98y30t8743yt83}
s_{\varrho_j}(F)\geq g\big(\varrho_j\big),
\quad j \in [i]. 
\end{equation}
In this section, we deal with an arbitrary growth function satisfying the conditions; a specific choice of $g(\cdot)$ will be made later in Section~\ref{s: fpenfqpwiufnpin}.

\medskip



Our objective in this section is to derive the following proposition.
\begin{prop} \label{prop: sec4Main}
There are universal constants $c\in(0,1]$, $\tilde C\geq 1$ with the following property.
Let
$$
B=
\begin{bmatrix}
  F & M \\
  W & Q
\end{bmatrix},
$$
where $F$ is a fixed $r\times r$ matrix, $M$ is $r\times x$,
$W$ is $x\times r$, $Q$ is $x\times x$, and the entries of $M$, $W$, $Q$ are mutually independent standard Gaussians.
Assume that parameters $\tilde\varepsilon\in (0,1]$, $h\in(0,1]$, $r$, $x$, and $i\in\N$ satisfy
\begin{align*}
&r-
\varrho_i\leq x\leq r,\;\;
\tilde \varepsilon x\geq 4, \;\;h\leq 2^{-11}(c')^2\tilde\varepsilon,\\
&3(1+\tilde\varepsilon)^{-i-1}r-(1+\tilde\varepsilon)^{-i}r\geq x+1+11\tilde \varepsilon x,\;\;
\tilde\varepsilon(1+\tilde\varepsilon)^{-i}r\geq 2,
\end{align*}
where $c' \in (0,1]$ is the constant from Proposition~\ref{p: small sing shifted}.
Further, assume \eqref{p498h3-98y30t8743yt83} for the singular values of $F$, for a positive function $g(\cdot)$
satisfying \eqref{p4t8yfip-98y-8}.
Then with probability at least
$$
1-2x^{\tilde\varepsilon x/2}\,
h^{(\tilde\varepsilon x)^2/64}
-4\exp\big(-cx^2\,\tilde\varepsilon/h^2\big)
-\tilde C\exp\big(-c\tilde\varepsilon^2 (1+\tilde\varepsilon)^{-i}r x/h^2\big)
$$
we have
$$
s_{\lfloor(1-(1+\tilde\varepsilon)^{-i-1})(r+x)\rfloor}(B)\geq \frac{c' \tilde\varepsilon h^5\,g\big(\varrho_i 
\big)}{32}.
$$
\end{prop}



Note that if $F=UDV$ is a singular values decomposition of $F$ then, in view of
rotational invariance of the Gaussian distribution,
$$
B=
\begin{bmatrix}
  UDV & M \\
  W & Q
\end{bmatrix}=
\begin{bmatrix}
  U & 0 \\
  0 & \Id_x
\end{bmatrix}
\begin{bmatrix}
  D & U^{-1}M \\
  WV^{-1} & Q
\end{bmatrix}
\begin{bmatrix}
  V & 0 \\
  0 & \Id_x
\end{bmatrix},
$$
where $U^{-1}M$, $WV^{-1}$, and $Q$ have mutually independent
standard Gaussian entries,
and where the singular spectrum of $B$
coincides with that of
$$B':=
\begin{bmatrix}
  D & U^{-1}M \\
  WV^{-1} & Q
\end{bmatrix}.
$$
We can assume without loss of generality that the
diagonal elements (the singular values) of $D$ are arranged in the non-decreasing order when moving from top left to bottom right corner. We will work with the singular spectrum of $B'$ as it will allow to somewhat
simplify the computations.

\medskip

The specific goal is to estimate from below the singular value
$$
s_{\lfloor(1-(1+\tilde\varepsilon)^{-i-1})(r+x)\rfloor}(B)
=s_{\lfloor(1-(1+\tilde\varepsilon)^{-i-1})(r+x)\rfloor}(B')
$$
in terms of $g(\varrho_i
)$. To have a better control on probability estimates, we introduce one more scalar parameter $h\in(0,1]$ which will allow us to balance the precision of the estimate and the probability with which the estimate holds (the smaller $h$ is, the less precise the estimate is and the stronger are probability bounds).
We set
\begin{equation}\label{taudefapfj}
    \tau:= h^4\cdot g(\varrho_i
),
\end{equation}
and let
\begin{equation}\label{elliplusoneafd}
\mbox{$\ell_{i+1}'=$\,the number of singular values of $F$ strictly less than $\tau$};\quad
\ell_{i+1}'':=r - \varrho_i-\ell_{i+1}'.
\end{equation}
Let us remark that $\ell_{i+1}' \leq r - \varrho_i$
since, by the intermediate singular values assumption \eqref{p498h3-98y30t8743yt83} on $F$, we have
$s_{\varrho_i}(F)\geq\tau$.

Set
\begin{equation}\label{Idefakjgns}
I:=[r]\setminus[\ell_{i+1}'].
\end{equation}
Our argument to control $s_{\lfloor(1-(1+\tilde\varepsilon)^{-i-1})(r+x)\rfloor}(B)$ splits into two parts depending on whether $\ell_{i+1}'$ is ``small'' or ``large''. In the former case (see Lemma \ref{lem: sec4lem1}), the matrix $F$ (or $D$) has a well controlled singular spectrum, and our goal is to show that attaching to it $x$ rows and columns of standard Gaussians cannot deteriorate the singular values estimates. In this setting, we completely ignore the first $\ell_{i+1}'$ rows of $B'$, and work with the matrix $B'_{I\times [r+x]}$. 
In the latter case (see Lemma \ref{lem: sec4lem2}), we show that by adding the Gaussian rows and columns we actually improve the control of the singular values. The fact that the top right $\ell_{i+1}'\times x$ corner of $B'$ is a standard Gaussian matrix, plays a crucial role in this setting. The proof of Proposition \ref{prop: sec4Main} follows from Lemma \ref{lem: sec4lem1} and Lemma \ref{lem: sec4lem2}.

\medskip

The high-level proof strategy for both Lemmas~\ref{lem: sec4lem1}
and~\ref{lem: sec4lem2} is similar. We construct a (random) subspace $H$ of $\R^{r+x}$ of dimension at least $(1-(1+\tilde\varepsilon)^{-i-1})(r+x)$, designed in such a way that, under appropriate assumptions on the singular spectra of certain submatrices of $U^{-1}M$, $WV^{-1}$, and $Q$, $\|B'v\|_2$ is large for every unit vector $v\in H$. By the minimax formula for singular values,
$$
s_{\lfloor(1-(1+\tilde\varepsilon)^{-i-1})(r+x)\rfloor}(B')\geq \inf\limits_{v\in H,\,\|v\|_2=1}\|B'v\|_2.
$$
The ``appropriate assumptions'' on the singular spectra are encapsulated in a {\it good} event $\Event_{good}$
which, as we show, has a very large probability.
In what follows, it will be convenient to use notation
\begin{equation}\label{elljdefahg}
\ell_j:= \varrho_j - \varrho_{j-1},\, j \in [i] , \quad \ell_{i+1}:= r - \varrho_i.
\end{equation}
We remark that for every $j=1,2,\dots,i$, by the same derivation as shown in \eqref{eq: deltar}, 
\begin{equation}\label{-9851-9uhfpy}
\ell_j\in [\tilde\varepsilon (1+\tilde\varepsilon)^{-j}r-1,\tilde\varepsilon (1+\tilde\varepsilon)^{-j}r+1].
\end{equation}

\begin{lemma} 
\label{lem: sec4lem1}
There exist universal constants $c\in(0,1]$, $\tilde C\geq 1$ with the following property.
Assume that $i \in \mathbb{N}$, $\tilde\varepsilon\in (0,1]$, $h\in(0,1]$, $r$, and $x$ satisfy
the assumptions of Proposition~\ref{prop: sec4Main},
and assume additionally that
$$
\ell_{i+1}'\leq (1+\tilde\varepsilon)^{-i-1}(r+x)-3\tilde\varepsilon x,
$$
where $\ell_{i+1}'$ is defined in \eqref{elliplusoneafd}.
Denote
$$
\beta:=
\frac{c' \tilde\varepsilon h\tau}{32},
$$
where $c' \in (0,1]$ is the constant from Proposition~\ref{p: small sing shifted}
and where $\tau$ is defined by \eqref{taudefapfj}.
Then with probability at least
$$
1-2x^{\tilde\varepsilon x/2}\,
h^{(\tilde\varepsilon x)^2/64}
-4\exp\big(-cx^2\,\tilde\varepsilon/h^2\big)
-\tilde C\exp\big(-c\tilde\varepsilon^2 (1+\tilde\varepsilon)^{-i}r x/h^2\big)
$$
we have
$$
s_{\lfloor(1-(1+\tilde\varepsilon)^{-i-1})(r+x)\rfloor}(B)\geq \beta.
$$
\end{lemma}
\begin{proof}{\bf{}Construction of subspace $H$.}
Denote by $X_1,X_2,\dots,X_{r+x}\in \R^{r+x}$ an orthonormal basis of the right singular vectors
of the matrix $B_{I\times [r+x]}'$, {\it measurable w.r.t the $\sigma$--field $\sigma(B_{I\times [r+x]}')$},
where $X_j$ corresponds to $s_j(B_{I\times [r+x]}')$, $1\leq j\leq r+x$,
and where $I$ is defined by \eqref{Idefakjgns}.
Note that by interlacing properties of the singular values (see, for example \cite{Djalil}), we have
$$
s_j(B_{I\times [r+x]}')\geq s_{j}(D_{\{\ell_{i+1}'+1,\dots,r\}\times\{\ell_{i+1}'+1,\dots,r\}}),\quad 1\leq j\leq r-\ell_{i+1}';
$$
in particular, $s_{r-\ell_{i+1}'}(B_{I\times [r+x]}')\geq \tau$ everywhere on the probability space.

\medskip

Observe that, conditioned on $\sigma(B_{[r]\times [r+x]}')$, the $x\times \ell_{i+1}''$ matrix 
$$
Y^{(i+1)}:=
\begin{bmatrix}
  WV^{-1} & Q
\end{bmatrix}\,
\begin{bmatrix}
   X_{r+1-\ell_{i+1}} & \dots & X_{r-\ell_{i+1}'}
\end{bmatrix}
=
\begin{bmatrix}
  WV^{-1} & Q
\end{bmatrix}\,
\begin{bmatrix}
   X_{\varrho_i+1} & \dots & X_{r-\ell_{i+1}'}
\end{bmatrix}
$$
has mutually independent standard Gaussian entries.
Denote by $e^{(i+1)}_q$, $1\leq q\leq \min(\lfloor \tilde \varepsilon x\rfloor,\ell_{i+1}'')$,
a random orthonormal system of right singular vectors
of $Y^{(i+1)}$ corresponding to $\min(\lfloor \tilde \varepsilon x\rfloor,\ell_{i+1}'')$ largest singular values of $Y^{(i+1)}$,
and let $E^{(i+1)}\subset \R^{\ell_{i+1}''}$ be the subspace
$$\spn\big\{e^{(i+1)}_q,\; 1\leq q\leq \min(\lfloor \tilde \varepsilon x\rfloor,\ell_{i+1}'')\big\}^\perp.$$

Similarly, for every $1\leq j\leq i$ and for $\ell_j$ given by \eqref{elljdefahg}, we define the $x\times \ell_j$ matrix
$$
Y^{(j)}:=
\begin{bmatrix}
  WV^{-1} & Q
\end{bmatrix}\,
\begin{bmatrix}
   X_{r+1-\sum_{d=j}^{i+1}\ell_{d}} & \dots & X_{r-\sum_{d=j+1}^{i+1}\ell_{d}}
\end{bmatrix}
=
\begin{bmatrix}
  WV^{-1} & Q
\end{bmatrix}\,
\begin{bmatrix}
   X_{\varrho_{j-1}+1} & \dots & X_{\varrho_{j}}
\end{bmatrix}
$$
(again, conditioned on $\sigma(B_{[r]\times [r+x]}')$, $Y^{(j)}$
has mutually independent standard normal entries).
Denote by $e^{(j)}_q$, $1\leq q\leq \min(\lfloor 2^{j-i-1}\tilde \varepsilon x\rfloor,\ell_j)$,
a random orthonormal system of right singular vectors
of $Y^{(j)}$ corresponding to $\min(\lfloor 2^{j-i-1}\tilde \varepsilon x\rfloor,\ell_j)$ largest singular values of $Y^{(j)}$,
and let $E^{(j)}\subset \R^{\ell_{j}}$ be the subspace
$$\spn\big\{e^{(j)}_q,\; 1\leq q\leq \min(\lfloor 2^{j-i-1}\tilde \varepsilon x\rfloor,\ell_j)\big\}^\perp.$$

Consider
the random $x\times (\ell_{i+1}'+x)$ matrix 
$$
\hat Y:=\begin{bmatrix}
  WV^{-1} & Q
\end{bmatrix}\,
\begin{bmatrix}
   X_{r+1-\ell_{i+1}'} & \dots & X_{r+x}
\end{bmatrix}.
$$
Let $\hat e_1,\hat e_2,\dots,\hat e_{x-\lfloor \tilde \varepsilon x\rfloor}$ be
a random orthonormal set of right singular vectors
of $\hat Y$
corresponding to $x-\lfloor \tilde \varepsilon x\rfloor$ largest singular values of $\hat Y$,
and let $\tilde E\subset \R^{\ell_{i+1}'+x}$
be the random subspace of dimension $x-\lfloor \tilde \varepsilon x\rfloor$
defined as
$$
\tilde E:=\spn\{\hat e_1,\hat e_2,\dots,\hat e_{x-\lfloor \tilde \varepsilon x\rfloor}\}.
$$

Now, we construct the (random) subspace $H\subset \R^{r+x}$ as
\begin{align*}
H:=\spn\Big\{
&\begin{bmatrix}
   X_{r+1-\ell_{i+1}'} & \dots & X_{r+x}
\end{bmatrix}(\tilde E),\\
&\begin{bmatrix}
   X_{r+1-\ell_{i+1}} & \dots & X_{r-\ell_{i+1}'}
\end{bmatrix}(E^{(i+1)});\\
&\begin{bmatrix}
   X_{\varrho_{j-1}+1
   } 
   & \dots & X_{\varrho_j
   }
\end{bmatrix}(E^{(j)}), \;1\leq j\leq i
\Big\}.
\end{align*}

Let us check that the constructed subspace satisfies the required lower bound on dimension, that is, $\dim H\geq (1-(1+\tilde\varepsilon)^{-i-1})(r+x)$.
In view of the assumptions on $\ell_{i+1}'$, we have
\begin{align*}
\dim H&\geq x-\lfloor\tilde\varepsilon x\rfloor+\ell_{i+1}''-\lfloor\tilde\varepsilon x\rfloor
+\sum_{j=1}^i \big(\ell_{j}-\lfloor 2^{j-i-1}\tilde \varepsilon x\rfloor\big)\\
&\geq r+x-\ell_{i+1}'-3\tilde\varepsilon x\\
&\geq r+x-(1+\tilde\varepsilon)^{-i-1}(r+x).
\end{align*}    

\bigskip

{\bf{}Defining a good event.}
Denote by $\tilde\Event$ the event
$$
\bigg\{\big\|\begin{bmatrix}
  WV^{-1} & Q
\end{bmatrix} v\big\|_2\geq \frac{c' \lfloor\tilde\varepsilon x\rfloor\,h}{\sqrt{x}}\mbox{ for every unit vector}
\mbox{ $v\in \Big(\begin{bmatrix}
   X_{r+1-\ell_{i+1}'} & \dots & X_{r+x}
\end{bmatrix}(\tilde E)\Big)$}\bigg\},
$$
where the constant $c'$ is taken from Proposition~\ref{p: small sing shifted}. According to our definition of the subspace $\tilde E$, for every unit vector $v$ as above we have
$$
\big\|\begin{bmatrix}
  WV^{-1} & Q
\end{bmatrix} v\big\|_2\geq s_{x-\lfloor \tilde \varepsilon x\rfloor}(\hat Y),
$$
where the matrix $\hat Y$ is $x\times (x+\ell_{i+1}')$ standard Gaussian, in view of the independence of $\begin{bmatrix}
  WV^{-1} & Q
\end{bmatrix}$ from the $\sigma$--field $\sigma(B_{[r]\times [r+x]}')$.

Hence, by Proposition~\ref{p: small sing shifted} applied to $\hat Y$, we get
$$
\Prob(\tilde \Event)\geq 1-x^{\lfloor\tilde\varepsilon x\rfloor/2}\,
h^{\lfloor\tilde\varepsilon x\rfloor^2/32}.
$$
Further, let
$$
\Event^{(i+1)}:=\big\{s_{\lfloor\tilde\varepsilon x\rfloor+1}(Y^{(i+1)})\leq \sqrt{x}/h\big\},
$$
and for every $1\leq j\leq i$, let
$$
\Event^{(j)}:=\big\{s_{\lfloor 2^{-i-1+j}\tilde\varepsilon x\rfloor+1}(Y^{(j)})\leq 2^{i+1-j}\sqrt{\ell_j}/h\big\}.
$$
Since, by our assumptions, $\sqrt{\tilde\varepsilon x}/h\geq 2\sqrt{\ell_{i+1}''}$,
we have, according to Proposition~\ref{stdev HS},
$$
\Prob\big(\big(\Event^{(i+1)}\big)^c\big)
\leq \Prob\big\{\|Y^{(i+1)}\|_{HS}\geq \sqrt{x}\cdot \sqrt{\tilde\varepsilon x}/h\big\}
\leq 2\exp\big(-cx^2\,\tilde\varepsilon/h^2\big),
$$
for a universal constant $c>0$.
Similarly, since for every $j=1,2,\dots,i$, $\sqrt{2^{i+1-j}\tilde\varepsilon x}/h\geq \sqrt{\tilde\varepsilon x}/h\geq 2\sqrt{x}$, we have
$$
\Prob\big(\big(\Event^{(j)}\big)^c\big)
\leq \Prob\big\{\|Y^{(j)}\|_{HS}\geq 2^{i+1-j}\sqrt{\ell_j}\cdot \sqrt{2^{-i-1+j}\tilde\varepsilon x}/h\big\}
\leq 2\exp\big(-c\,2^{i+1-j}\ell_j x\,\tilde\varepsilon/h^2\big).
$$
We define
$$
\Event_{\rm good}:=\tilde\Event\cap \bigcap_{j=1}^{i+1}\Event^{(j)}.
$$
In view of the above,
\begin{align*}
\Prob\big(\Event_{\rm good}\big)
&\geq 1-2x^{\lfloor\tilde\varepsilon x\rfloor/2}\,
h^{\lfloor\tilde\varepsilon x\rfloor^2/32}
-
4\exp\big(-cx^2\,\tilde\varepsilon/h^2\big)\\
&\hspace{2cm}-2\sum_{j=1}^i \exp\big(-c\,2^{i+1-j}\ell_j x\,\tilde\varepsilon/h^2\big)\\
&\geq 1-2x^{\tilde\varepsilon x/2}\,
h^{(\tilde\varepsilon x)^2/64}
-4\exp\big(-cx^2\,\tilde\varepsilon/h^2\big)\\
&\hspace{2cm}-\tilde C\exp\big(-c\tilde\varepsilon^2 (1+\tilde\varepsilon)^{-i}r x/h^2\big),
\end{align*}
for a universal constant $\tilde C>0$.

\bigskip

{\bf{}Checking that $H$ satisfies the required property conditioned on $\Event_{\rm good}$.}
Assuming the conditioning,
pick any unit vector $v\in H$.
We represent $v$ in terms of the basis $X_1,\dots,X_{r+x}$ as
$$
v=\sum_{q=1}^{r+x}a_q\, X_q,
$$
for some coefficients $a_1,\dots,a_{r+x}$ with $\sum_{q=1}^{r+x}a_q^2=1$.
Note that
$$
\|B_{I\times [r+x]}'v\|_2^2
=\sum_{q=1}^{r+x}a_q^2 \,s_{q}(B_{I\times [r+x]}')^2.
$$
If the last expression is greater than $\beta^2$ then we are done.
Otherwise, we have
$$
\sum_{q=1}^{r+x}a_q^2\, s_{q}(B_{I\times [r+x]}')^2\leq \beta^2,
$$
and hence, in particular,
\begin{equation}\label{1-98rh-q98743y087}
\sum_{q=r+1-\ell_{i+1}}^{r-\ell_{i+1}'}a_q^2\leq \frac{\beta^2}{\tau^2}\leq \frac{h^2}{16^2},
\end{equation}
and for every $j=1,2,\dots,i$,
\begin{equation}\label{1iu40837rh0343y087}
\sum_{q=\varrho_{j-1}+1
}^{\varrho_j
}a_q^2
\leq\frac{\beta^2}{g\big(\varrho_j
\big)^2}
\leq \frac{h^2}{16^2}\cdot 16^{j-i}.
\end{equation}
Observe that the last conditions yield
$$
\sum_{q=r+1-\ell_{i+1}'}^{r+x}a_q^2\geq \frac{1}{4}.
$$
In view of conditioning on $\tilde\Event$, this immediately implies
$$
\Big\|\begin{bmatrix}
  W & Q
\end{bmatrix} \sum_{q=r+1-\ell_{i+1}'}^{r+x}a_q X_q\Big\|_2\geq \frac{c' \lfloor\tilde\varepsilon x\rfloor}{4\sqrt{x}}.
$$
Further, in view of conditioning on events $\Event^{(1)},\dots,\Event^{(i+1)}$,
$$
\Big\|\begin{bmatrix}
  W & Q
\end{bmatrix} \sum_{q=r+1-\ell_{i+1}}^{r-\ell_{i+1}'}a_q X_q\Big\|_2
\leq \frac{\beta}{\tau}\cdot \frac{\sqrt{x}}{h},
$$
and for every $j=1,2,\dots,i$,
$$
\bigg\|\begin{bmatrix}
  W & Q
\end{bmatrix} \sum_{q=\varrho_{j-1}+1
}^{\varrho_j
}a_q X_q\bigg\|_2
\leq \frac{\beta}{g\big(\varrho_j
\big)}\cdot \frac{2^{i+1-j}\sqrt{\ell_j}}{h}.
$$
Thus, by the triangle inequality,
\begin{align*}
\big\|\begin{bmatrix}
  W & Q
\end{bmatrix} v\big\|_2
&\geq \frac{c' \lfloor\tilde\varepsilon x\rfloor}{4\sqrt{x}}-\frac{\beta}{\tau}\cdot \frac{\sqrt{x}}{h}
-\sum_{j=1}^i \frac{\beta}{g\big(\varrho_j
\big)}\cdot \frac{2^{i+1-j}\sqrt{\ell_j}}{h}\\
&\geq \frac{c' \tilde\varepsilon \sqrt{x}}{8}-\frac{\beta}{\tau}\cdot \frac{\sqrt{x}}{h}
-\frac{8\,\beta}{g\big(\varrho_i
\big)}\cdot \frac{\sqrt{\tilde\varepsilon(1+\tilde\varepsilon)^{-i}r}}{h},
\end{align*}
where the last relation follows from our assumptions on parameters \eqref{-9851-9uhfpy} and \eqref{p4t8yfip-98y-8}.
The assumption on $\beta$ then implies the result.
\end{proof}

\begin{lemma}
\label{lem: sec4lem2}
There are universal constants $c\in(0,1]$, $\tilde C\geq 1$ with the following property.
Assume that $i \in \mathbb{N}$, $\tilde\varepsilon\in (0,1]$, $h\in(0,1]$, $r$, and $x$ satisfy
the assumptions of Proposition~\ref{prop: sec4Main},
and assume additionally that
$$
\ell_{i+1}'> (1+\tilde\varepsilon)^{-i-1}(r+x)-3\tilde\varepsilon x,
$$
where $\ell_{i+1}'$ is given in \eqref{elliplusoneafd}.
Then with probability at least
$$
1-2x^{\tilde\varepsilon x/2}\,
h^{(\tilde\varepsilon x)^2/64}-
\tilde C\exp\big(-c\tilde\varepsilon^2 (1+\tilde\varepsilon)^{-i}r x/h^2\big)
-2\exp\big(-c\,x^2\,\tilde\varepsilon/h^2\big)
$$
we have
$$
s_{\lfloor(1-(1+\tilde\varepsilon)^{-i-1})(r+x)\rfloor}(B)\geq \tau,
$$
where $\tau$ is defined by \eqref{taudefapfj}.
\end{lemma}
\begin{proof}
{\bf{}Construction of subspace $H$.}
Consider a refinement of the block representation of $B'$:
$$
B'=
\begin{bmatrix}
  \begin{bmatrix}
  D_{i+1}' & 0 & 0 & \dots & 0\\
  0 &   D_{i+1}'' & 0 & \dots & 0\\
  0 &   0 & D_i & \dots & 0\\
  \dots & \dots & \dots & \dots & \dots \\
  0 & 0 & 0 & \dots & D_1
  \end{bmatrix} & \begin{bmatrix}
  M_{i+1}' \\ M_{i+1}'' \\ M_{i}\\ \dots\\ M_1
  \end{bmatrix}\\
  \begin{bmatrix}
  W_{i+1}' & W_{i+1}'' & W_{i} & \dots & W_1
  \end{bmatrix} & Q
\end{bmatrix},
$$
where
$$
U^{-1}M=\begin{bmatrix}
  M_{i+1}' \\ M_{i+1}'' \\ M_{i}\\ \dots\\ M_1
  \end{bmatrix};\; WV^{-1}=\begin{bmatrix}
  W_{i+1}' & W_{i+1}'' & W_{i} & \dots & W_1
  \end{bmatrix};\; D=\begin{bmatrix}
  D_{i+1}' & 0 & 0 & \dots & 0\\
  0 &   D_{i+1}'' & 0 & \dots & 0\\
  0 &   0 & D_i & \dots & 0\\
  \dots & \dots & \dots & \dots & \dots \\
  0 & 0 & 0 & \dots & D_1
  \end{bmatrix}.
$$
In particular, for every $1\leq j\leq i$, the matrix $D_j$ above is $\ell_j\times\ell_j$,
$M_j$ is $\ell_j\times x$, and $W_j$ is $x\times \ell_j$, where $\ell_j$'s are given by
\eqref{elljdefahg}.
Further, $D_{i+1}'$ is $\ell_{i+1}'\times\ell_{i+1}'$, 
$M_{i+1}'$ is $\ell_{i+1}'\times x$, and $W_{i+1}'$ is $x\times \ell_{i+1}'$;
the dimensions of $D_{i+1}''$, $M_{i+1}''$ and $W_{i+1}''$ are defined accordingly.
In this proof, we denote by 
$P_{i+1}':\R^{r+x}\to \R^{\ell_{i+1}'}$ the coordinate projection onto first $\ell_{i+1}'$ coordinates,
by $P_x:\R^{r+x}\to \R^{x}$ the coordinate projection onto last $x$ coordinates,
and, for every $1\leq j\leq i$, denote by $P_j:\R^{r+x}\to \R^{\ell_j}$
the coordinate projection onto $\ell_j$ components starting from $1+\sum_{d=j+1}^{i+1}\ell_d$.

Denote by $e^{(i+1)'}_q$, $1\leq q\leq \ell_{i+1}'-\lfloor \tilde \varepsilon x\rfloor$,
a random orthonormal system of right singular vectors
of $W_{i+1}'$ corresponding to $\ell_{i+1}'-\lfloor \tilde \varepsilon x\rfloor$ largest singular values of $W_{i+1}'$,
and let $\tilde E\subset \R^{\ell_{i+1}'}$ be the subspace
$$\spn\big\{e^{(i+1)'}_q,\; 1\leq q\leq \ell_{i+1}'-\lfloor \tilde \varepsilon x\rfloor\big\}.$$

For every $1\leq j\leq i$,
denote by $e^{(j)}_q$, $1\leq q\leq \min(\lfloor 2^{j-i-1}\tilde \varepsilon x\rfloor,\ell_j)$,
a random orthonormal system of right singular vectors
of $W_j$ corresponding to $\min(\lfloor 2^{j-i-1}\tilde \varepsilon x\rfloor,\ell_j)$ largest singular values of $W_j$,
and let $E^{(j)}\subset \R^{\ell_{j}}$ be the subspace
$$\spn\big\{e^{(j)}_q,\; 1\leq q\leq \min(\lfloor 2^{j-i-1}\tilde \varepsilon x\rfloor,\ell_j)\big\}^\perp.$$

Finally, we construct a random subspace $\hat E\subset\R^x$ as follows.
Denote by $e^{(Q)}_q$, $1\leq q\leq \lfloor \tilde \varepsilon x\rfloor$,
a random orthonormal system of right singular vectors
of $Q$ corresponding to $\lfloor \tilde \varepsilon x\rfloor$ largest singular values of $Q$,
and let $E^{(Q)}\subset \R^x$ be the subspace
$$\spn\big\{e^{(Q)}_q,\; 1\leq q\leq \lfloor \tilde \varepsilon x\rfloor\big\}^\perp.$$
Further, let $e^{(M_{i+1}')}_q$, $1\leq q\leq \ell_{i+1}'-\lfloor \tilde \varepsilon x\rfloor$,
be a random orthonormal system of right singular vectors
of $M_{i+1}'$ corresponding to $\ell_{i+1}'-\lfloor \tilde \varepsilon x\rfloor$ largest singular values of $M_{i+1}'$,
and let
$$
E^{(M_{i+1}')}:=\spn\big\{e^{(M_{i+1}')}_q,\; 1\leq q\leq \ell_{i+1}'-\lfloor \tilde \varepsilon x\rfloor\big\}.
$$
For every $1\leq j\leq i$, let $e^{(M_{j})}_q$, $1\leq q\leq \lfloor 2^{j-i-1}\tilde \varepsilon x\rfloor$,
be a random orthonormal system of right singular vectors
of $M_j$ corresponding to $\lfloor 2^{j-i-1}\tilde \varepsilon x\rfloor$ largest singular values of $M_{j}$,
and let
$$
E^{(M_j)}:=\spn\big\{e^{(M_{j})}_q,\; 1\leq q\leq \lfloor 2^{j-i-1}\tilde \varepsilon x\rfloor\big\}^\perp
\subset\R^x.
$$
We then set
$$
\hat E:=E^{(M_{i+1}')}\;\cap\;\bigcap_{j=1}^i E^{(M_j)}.
$$

\medskip

The subspace $H$ is now defined as
$$
H:=\big\{v\in R^{r+x}:\;P_{i+1}'v\in \tilde E;\;P_{i+1}''v=0;\;P_j v\in E^{(j)},\; 1\leq j\leq i;\;
P_x v\in \hat E\big\}.
$$

Let us check that $H$ satisfies the required assumptions on the dimension.
We have
\begin{align*}
\dim H&=\ell_{i+1}'-\lfloor \tilde \varepsilon x\rfloor
+\sum_{j=1}^i\big(\ell_j-\min(\lfloor 2^{j-i-1}\tilde \varepsilon x\rfloor,\ell_j)\big)
+\dim\hat E\\
&\geq r-\ell_{i+1}''- 2\tilde \varepsilon x
+x-\tilde \varepsilon x-(x-\ell_{i+1}'+\lfloor \tilde \varepsilon x\rfloor)
-\sum_{j=1}^i \lfloor 2^{j-i-1}\tilde \varepsilon x\rfloor\\
&\geq r+2\ell_{i+1}'-\ell_{i+1}-5\tilde \varepsilon x.
\end{align*}
Next, we use the assumption on $\ell_{i+1}'$ and the assumptions on parameters to obtain
\begin{align*}
r+2\ell_{i+1}'-\ell_{i+1}-5\tilde \varepsilon x
&\geq r+2(1+\tilde\varepsilon)^{-i-1}(r+x)-11\tilde\varepsilon x
-(1+\tilde\varepsilon)^{-i}r-1\\
&\geq (1-(1+\tilde\varepsilon)^{-i-1})(r+x).
\end{align*}

\bigskip

{\bf{}Defining a good event.}
Denote by $\tilde\Event$ the event
$$
\bigg\{\big\|  W_{i+1}'v\big\|_2\geq
\frac{c' \lfloor\tilde\varepsilon x\rfloor\,h}{\sqrt{x}}\mbox{ for every unit vector}
\mbox{ $v\in \tilde E$}\bigg\},
$$
where the constant $c'$ is taken from Proposition~\ref{p: small sing shifted}.
According to our definition of the subspace $\tilde E$, for every unit vector $v$ as above we have
$$
\big\|W_{i+1}' v\big\|_2\geq s_{\ell_{i+1}'-\lfloor \tilde \varepsilon x\rfloor}(W_{i+1}').
$$
Hence, by Proposition~\ref{p: small sing shifted} applied to $W_{i+1}'$, we get
$$
\Prob(\tilde \Event)\geq 1-x^{\lfloor\tilde\varepsilon x\rfloor/2}\,
h^{\lfloor\tilde\varepsilon x\rfloor^2/32}.
$$
Further, for every $1\leq j\leq i$, let
$$
\Event^{(j)}:=\big\{s_{\lfloor 2^{-i-1+j}\tilde\varepsilon x\rfloor+1}(W_j)\leq 2^{i+1-j}\sqrt{\ell_j}/h\big\}.
$$
Note that conditioned on $\Event^{(j)}$, we have
$$
\big\|W_{j} v\big\|_2\leq 2^{i+1-j}\sqrt{\ell_j}/h\mbox{ for every unit vector $v\in E^{(j)}$}.
$$
Since for every $j=1,2,\dots,i$, $\sqrt{2^{i+1-j}\tilde\varepsilon x}/h\geq \sqrt{\tilde\varepsilon x}/h\geq 2\sqrt{x}$, we have
$$
\Prob\big(\big(\Event^{(j)}\big)^c\big)
\leq \Prob\big\{\|W_j\|_{HS}\geq 2^{i+1-j}\sqrt{\ell_j}\cdot \sqrt{2^{-i-1+j}\tilde\varepsilon x}/h\big\}
\leq 2\exp\big(-c\,2^{i+1-j}\ell_j x\,\tilde\varepsilon/h^2\big).
$$

Finally, we define events corresponding to a ``good'' realization of $\hat E$.
Let $\Event_{M_{i+1}'}$ be the event
$$
\bigg\{\big\|  M_{i+1}'v\big\|_2\geq
\frac{c' \lfloor\tilde\varepsilon x\rfloor\,h}{\sqrt{x}}\mbox{ for every unit vector}
\mbox{ $v\in E^{(M_{i+1}')}$}\bigg\}.
$$
Repeating the argument for $\tilde \Event$, we get
$$
\Prob(\Event_{M_{i+1}'})\geq 1-x^{\lfloor\tilde\varepsilon x\rfloor/2}\,
h^{\lfloor\tilde\varepsilon x\rfloor^2/32}.
$$
Similarly, adjusting the argument for $\Event^{(j)}$ accordingly, we get that for every $1\leq j\leq i$,
the event
$$
\Event_{M_{j}}:=\big\{s_{\lfloor 2^{-i-1+j}\tilde\varepsilon x\rfloor+1}(M_j)\leq 2^{i+1-j}\sqrt{\ell_j}/h\big\}
$$
has probability at least $1-2\exp\big(-c\,2^{i+1-j}\ell_j x\,\tilde\varepsilon/h^2\big)$,
and that the event
$$
\Event_Q:=\big\{s_{\lfloor \tilde\varepsilon x\rfloor+1}(Q)\leq \sqrt{x}/h\big\}
$$
has probability at least
$$
1-2\exp\big(-c\,x^2\,\tilde\varepsilon/h^2\big).
$$
We define
$$
\Event_{\rm good}:=\tilde\Event\cap\; \bigcap_{j=1}^i \Event^{(j)}\;\cap
\Event_{M_{i+1}'}\cap \Event_Q\cap \;\bigcap_{j=1}^i \Event_{M_{j}}.
$$
In view of the above,
\begin{align*}
\Prob(\Event_{\rm good})
&\geq 1-2x^{\tilde\varepsilon x/2}\,
h^{(\tilde\varepsilon x)^2/64}
-4\sum_{j=1}^i \exp\big(-c\,2^{i+1-j}\ell_j x\,\tilde\varepsilon/h^2\big)
-2\exp\big(-c\,x^2\,\tilde\varepsilon/h^2\big)\\
&\geq
1-2x^{\tilde\varepsilon x/2}\,
h^{(\tilde\varepsilon x)^2/64}-
\tilde C\exp\big(-c\tilde\varepsilon^2 (1+\tilde\varepsilon)^{-i}r x/h^2\big)
-2\exp\big(-c\,x^2\,\tilde\varepsilon/h^2\big),
\end{align*}
for a universal constant $\tilde C\geq 1$.

\bigskip

{\bf{}Checking that $H$ satisfies the required property conditioned on $\Event_{\rm good}$.}
Assuming the conditioning, pick any unit vector $v\in H$.
First, we observe that
$$\|D_{i+1}'P_{i+1}'v\|_2\leq \|D_{i+1}'\|\leq \tau,$$
whereas, by the definition of $\hat E$ and the conditioning,
$$
\|M_{i+1}'P_x v\|_2\geq \frac{c' \lfloor\tilde\varepsilon x\rfloor\,h}{\sqrt{x}}\,\|P_x v\|_2.
$$
Thus, if $\|P_x v\|_2\geq \frac{2\sqrt{x}\tau}{c' \lfloor\tilde\varepsilon x\rfloor\,h}$
then $\|B'v\|_2\geq \|M_{i+1}'P_x v\|_2-\|D_{i+1}'P_{i+1}'v\|_2\geq \tau$, and we are done.

Otherwise, if 
\begin{equation}\label{p0gu340ht837ty0387h}
\|P_x v\|_2< \frac{2\sqrt{x}\tau}{c' \lfloor\tilde\varepsilon x\rfloor\,h}
\leq \frac{4\tau}{c' \tilde\varepsilon \sqrt{x}\,h},
\end{equation}
then, in view of the conditioning (see the definition of $\Event_{M_{j}}$),
$$
\|M_j P_x v\|_2\leq \frac{2^{i+1-j}\sqrt{\ell_j}}{h}\,\frac{2\sqrt{x}\tau}{c' \lfloor\tilde\varepsilon x\rfloor\,h}
\leq \frac{2^{i+3-j}\sqrt{\ell_j/x}\,\,\tau}{c'\tilde\varepsilon h^2},\quad
1\leq j\leq i.
$$
On the other hand, by our assumptions
$$
\|D_j P_j v\|_2\geq g\big(\varrho_j
\big)\,\|P_j v\|_2,\quad 1\leq j\leq i.
$$
Thus, unless $\|B'_{[r]\times[r+x]}v\|_2\geq \tau$, we must have
$$
g\big(\varrho_j
\big)\,\|P_j v\|_2
-\frac{2^{i+3-j}\sqrt{\ell_j/x}\,\,\tau}{c'\tilde\varepsilon h^2}\leq
\|D_j P_j v\|_2-\|M_j P_x v\|_2\leq \tau,\quad 1\leq j\leq i,
$$
implying
\begin{equation}\label{4502t-94t039ty8}
\|P_j v\|_2\leq \frac{2^{i+4-j}\sqrt{\ell_j/x}\,\,\tau}{c'\tilde\varepsilon h^2\,
g\big(\varrho_j
\big)},\quad 1\leq j\leq i.
\end{equation}

As a final step of the proof, we will show that for any unit vector $v\in H$ satisfying conditions
\eqref{p0gu340ht837ty0387h} and~\eqref{4502t-94t039ty8}, one has $\|B'_{\{r+1,\dots,r+x\}\times[r+x]}v\|_2\geq \tau$.
First, note that \eqref{p0gu340ht837ty0387h} and~\eqref{4502t-94t039ty8} imply that
\begin{align*}
\|P_{i+1}'v\|_2&\geq 1-\frac{4\tau}{c' \tilde\varepsilon \sqrt{x}\,h}
-\sum_{j=1}^i \frac{2^{i+4-j}\sqrt{\ell_j/x}\,\,\tau}{c'\tilde\varepsilon h^2\,
g\big(\varrho_j
\big)}\\
&\geq 1-4h^2-
h\,\sum_{j=1}^i \frac{2^{i+4-j}\sqrt{2\tilde\varepsilon (1+\tilde\varepsilon)^{i-j}}}{4^{i-j}}
>1-1-4h^2-64h\geq 1/2,
\end{align*}
whence, in view of conditioning on $\tilde\Event$,
$$
\|W_{i+1}'P_{i+1}'v\|_2\geq
\frac{c' \tilde\varepsilon \sqrt{x}\,h}{4}
\geq \frac{c' \tilde\varepsilon \sqrt{(1+\tilde\varepsilon)^{-i}r}\,h}{4}.
$$
Now, for every $1\leq j\leq i$, by the above and in view of conditioning on $\Event^{(j)}$,
\begin{align*}
\|W_jP_j v\|_2&\leq 
\frac{2^{i+1-j}\sqrt{\ell_j}}{h}
\frac{2^{i+4-j}\sqrt{\ell_j/x}\,\,\tau}{c'\tilde\varepsilon h^2\,
g\big(\varrho_j
\big)}\\
&\leq
\frac{2^{2i+6-2j} (1+\tilde\varepsilon)^{-j}r\,\,\tau/\sqrt{(1+\tilde\varepsilon)^{-i}r}}{c' 16^{i-j}\, h^2\,
g\big(\varrho_i
\big)},
\end{align*}
whence
$$
\sum_{j=1}^i \|W_jP_j v\|_2\leq 
\frac{2^{7} \sqrt{(1+\tilde\varepsilon)^{-i}r}\,\,\tau}{c'  h^2\,
g\big(\varrho_i
\big)}
\leq\frac{c' \tilde\varepsilon \sqrt{(1+\tilde\varepsilon)^{-i}r}\,h}{16}.
$$
Similarly, in view of conditioning on $\Event_Q$, we get
$$
\|QP_x v\|_2\leq \frac{4\tau}{c' \tilde\varepsilon \sqrt{x}\,h} \frac{\sqrt{x}}{h}
\leq \frac{4 h^4\sqrt{x}}{c' \tilde\varepsilon h^2}<\frac{c' \tilde\varepsilon \sqrt{(1+\tilde\varepsilon)^{-i}r}\,h}{16}.
$$
Thus,
$$
\|B'_{\{r+1,\dots,r+x\}\times[r+x]}v\|_2
\geq \|W_{i+1}'P_{i+1}'v\|_2-\sum_{j=1}^i \|W_jP_j v\|_2
-\|QP_x v\|_2\geq \frac{c' \tilde\varepsilon \sqrt{(1+\tilde\varepsilon)^{-i}r}\,h}{8}\geq \tau,
$$
and the proof is complete.
\end{proof}

\bigskip


\section{Random polytopes, and distances to pivot rows}\label{fepijnfeofiniqwjnpij}

Let $B$ be an $n \times m$ matrix with $m\ge n$, and assume that every
square submatrix of $B$ is invertible.  We define recursively the sequence of indices
$\{i_r(B)\}_{r \in [n]}$, vectors $\{v_r(B)\}_{r \in [n]}$ in $\R^m$,
and polytopes $\{K_r(B)\}_{r \in [n]}$ in $\R^m$ as follows.

\medskip

Set $v_1(B):=e_1$ and $I_0(B):=\emptyset$. For $r$ from $1$ to $n$, 
\begin{align}
	i_r(B) &:= \mbox{argmax}_{i \in [n]\backslash I_{r-1}(B)} 
		| \langle v_r(B) ,\, (B_{i,[m]})^\top \rangle |,  \nonumber \\
	I_r(B) &:= \{i_s(B)\}_{s\in [r]}, \nonumber \\
	v_r(B) &:= \big( -((B_{I_{r-1},[r-1]})^{-1} B_{I_{r-1},r})^\top ,\, 1 ,\, 
		\underbrace{0, \dots, 0}_{ m-r \mbox{ components } }
	\big)^\top, \nonumber \\
	K_r(B) &:= \big\{ x \in \mathbb{R}^m \, :\, \forall s \in [r],\, | 
		\langle v_s(B) ,\, x \rangle |  \le | \langle v_s(B),\, (B_{i_s(B),[m]})^\top \rangle|
		\big\}.  \label{eq: polyDef}
\end{align}
Observe that $v_r(B)$ is a null vector of $B_{I_{r-1}, [r]}$ such that $v_{r}(B)=1$,
and that
$i_r(B)$ can be viewed as the index of the $r$-th
pivot row in the Gaussian Elimination with Partial Pivoting
with the [rectangular] input matrix $B$.
Note also that our definition of the sets $I_r(B)$ is consistent with that of the sets $I_r(A)$ discussed earlier.
The above construction does not provide any tie-breaking rules for the choice of the indices $i_r(B)$ in case when respective expressions have multiple maximizers.
In our setting, however (when $B$ is Gaussian), each pivot is unique with probability one, and hence the choice of a tie-breaking rule is irrelevant.
We have an immediate relation
\begin{align} \label{eq: polyDim}
  K_r(B) = K_r(B_{[n],[r]}) \times \R^{m-r} \mbox{ and }
  \sigma_m(K_r(B)) = \sigma_r(K_r(B_{[n],[r]})),\quad r \in [m-1],
\end{align}
where $\sigma_k$ is the standard Gaussian measure for the corresponding dimension. 

Suppose we have performed $r$ steps of the {\rm GEPP} algorithm on the $n\times n$ Gaussian matrix $A$. Let $I\subset[n]$ have size $r$, and
condition on a realization of $I_r=I$ and $A_{I,[r]}$, which determines $K_r(A)$. Then, for every $j\in [n]\setminus I$,
the $j$--th row of $A$ is a Gaussian vector conditioned to stay within the polytope $K_r(A)$.
Formally, for every $I\subset[n]$ of size $r$, every $j\in[n]\setminus I$, and every Borel subset $\mathcal B$
of $\R^n$,
$$
\Prob\big\{(A_{j,[n]})^\top\in \mathcal B\;|\;A_{I,[r]}\big\}
=\frac{\sigma_n(\mathcal B\cap K_r(A))}{\sigma_n(K_r(A))}\quad
\mbox{almost everywhere on the event $\{I_r(A)=I\}$.}
$$
We will not directly use the above description of the conditional
distribution of $(A_{j,[n]})^\top$ given $A_{I,[r]}$;
instead, we will apply a simple
decoupling based on Lemma~\ref{lem: secPolyBasic}
which essentially establishes the same property.
We provided the above formula only to clarify our argument.


\begin{lemma} \label{lem: secPolyBasic}
Suppose $B$ is an $n\times m$ random matrix such that its entries are i.i.d and have continuous distribution.  Then, for $r \in [n]$ and $I \subset [n]$ with $|I|=r$, almost surely the following assertions are equivalent:
\begin{enumerate}
    \item $I_r(B)=I$, 
    \item $\forall s \in [r],\, i_s(B) = i_s(B_{I,[m]}),\, 
	v_s(B)= v_s(B_{I,[m]}), \mbox{ and } K_s(B) = K_s(B_{I,[m]}) $,
	\item $\forall j \in [n]\backslash I$, $(B_{j,[m]})^\top \in K_r(B_{I,[m]})$. 
\end{enumerate}
\end{lemma}
\begin{proof}
We can assume without loss of generality that everywhere on the probability space,
all square submatrices of $B$ are invertible, and for all $1\leq s\leq n-1$
and $I'\subset [n]$ of size $s$, the expression
$$
| B_{i,s+1} - B_{i,[s]} (B_{I',[s]})^{-1}B_{I',s+1}|
$$
attains its maximum on $i\in [n]\backslash I'$ at a unique point.
These conditions ensure that the above algorithm for generating
$\{i_r(B)\}_{r \in [n]}$, $\{v_r(B)\}_{r \in [n]}$, $\{K_r(B)\}_{r \in [n]}$
have a uniquely determined output i.e no ambiguity in the choice
of the indices $i_r(B)$ occurs.

Notice that the implication
$2 \Rightarrow 3$ is straightforward by the
above definitions. We will check the implications $1\Rightarrow 2$
and $3\Rightarrow 1$ below.


\medskip

{\bf Implication $1 \Rightarrow 2$}.
Condition on the event $\{I_r(B)=I\}$.
We have $v_1(B)=v_1(B_I) = e_1$ and   
$$i_1(B_{I,[m]})= {\rm argmax}_{i \in I}|\langle v_1(B) , (B_{i,[m]})^\top\rangle|
={\rm argmax}_{i \in [n]}|\langle v_1(B) , (B_{i,[m]})^\top\rangle|=i_1(B).$$ 
Further, assume that $k<r$ is such that
$\forall s \in [k]$, $i_s(B)= i_s(B_{I,[m]})$.
Since $I_k(B)=I_k(B_{I,[m]})$, we also have $v_{k+1}(B)=v_{k+1}(B_{I,[m]})$, and thus, 
\begin{align*}
i_{k+1}(B_{I,[m]}) &= {\rm argmax}_{i \in I\backslash I_k(B_{I,[m]})}|\langle v_{k+1}(B_{I,[m]}) , B_{i,[m]}^\top\rangle|\\
&={\rm argmax}_{i \in [n]\backslash I_k(B)}|\langle v_{k+1}(B) , B_{i,[m]}^\top\rangle|
=i_{k+1}(B).
\end{align*}
Thus, by induction, $i_s(B)= i_s(B_{I,[m]})$ for all $s\in [r]$, whence
$v_s(B)=v_s(B_{I,[m]})$, $K_s(B)=K_s(B_{I,[m]})$, and
$I_s(B)=I_s(B_{I,[m]})$
for all $s \in [r]$.

\medskip

{\bf Implication $3 \Rightarrow 1$}. The argument is based on induction
just as above.
We assume that $\forall j \in [n]\backslash I$, $(B_{j,[m]})^\top \in K_r(B_{I,[m]})$.
First, 
$v_1(B)=v_1(B_{I,[m]})=e_1$,
and since
$$
K_r(B_{I,[m]})\subset\big\{x \in \mathbb{R}^m \, :\,
		|\langle v_1(B_{I,[m]}) ,\, x \rangle |  \le | \langle v_1(B_{I,[m]}),\, (B_{i_1(B_{I,[m]}),[m]})^\top \rangle|\big\},$$
we have
$$
{\rm argmax}_{j \in [n]\backslash I}| \langle e_1, B_{j,[m]}^\top\rangle| \le
|\langle e_1, (B_{i_1(B_{I,[m]}),[m]})^\top \rangle |.
$$
On the other hand, by the definition of $i_1(B_{I,[m]})$,
$$
|\langle e_1, (B_{i_1(B_{I,[m]}),[m]})^\top \rangle |=
{\rm argmax}_{i \in I} |\langle e_1, (B_{i,[m]})^\top\rangle|.
$$
As a consequence, $i_1(B)=i_1(B_{I,[m]})\in I$, completing the base step of the induction.
Now, let $k<r$
be an integer such that $\forall s \in [k], i_s(B)=i_s(B_{I,[m]})$. Since $v_{k+1}(B)=v_{k+1}(B_{I,[m]})$ by our construction, and since
$$
K_r(B_{I,[m]})\subset\big\{x \in \mathbb{R}^m \, :\,
		|\langle v_{k+1}(B_{I,[m]}) ,\, x \rangle |  \le | \langle v_{k+1}(B_{I,[m]}),\, (B_{i_{k+1}(B_{I,[m]}),[m]})^\top \rangle|\big\},
$$
we get
\begin{align*}
{\rm argmax}_{j \in [n]\backslash I}| \langle v_{k+1}(B), (B_{j,[m]})^\top\rangle| &\le
|\langle v_{k+1}(B_{I,[m]}), (B_{i_{k+1}(B_{I,[m]}),[m]})^\top \rangle |\\
&=
{\rm argmax}_{i \in I\backslash I_{k}(B)} |\langle e_1, (B_{i,[m]})^\top\rangle|,
\end{align*}
which implies that $i_{k+1}(B)=i_{k+1}(B_{I,[m]})\in I$.
Thus, we conclude by induction that $i_s(B)=i_s(B_{I,[m]})\in I$ for all $s\in [r]$,
and the result follows.
\end{proof}

As the first main result of the section, we have a probability
estimate for the event that the Gaussian measure of the polytope
$K_r(A)$ is below a given threshold:
\begin{prop}[Gaussian measure of $K_r(A)$] \label{prop: polyMain}
Let $A$ be an $n\times n$ Gaussian matrix. Then for any $r \in [n-1]$ and any $t \geq 2$, 
	$$
\Prob \{ \sigma_n(K_r(A)) \le n^{-t} \} 
\le	n^{-t(n-r)/2}.
	$$
\end{prop}
\begin{proof}
We start by writing
\begin{align*}
\Prob \{ \sigma_n(K_r(A)) \le n^{-t} \} 
& = 
\sum_{ I \subset [n],\, |I|=r}
\Prob \big\{ I_r(A) = I \mbox{ and }\sigma_n(K_r(A)) \le n^{-t} \big\}.
\end{align*}
For each summand, we apply Lemma~\ref{lem: secPolyBasic} to get
\begin{align}
	 \Prob &\big\{ I_r(A) = I \mbox{ and }\sigma_n(K_r(A)) \le n^{-t} \big\} \nonumber \\
	&=	\Prob \big\{ I_r(A) = I \mbox{ and } \sigma_n(K_r(A_{I,[n]})) \le n^{-t}  \big\} 
	\nonumber \\
	&=	\Prob \Big\{  \sigma_n(K_r(A_{I,[n]})) \le n^{-t} 
		\mbox{ and } \forall i \in [n]\backslash I,\, 
               (A_{i,[n]})^\top  \in K_r(A_{I,[n]}) \Big\}. \label{eq: poly1}
\end{align}
Since $K_r(A_{I,[n]}) $ and $(A_{i,[n]})^\top$ for
$i\in [n]\backslash I$ are independent, we get
\begin{align*}
	\eqref{eq: poly1} \;&\le   \Prob \big\{  \sigma_n(K_r(A_{I,[n]})) \le n^{-t} \big\}
	\cdot \Prob \Big \{ \forall i \in [n]\backslash I,\, 
               (A_{i,[n]})^\top  \in K_r(A_{I,[n]}) \; \big\vert \;
               \sigma_n(K_r(A_{I,[n]})) \le n^{-t} \Big\} \\
	&\le 1\cdot (n^{-t})^{n-r}. 
\end{align*}
Finally, in view of
the standard bound ${ n \choose n-r} \le n^{n-r} $ for the number of subsets $I \subset [n]$ of size $r$,
and by the union bound argument, the result follows.
\end{proof}

\begin{lemma} \label{lem: polyToDist}
Let $A$ be an $n\times n$ Gaussian matrix, and let $r\in[n]$ and $\tau\in(0,1)$
be parameters.
Then, conditioned on the event $\big\{\sigma_n(K_r(A)) \ge \tau\big\}$,
\begin{align*}
		\dist( H ,  (A_{i_r,[r]})^\top )  \ge 
		\sqrt{\frac{\pi}{2} }\, \tau,
	\end{align*}
where $H$ is the subspace of $\R^r$ spanned by vectors
$(A_{i_s,[r]})^\top$, $s\in [r-1]$. 
\end{lemma}
\begin{proof}
  Let $v := v_r(A)/ \|v_r(A)\|_2$ and let
  $P:\R^n\to\R^r$ be the orthogonal projection onto the span of $\{e_s\}_{s\in[r]}$. 
From the definition of $v$, we have that
$Pv$ is a unit normal to the hyperplane
$H$ in $\R^r$. Then
$$
\dist( H ,  (A_{i_r,[r]})^\top )
                =  |\langle Pv, (A_{i_r,[r]})^\top \rangle| = |\langle v, (A_{i_r,[n]})^\top \rangle| :=s.
$$
It remains to note that, by the definition of $K_r(A)$, on the event
$\big\{\sigma_n(K_r(A)) \ge \tau\big\}$ we have 
\begin{align*}
		\tau &\le  \sigma_n( K_r(A) ) 
		\le \sigma_n\big( \{ x \in \mathbb{R}^n\,:\, 
			| \langle v_r(A), x \rangle | \le | \langle v_r(A), (A_{i_r,[n]})^\top \rangle | 
			 \} \big) \\
			 & = \sigma_n\big( \{ x \in \R^n \,:\, 
			| \langle v, x \rangle | \le s
			 \} \big) 
		=  \int_{ -s }^{s}  \frac{1}{\sqrt{2\pi}} \exp( - t^2/2) {\rm d}t 
		\le \frac{2s}{\sqrt{2\pi}}.
	\end{align*}
\end{proof}

As a corollary, we obtain the following probabilistic bound
on the distance between $(A_{i_r,[r]})^\top$
and the span of ``previous'' rows (selected at previous steps of the GEPP process)
$(A_{i_s,[r]})^\top$, $s\in [r-1]$:
\begin{cor} \label{cor: polyDist}
Let $A$ be an $n\times n$ Gaussian matrix. For $ t \ge 2$ and $ r \in [n-1]$, with probability at least $1-n^{ - t(n-r)/2 }$ we have 
$$
\dist( H ,  (A_{i_r,[r]})^\top ) \geq \sqrt{\frac{\pi}{2} }\, n^{-t},
$$
where $H$ is the random subspace of $\R^r$
spanned by vectors $(A_{i_s,[r]})^\top$, $s\in [r-1]$. 
\end{cor}
\begin{proof}
In view of Lemma~\ref{lem: polyToDist}, the statement would follow
as long as $\Prob\big\{\sigma_n(K_r(A)) \ge n^{-t}\big\}\geq
1-n^{ - t(n-r)/2 }$.
The latter is verified in Proposition~\ref{prop: polyMain}.
\end{proof}

\section{A recursive argument}\label{s: fpenfqpwiufnpin}

The goal of this section is to bound from below the
intermediate singular values $s_{r-k}(A_{I_r,[r]})$ for every $r$ greater than some absolute constant and for $k$ of a constant order. 
We will start with bounding the intermediate singular values in the bulk of the singular spectrum first and then will recursively apply Proposition~\ref{prop: sec4Main} to provide lower bounds for smaller and smaller intermediate singular values. 

As we mentioned in the overview of the proof, the
intermediate singular values $s_{r-k}(A_{I_r},[r])$ for $k \gg r^{1/2}\,{\rm polylog}(n)$
can be easily estimated from below with high probability
by taking the union bound over the estimates
of $s_{r-k}(A_{I,[r]})$ (see Proposition \ref{p: small sing shifted})
for $I \subset[n]$ with $|I|=r$.
To bound $s_{r-k}(A_{I_r,[r]})$ from below for smaller values of $k$ we apply the following strategy. We choose an appropriate positive integer $r'<r$, condition on a realization of $I_{r'}$ and $A_{I_{r'},[r']}$, and,
for any $I$ with $I_{r'} \subset I \subset [n]$ and $|I|=r$, apply Proposition~\ref{prop: sec4Main} with $B:=A_{I,[r]}$ and $F:=A_{I_{r'},[r']}$. This way,
$s_{r-k}(A_{I,[r]})$ is bounded below with high probability conditioned on
an event that the intermediate singular values
$s_{r-k'}(A_{I_{r'},[r']})$ are well bounded for every $k'$ slightly bigger than $k$.

\medskip

\begin{defi}\label{def: eventIntSing}
For an integer $k \in [n]$ and parameters
$p,\beta \ge 1$, let $\Event_{\rm is}(p,k,\beta)$ be the event that 
\begin{align*}
	\forall r \in [k+1,n],\quad\quad
			s_{r - k}(A_{I_r,[r]}) \ge n^{-\beta/(50p)}.
\end{align*}
and $\Event_{\rm rec}(p,k,\beta)$ be the event that 
\begin{align*}
	\forall r \in [k+1,n-2k],\quad\quad
			s_{\min}\big((A_{I_r,[r+2k]})^\top\big) \ge n^{-\beta/(20p)}.
\end{align*}
\end{defi}
Note that although $n$ is not mentioned explicitly in the list of parameters
for $\Event_{\rm is}(p,k,\beta)$, it clearly depends on the underlying matrix dimension.

\medskip

The next proposition is the main result of this section:
\begin{prop}\label{prop: sec5Main}
There is a universal constant $C>0$ with the following property.
Let $p \ge 1$. Then there exist positive integers $n_0:=n_0(p)$, $120p\leq k_0:=k_0(p)\leq Cp$, and a positive real number $300p \le \beta_0:=\beta_0(p)\leq Cp$, so that for any $n\ge n_0$ and $\beta \ge \beta_0$, 
\begin{align*}
		\Prob \big( \Event_{\rm is}(p,k_0(p),\beta)^c \big)
	\le n^{-2\beta+o_n(1)}. 
\end{align*}
\end{prop}
We remark that the lower bounds on $k_0(p)$ and $\beta_0(p)$ in the assumptions of the proposition are not required in the proof but will be needed later.
As a corollary of the proposition (proved in the end of this section), we have
\begin{cor} \label{cor: sminTall}
For any $p\ge 1$, $\beta \ge \beta_0(p)$,
\begin{align} \label{eq: sminTall}
\Prob\big(  \Event_{\rm rec}(p,k_0(p), \beta)^c \big) =
n^{-2\beta +o_n(1)}. 
\end{align}
\end{cor}

Now, we present a technical version of the above proposition. We
introduce several ``section-level'' parameters.
Let $\tilde \varepsilon>0$ be a small constant
and $L$ be a large integer to be determined later. The parameter $\tilde \varepsilon$ will play the same role as in Proposition~\ref{prop: sec4Main}.
Next, let $$m_0:= \lceil L/ \tilde \varepsilon^5 \rceil$$
and let $s_1$ be the smallest integer such that $ 2^{s_1} m_0 \ge n$. Then we define the finite sequence
$m_1,\dots,m_{s_1+1}$, where
$$
	\forall s \in [s_1-1],\quad m_s := 2^s m_0 \quad
	\mbox{ and }\quad m_{s_1+1} := n. 
$$
The main technical result in this section is the following
\begin{lemma} \label{lem: sec5Main}
  Fix $\tilde \varepsilon \in (0, 1/100]$ and $L\ge 1/ \tilde \varepsilon$. Then there exists a positive integer $n_0$ (depending on $\tilde \varepsilon$ and $L$)
  such that for any $n \ge n_0$ and $s \in [0, s_1-1]$, we have for every $\alpha\geq 4$:
\begin{align}
		\Prob \big\{  \exists r \in [m_{s+1}, m_{s+2}]
			\mbox{ s.t.\ } 
			s_{r - \lceil 9L/ \tilde \varepsilon \rceil }(A_{I_r,[r]}) 
			\le
			n^{-C( \tilde \varepsilon )\alpha }
			\big\}
	\le 
		n^{-c( \tilde \varepsilon) \alpha L },
	\end{align}
where $ c( \tilde \varepsilon )$ and $C( \tilde \varepsilon )$ are positive
constants which depend on $c, \tilde C$ from Proposition \ref{prop: sec4Main} and on $ \tilde \varepsilon$. 
\end{lemma}

\begin{proof}[Proof of Proposition~\ref{prop: sec5Main}]
Let $\tilde{\varepsilon}:= 1/100$.
We can safely assume that the constants $ c( \tilde \varepsilon )$ and $C( \tilde \varepsilon )$
from Lemma~\ref{lem: sec5Main} satisfy $c( \tilde \varepsilon )\in(0,1]$
and $C( \tilde \varepsilon )\geq 1$.
Choose
\begin{align*}
		L:=\max\bigg(\frac{1}{\tilde \varepsilon},
		80p\, \frac{ C( \tilde{\varepsilon}) }{ c(\tilde{\varepsilon})}\bigg).
\end{align*}	
Let $\beta_0:= \max\{ 4 c(\tilde{\varepsilon})L,300p\}$,
$k_0(p):= \max\{ \lceil 9L/ \tilde \varepsilon \rceil, 120p\}$, and
let $\beta\geq\beta_0$. Applying Lemma~\ref{lem: sec5Main} with
$\alpha\geq 4$ satisfying $ \beta/(40p) = C(\tilde{\varepsilon})\alpha$, we get
$$
\Prob \big\{\exists\,\,r\geq m_1\mbox{ s.t.\ } 
			s_{r - \lceil 9L/ \tilde \varepsilon \rceil }(A_{I_r,[r]}) 
			\le
			n^{-\beta/(40p) }
			\big\}
	\le 
		(s_1+1) n^{-c( \tilde \varepsilon) \alpha L }
		\leq (s_1+1) n^{-2\beta},
$$
implying the result for large enough $n$.
\end{proof}

\bigskip

\noindent {\it{}For the rest of the section, we fix $s \in [0, s_1-1]$.}

\medskip

\subsection{Choice of parameters and the growth function}\label{akjdfnoifuwbfoiubfo}


\begin{defi}[Definition of $\izero,\ione$, $f_i$, $r_i$]\label{iifrdef}
For a given positive integer $L$ and for $ \tilde \varepsilon \in (0,1/4]$, let $\izero$ be the integer such that 
$$ ( 1 + \tilde \varepsilon)^{-\izero} m_s  \ge  
   \tilde \varepsilon m_s/10 
	> ( 1+ \tilde \varepsilon)^{-\izero-1} m_s
  .$$  
and let $\ione$ be the integer such that 
\begin{align} \label{eq: i_1}
	 ( 1 + \tilde \varepsilon )^{-\ione} m_s \ge L/\tilde \varepsilon
	> ( 1+ \tilde \varepsilon)^{-\ione-1} m_s. 
\end{align}
Note that $\tilde \varepsilon m_s/10\geq L/\tilde \varepsilon$, and hence $\izero\leq \ione$.

\medskip

For every $i \in [\izero,\ione]$, we define a non-decreasing function 
\begin{equation}\label{akejnfpifunfpi3unfp3i}
	f_i(r):= \Big\lfloor \frac{r}{1 + (1 + \tilde \varepsilon )^{-i} } 
	\Big\rfloor,\quad r\in\N. 
\end{equation}
Further, we define a collection of integers $\{r_i\}_{i\in[\ione+1]}$ inductively as follows.
Whenever $i \in [\izero]$, we set $r_i := m_s $.
Further, assuming that $r_i$ has been defined for some $i\in[\izero,\ione]$, we let $r_{i+1}$
be the smallest integer such that $f_i(r_{i+1}) \ge r_i$. 
Note that $m_s=r_1\le r_2 \le \dots \le r_{\ione+1}$.
\end{defi}

We recall our strategy:
to bound the singular value
$s_{\lfloor ( 1 - ( 1+ \tilde \varepsilon)^{-i-1})r \rfloor}(A_{I_r,[r]})$ from below, we will select an appropriate integer $r'<r$ and apply Proposition~\ref{prop: sec4Main} with $B := A_{I,[r]}$ and $F:=A_{I_{r'},[r']}$, taking the union
bound over all subsets $I\subset[n]$ with $I_{r'}\subset I$ and $|I|=r$. The function $f_i$ defined above, determines the choice of $r'$, namely, we choose
 $$
 r' := f_i(r),
 $$
 for $r_i \le r \le m_{s+2}$. The indices $\izero$ and $\ione$ defined above, determine the
 range of application for the inductive strategy; namely, $\ione$ marks the largest index $i$ for which our induction argument can be applied, and $\izero$ indicates a threshold value
 below which the corresponding singular values $s_{\lfloor ( 1 - ( 1+ \tilde \varepsilon)^{-i})r \rfloor}(A_{I_r,[r]})$ concentrate very strongly and can bounded directly with help of Proposition~\ref{p: small sing shifted} and a simple union bound argument.
 
 The goal of this subsection is to verify certain relations between the introduced parameters, that need to be satisfied in order to apply the results on the singular values established earlier. Since the results here are of purely computational nature, we present
 the proofs in the Appendix.

\begin{lemma}[Inequalities for $\ione$] \label{lem: sec5i0i1bound}
Let $ \tilde \varepsilon \in (0, 1/4)$ and $L\ge 1/ \tilde \varepsilon $. For $r \in [m_s, m_{s+2} ]$, 
\begin{align} \label{eq: sec5i1DefBound} 
		r - \lfloor ( 1 - ( 1 + \tilde \varepsilon)^{-\ione-1})r \rfloor 	
	\le  9L/ \tilde \varepsilon. 
	\end{align}
Further, 
\begin{align} \label{eq: sec5i0i1bound}
		\ione \le  2\log (m_s) / \tilde \varepsilon. 
	\end{align}
\end{lemma}

\begin{lemma}[Assumptions in Proposition~\ref{prop: sec4Main}] \label{lem: irxepsilon}
Let $\tilde \varepsilon \in (0, \frac{1}{28})$ and $L\ge 4$. Fix $i \in [\izero, \ione]$ and assume that
$\tilde r$ satisfies $r_{i+1}\leq \tilde r\leq m_{s+2}$.
Let $ r : =  f_i( \tilde r) $ and $ x :=  \tilde r - r$. Then, 
\begin{align} \label{eq: xRange} 
		( 1 + \tilde \varepsilon )^{-i} r \le 
		x
		\le \frac{21}{20} ( 1 + \tilde \varepsilon )^{-i} r. 
	\end{align}
Moreover, $ i, r , x$, and $\tilde \varepsilon$ satisfy the assumptions in Proposition \ref{prop: sec4Main}, specifically, 
	\begin{align*}
		&r - \lfloor (1 - ( 1 + \tilde \varepsilon)^{-i}) r \rfloor \le x \le r, & 
		\tilde \varepsilon x &\ge 4, \\ 
		&3 ( 1 + \tilde \varepsilon )^{-i-1} r - ( 1+ \tilde \varepsilon )^{-i}r 
				      \ge x + 1 + 11 \tilde \varepsilon x, &
		\tilde \varepsilon (1 + \tilde \varepsilon )^{-i} r & \ge 2 .
	\end{align*}
\end{lemma}

For a given $i \in [\izero, \ione]$, the number $\tilde r$
satisfying the assumptions of the above lemma can only be chosen if $r_{i+1}\leq m_{s+2}$.
In the next statement, we show that the inequality is satisfied for every admissible $i$
(and in fact verify a slightly stronger bound):

\begin{lemma}[An upper bound on $r_{\ione+1}$] \label{lem: r_iBound}
Let $ \tilde \varepsilon \in (0, 1/28)$ and $L\ge 4$. Then, $r_{\ione+1}\le 2m_s=m_{s+1}$.
\end{lemma}

\medskip

To construct the growth function $g(\cdot)$ from 
\eqref{p4t8yfip-98y-8}, we first define an auxiliary positive function $g_{s}(\cdot)$,
and then set
$$
g\big(\lfloor(1-(1+\tilde\varepsilon)^{-j})r\rfloor\big):=g_s(j)
$$
for all admissible $j$. The formal definition of $g_s(\cdot)$
is given below.

\begin{defi}\label{apskjfnoiuoiuniunp}
Let $\alpha \ge 1$ be a parameter. For $ i \in [\izero]$, we set
\begin{align}
	g_s (i)  :=
		\frac{c'}{2\sqrt{m_s}}  16^{-i}m_s n^{-\alpha},
\end{align}
where $c'$ is the constant from Proposition~\ref{p: small sing shifted}.

For $i \in [\izero,\ione]$, we apply a recursive definition:
$$
	g_s(i+1)
:=
	\frac{c' \tilde \varepsilon}{32}  h_s(i)^5 
	g_s(i),
$$
where $h_s(i)$ is given by
\begin{align} \label{eq: defHi}
	h_s(i) := &  \exp  \Big( - \max \Big\{ \frac{128 \alpha\,\log n }{
		\tilde \varepsilon^2 ( 1+ \tilde \varepsilon )^{-i} m_s } ,\, 
	    C_h \Big\} \Big),
\end{align}
and where $C_h\geq -\log\big(2^{-11}(c')^2\tilde\varepsilon\big)$ is a constant depending only
$c, \tilde C$ (from Proposition \ref{prop: sec4Main}) and $c'$,
and which we shall determine in Lemma~\ref{lem: Probhx}.
\end{defi}

The function $h_s(i)$ corresponds to the parameter $h$ in Proposition \ref{prop: sec4Main},
and is constructed in such a way that certain union bound argument that we are going to apply further works.
The next lemma clarifies the choice of the constant $C_h$ from the above definition:

\begin{lemma} \label{lem: Probhx}
The constant $C_h$ can be chosen so that the following holds.
For $i \in [\izero, \ione]$ and $ \tilde r \in [r_{i+1}, m_{s+2}]$, let $r := f_i( \tilde r )$ and $ x := \tilde r - r$. 
Then, 
\begin{align}
		&2x^{\tilde\varepsilon x/2}\,
h_s(i)^{(\tilde\varepsilon x)^2/64}
+4\exp\big(-cx^2\,\tilde\varepsilon/h_s(i)^2\big)
+\tilde C\exp\big(-c\tilde\varepsilon^2 (1+\tilde\varepsilon)^{-i}r x/h_s(i)^2\big) 
\nonumber \\
  &\hspace{1cm}\le \exp \big( - \alpha \log(n) x \big).
\end{align}
\end{lemma}

In the next lemma we verify the crucial bound on the growth function which will ultimately guarantee
a polynomial in $n$ bound on the intermediate singular values:
\begin{lemma}\label{fekjfnwefkjnfpwijfnwelkfjn}
	There exists $C( \tilde \varepsilon )>1$ which depends on $c', \tilde C$ from Proposition \ref{prop: sec4Main}, on $C_h$, and on $ \tilde \varepsilon$, such that 
	\begin{align}
          \forall \alpha \ge 1,\,\quad\quad
	g_s(\ione+1) 
	\ge  n^{ - C( \tilde \varepsilon )\alpha}.
	\end{align}
\end{lemma}

\subsection{Good events, and probability estimates}

\begin{defi}\label{apfjnepfiunfpweifunwpin}
For $ i \in [\ione+1]$ and $r \in [r_i, n]$, let $ \Event(r, i)$ be the event that
$$
	s_{\lfloor ( 1 - ( 1+ \tilde \varepsilon)^{-i})r \rfloor  } 
	(A_{I_r, [r]})  \ge 
	g_s(i),
$$
where $g_s(\cdot)$ is given in Definition~\ref{apskjfnoiuoiuniunp}
and $r_i,\ione$ are taken from Definition~\ref{iifrdef}.
Further, we denote $ \Event(r, [i]) := \bigcap_{ j \in [ i ] } \Event(r,j)$. 
\end{defi}

\begin{lemma} \label{lem: sec5Base}
For 
$\tilde \varepsilon \in (0, 1/100]$, $L\ge 4$ and $ \alpha \ge 4$,
\begin{align}
		\Prob \Big( 
			\bigcup\limits_{r \in [m_s, m_{s+2}]} \Event( r, [\izero] )^c \Big)
	\le 
		\exp \big( - \alpha L \log n \big). 
	\end{align}
\end{lemma}
\begin{proof}
Fix $i \in [\izero]$ and $r \in [m_s, m_{s+2}]$. Let $q:= r - \lfloor (1 - (1+ \tilde \varepsilon)^{-i})r \rfloor $.
Then
\begin{align} \label{eq: qEstimate}
	q \ge  ( 1 + \tilde \varepsilon )^{-i}r \ge (1 + \tilde \varepsilon)^{-i}m_s \ge 16^{-i}m_s.
\end{align}
We recall that in view of the definition of $m_s$ and $\izero$, necessarily
$q<r$; furthermore,
\begin{equation}\label{akdfjnpfiwunfpwi}
q \ge (1 + \tilde \varepsilon )^{-i}m_s \ge (1+ \tilde \varepsilon )^{-\izero}m_s \ge 
\tilde \varepsilon m_s/10  \ge \tilde \varepsilon m_0/10 \ge  L/ \tilde \varepsilon^4 \ge 32.
\end{equation}
For each $I \subset [n]$ with $|I|=r$,  
\begin{align*}
	\Prob\big\{ 
		s_{r-q}(A_{I,[r]}) < g_s(i)		
		\big\}
&= 	\Prob \Big\{ 
		s_{r-q}(A_{I,[r]}) < \frac{ c' }{  2 \sqrt{m_s} } 
		 16^{-i}\,m_s\,  n^{-\alpha} 
		\Big\} \\
&\le  	\Prob \Big\{ 
		s_{r-q}(A_{I,[r]}) < \frac{ c' }{\sqrt{r} } 
		 q n^{-\alpha} 
		\Big\} \, \, \quad ( \mbox{by \eqref{eq: qEstimate} and the definition of $m_s,m_{s+2}$}) \\
&\le    \exp\big( \log(r)q/2 - \log(n)\alpha q^2/32 \big) \,\, \quad
( \mbox{by Proposition \ref{p: small sing shifted}}).
\end{align*}
Applying \eqref{akdfjnpfiwunfpwi},
we conclude that 
\begin{align*}
	\Prob\big\{ 
		s_{r - q}(A_{I,[r]}) < g_s(i)		
		\big\}
&\le  \exp\big( \log(n)q^2/64 - \log(n) \alpha q^2/32 \big) \\
&\le  
\exp \Big( - \log(n) \alpha  \Big(\frac{\tilde \varepsilon m_s}{10} \Big)^{2}/64 \Big).
\end{align*}
We complete the proof with the
union bound argument.
There are ${n \choose r} \le (en/r)^r \le \exp( r\log n) $ subsets $I \subset [n]$ with $|I|=r$. As $r \le m_{s+2}\le 4m_s$, and in view of the definition of $m_s$ and our choice of $\tilde\varepsilon$,
\begin{align*}
	\Prob\big\{ s_{r-q}(A_{I_r,[r]} ) <g_s(i) \big\} &\le 
	\Prob \big\{ 
		\exists I\subset [n] \mbox{ with }|I|=r \mbox{ such that }  
		s_{r - q}(A_{I,[r]}) < g_s(i)		
	\big\} \\
	&\le   \exp\Big( - \log(n)\alpha \Big(\frac{\tilde \varepsilon m_s}{10} \Big)^{2}/128 \Big),
\end{align*}
By applying the union bound argument again over all
$i \in [\izero]$ and all $ r \in [m_s, m_{s+2}]$, the statement of the lemma follows. 
\end{proof}

\begin{lemma} \label{lem: sec5Ind}
Assume 
$\tilde \varepsilon \in (0, 1/100]$, $L\ge 4$ and $ \alpha \ge 4$.
For $ i \in [\izero,\ione]$ and $ \tilde r \in [r_{i+1}, m_{s+2}] $, set $r := f_i( \tilde r)$
and $x:=\tilde r-r$.
Then
	\begin{align*}
		\Prob \Big( \Event(\tilde{r}, i+1)^c 
			\cap \Event( r, [i] ) 
			\cap \{ 
				\sigma_n( K_r(A) ) \ge n^{-\alpha/2} 
			\} 
		\Big)
	\le 
		\exp \Big( - \frac{1}{4} \alpha \log(n) x \Big)
	\le \exp \Big( - \frac{1}{4} \alpha \log(n) \frac{L}{\tilde \varepsilon} \Big),
	\end{align*}
	where the random polytope $K_r(A) \subset \R^n$ was
	defined in \eqref{eq: polyDef}, and where $\sigma_n$ is the standard Gaussian measure in $\R^n$. 
\end{lemma}
\begin{proof}
We start by noting that the
last inequality in the statement of the lemma follows from the estimate $x \ge (1 + \tilde \varepsilon)^{-\ione}m_s \ge L/ \tilde \varepsilon$ (see Lemma~\ref{lem: irxepsilon}
and the definition of $\ione$). 

We further partition the event in question so that
\begin{align} \label{eq: secInd00}
	 \Prob &\Big( \Event(\tilde{r}, i+1)^c \cap \Event(r, [i]) \cap 
	\{\sigma(K_r(A)) \ge n^{-\alpha/2}\}  \Big) \nonumber \\
	&=  
	\sum_{I \subset [n], |I|=r } 
	\Prob \Big( \Event(\tilde{r}, i+1)^c \cap \Event(r, [i]) \cap 
		\{\sigma_n(K_r(A)) \ge n^{-\alpha/2}\} \cap \{ I_r(A) = I\}  \Big) .
\end{align}
For each $I \subset [n]$ with $|I|=r$, we define $\Event(I)$ to be the event
\begin{align*}
	\forall j \in [i], \, 
	s_{\lfloor ( 1 - ( 1+ \tilde \varepsilon)^{-j})r \rfloor  } 
	(A_{I,[r]} )  \ge 
	g_s(j),
\end{align*}
and note that for each admissible $I$, $\Event(r, [i])\cap \{ I_r(A) = I\}
\subset \Event(I)$.

For $I\subset [n]$ with $|I|=r$ and $J \subset [n] \backslash I $ with $|J| = x $, let $\Event(I, J)$ be the event that 
$$
	s_{\lfloor (1 - (1 + \tilde \varepsilon)^{-i-1})(\tilde r) \rfloor }
	( A_{ I\cup J, [\tilde r]} )
\ge
	g_s(i+1).
$$
Denote $K(I):= K_r(A_{I,[r]})\subset \R^r$. 
Then each term in \eqref{eq: secInd00} can be bounded as  
\begin{align}\label{eq: Eventri0}
	\Prob &\Big( \Event(\tilde{r}, i+1)^c \cap \Event(r, [i]) \cap 
		\{\sigma_n(K_r(A)) \ge n^{-\alpha/2}\} \cap \{ I_r(A) = I\}  \Big)  \nonumber \\
&\hspace{1cm}\le 
	\sum_{ J \subset [n] \backslash I ,\, |J| =x }
	\Prob \Big( \Event(I,J)^c \cap \Event(I) \cap 
		\{\sigma_n(K(I)) \ge n^{-\alpha/2}\} \cap \{ I_r(A) = I\}  \Big).
\end{align}

\medskip

Now, assume
that for every $I \subset [n]$ and $J \subset [n] \backslash I$ with $|I| = r$ and $|J|=x$, 
\begin{align} \label{eq: Eventri1}
\Prob \Big( \Event(I,J)^c \cap 
		\Event(I)\; \Big\vert\;\{\sigma_n(K(I)) \ge n^{-\alpha/2}\} \cap  \{ I_r(A) = I\}  \Big)
\le \exp \Big( - \frac{1}{2} \alpha \log (n) x \Big). 
\end{align}
Clearly, for each $I \subset [n]$ with $|I|=r$, 
$$
|\{ J \subset [n]\backslash I\,:\, |J|=x\}| = {n-r \choose x} \le \Big( \frac{en}{x}\Big)^x \le \exp( \log(n)x) \le \exp\Big( \frac{1}{4}\alpha \log (n)x\Big).
$$
Together with \eqref{eq: Eventri0} and \eqref{eq: Eventri1}, this gives 
\begin{align*}
	\eqref{eq: secInd00} \le & 
	\sum_{I \subset [n], |I|=r }
	\exp\Big( \frac{1}{4}\alpha \log (n)x\Big)
	\exp \Big(-\frac{1}{2} \alpha \log (n) x \Big) \,
\Prob\{ I_r(A) = I\}= \exp \Big(-\frac{1}{4} \alpha \log (n) x \Big),
\end{align*}
and the result follows.

\bigskip

Thus, it remains to show \eqref{eq: Eventri1}.
By Lemma \ref{lem: secPolyBasic}, almost everywhere on the probability space
we have 
$$ {\bf 1}_{\{ I_r(A) = I\}}
= {\bf 1}_{\{ \forall j \in [n]\backslash I,\, (A_{j,[n]})^\top \in K_r(A_{I,[n]})\}}
={\bf 1}_{\{ \forall j \in [n]\backslash I,\, (A_{j,[r]})^\top \in K(I)\}}.
$$
Hence,
\begin{align}\nonumber
 \Prob &\Big( \Event(I,J)^c \cap \Event(I) \,\Big\vert  \,
		\{\sigma_r(K(I)) \ge n^{-\alpha/2}\} \cap  \{ I_r(A) = I\}  \Big)\\
&\le  
\frac{
\Prob \Big( \Event(I,J)^c \cap \Event(I) \cap  
		\{\sigma_r(K(I)) \ge n^{-\alpha/2}\} \cap \big\{ \forall j \in [n]\backslash (I\cup J),\, (A_{j,[r]})^\top \in K(I)\big\} \Big)
}{
\Prob \Big(
		\{\sigma_r(K(I)) \ge n^{-\alpha/2}\} \cap   \big\{ \forall j \in [n]\backslash I,\, (A_{j,[r]})^\top \in K(I)\big\} \Big)
}.\label{pifnfq034ufn0384nqpf34npu}
\end{align}
In view of the joint independence of the entries of $A$, we obtain
\begin{align}
 \Prob &\Big( \Event(I,J)^c \cap \Event(I) \cap  
		\{\sigma_r(K(I)) \ge n^{-\alpha/2}\} \cap \big\{ \forall j \in [n]\backslash (I\cup J),\, (A_{j,[r]})^\top \in K(I)\big\} \Big) \nonumber \\
&=  
\mathbb{E}_{A_{I,[r]}} \Big[ {\bf 1}_{ \Event(I) \cap \{\sigma_r(K(I)) \ge n^{-\alpha/2}\}}
\cdot \Prob\big\{ \Event(I,J)^c \,\big|\, A_{I,[r]}\big\} \nonumber \\
&\hspace{3cm}\cdot \Prob\big\{ \forall j \in [n]\backslash (I\cup J) ,\, (A_{j,[r]})^\top \in K(I)\,\big|\, A_{I,[r]} \big\}
\Big], \label{eq: secInd01}
\end{align}
where the outer expectation is with respect to $A_{I,[r]}$.

\medskip

For each realization
of $A_{I,[r]}$ such that the event $\Event(I)$ holds,
we apply Proposition \ref{prop: sec4Main} with 
$$
\begin{bmatrix}
	F& M  \\
	W& Q 
\end{bmatrix}
:=
\begin{bmatrix}
	A_{I,[r]}& A_{I,[r+1,r+x]}  \\
	A_{J,[r]}& A_{J,[r+1,r+x]} 	
\end{bmatrix}
$$
to bound $ \Prob( \Event(I,J)^c\, |\,A_{I,[r]} )$.
Let $g(\cdot)$ be a growth function satisfying 
\begin{align*}
	\forall j \in [i],\quad
	 g(\lfloor ( 1 - ( 1 + \tilde \varepsilon)^{-j})r \rfloor ) = 
	 g_s(j),
\end{align*}
where $g_s(\cdot)$ is given in Definition~\ref{apskjfnoiuoiuniunp}.
Since $16\, g_s(j+1) \le g_s(j)$ for $j \in [i-1]$ and $g_s(j) \le 1$ for $ j \in [i]$,
the function $g(\cdot)$ defined this way
satisfies \eqref{p4t8yfip-98y-8}. Recall that on the event
$\Event(I)$ we have 
$$
	s_{\lfloor (1 - ( 1 + \tilde \varepsilon)^{-j})r \rfloor }(A_{I,[r]})
	\ge g(\lfloor ( 1 - ( 1 + \tilde \varepsilon)^{-j})r \rfloor ),\quad  j \in [i]. 
$$
We apply Proposition~\ref{prop: sec4Main} with $g(t)$
and with $h:=h_s(i)$ (see Definition~\ref{apskjfnoiuoiuniunp}) so that
$$ \frac{c' \tilde\varepsilon h^5\,g\big(\lfloor(1-(1+\tilde\varepsilon)^{-i})r\rfloor\big)}{32}
= g_s(i+1)
$$
(observe that our parameters $r,x$ satisfy the assumption of the proposition
due to Lemma \ref{lem: irxepsilon}, and that $h$ satisfies the assumption $h\leq 2^{-11}
(c')^2 \tilde\varepsilon$ in view of the assumptions on the constant $C_h$
in Definition~\ref{apskjfnoiuoiuniunp}).
We get
\begin{align*}
\Prob\big( \Event(I,J)\,\big|\,
A_{I,[r]}\big) &=\Prob\Big(
s_{\lfloor(1-(1+\tilde\varepsilon)^{-i-1})(r+x)\rfloor}(A_{I\cup J,
[\tilde r]})\geq g_s(i+1)\;\big|\;
A_{I,[r]}
\Big)\\
&\geq 1-2x^{\tilde\varepsilon x/2}\,
h^{(\tilde\varepsilon x)^2/64}
-4\exp\big(-cx^2\,\tilde\varepsilon/h^2\big)
-\tilde C\exp\big(-c\tilde\varepsilon^2 (1+\tilde\varepsilon)^{-i}r x/h^2\big).
\end{align*}
In view of Lemma~\ref{lem: Probhx}, this implies
$$
\Prob\big( \Event(I,J)^c \,\big|\, A_{I,[r]}\big) \le \exp\big( - \alpha \log(n)x\big).
$$
Combining the last inequality with \eqref{eq: secInd01}, we obtain
\begin{align}
\Prob &\Big( \Event(I,J)^c \cap \Event(I) \cap  
		\{\sigma_r(K(I)) \ge n^{-\alpha/2}\} \cap \big\{ \forall j \in [n]\backslash (I\cup J),\, (A_{j,[r]})^\top \in K(I)\big\} \Big) \nonumber \\
&\le 
\exp\big( - \alpha \log(n)x\big)\cdot
\Prob \Big( 
		\{\sigma_r(K(I)) \ge n^{-\alpha/2}\} \cap \big\{ \forall j \in [n]\backslash (I\cup J),\, (A_{j,[r]})^\top \in K(I)\big\} \Big).
		\label{eq: secIndNominator}
\end{align}
Next, we will treat the denominator in the estimate
\eqref{pifnfq034ufn0384nqpf34npu}.
By Fubini's theorem,
\begin{align*}
    \Prob &\Big(   
		\{\sigma_r(K(I)) \ge n^{-\alpha/2}\} \cap   \big\{ \forall j \in [n]\backslash I,\, (A_{j,[r]})^\top \in K(I)\big\} \Big)\\
&=
\Exp_{A_{I,[r]}} \Big[ {\bf 1}_{\{\sigma_r(K(I)) \ge n^{-\alpha/2}\}}
\cdot \Prob\big( \forall j \in J,\, (A_{j,[r]})^\top \in K(I)\,  \big|\, A_{I,[r]}\big)\cdot \\
& \hspace{4.55cm} \cdot \Prob\big( \forall j \in [n]\backslash (I\cup J),\, (A_{j,[r]})^\top \in K(I)\, \big| \,A_{I,[r]} \big)
\Big] .
\end{align*}
Almost everywhere on the event $\{\sigma_r(K(I)) \ge n^{-\alpha/2}\}$ we have 
$$
\Prob\big( \forall j \in J,\, (A_{j,[r]})^\top \in K(I)\,  \big|\, A_{I,[r]}\big)
\geq  n^{-\alpha x/2},
$$
whence
\begin{align*}
\Prob& \Big(   
		\{\sigma_r(K(I)) \ge n^{-\alpha/2}\} \cap   \big\{ \forall j \in [n]\backslash I,\, (A_{j,[r]})^\top \in K(I)\big\} \Big)\\
&\geq n^{-\alpha x/2}\, \Prob\Big(   
		\{\sigma_r(K(I)) \ge n^{-\alpha/2}\} \cap   \big\{ \forall j \in [n]\backslash (I\cup J),\, (A_{j,[r]})^\top \in K(I)\big\} \Big).
\end{align*}
Together with \eqref{eq: secIndNominator} and \eqref{pifnfq034ufn0384nqpf34npu},
this yields
$$
\Prob \Big( \Event(I,J)^c \cap \Event(I) \,\Big\vert  \,
		\{\sigma_r(K(I)) \ge n^{-\alpha/2}\} \cap  \{ I_r(A) = I\}  \Big)
\le \exp\Big( - \frac{1}{2} \alpha \log(n) x\Big),
$$
and the proof of \eqref{eq: Eventri1} is complete.
\end{proof}

\medskip

At this point, we are ready to prove the main lemma in this section.

\begin{proof}[Proof of Lemma \ref{lem: sec5Main}]
First, recall that in view of Lemma~\ref{lem: r_iBound}, $r_{\ione+1}\leq m_{s+1}$,
and that in view of \eqref{eq: sec5i1DefBound} we have
$r - \lceil 9/ \tilde \varepsilon \rceil\leq \lfloor(1-(1+\tilde\varepsilon)^{-\ione+1})r\rfloor$,
whence
\begin{align*}
\bigcup_{ r \in [m_{s+1}, m_{s+2}] } 
	&\big\{ s_{r - \lceil 9/ \tilde \varepsilon \rceil}(A_{I_r,[r]})
	< g_s(\ione+1) \big\}\\
	&\subset
	\bigcup_{ r \in [r_{\ione+1}, m_{s+2}] } 
	\big\{ s_{\lfloor(1-(1+\tilde\varepsilon)^{-\ione+1})r\rfloor}(A_{I_r,[r]})
	< g_s(\ione+1) \big\}\\
	&=\bigcup_{r \in [r_{\ione+1}, m_{s+2}] } \Event^c(r,\ione+1),
\end{align*}
where we used the definition of the events $\Event(r,i)$ (Definition~\ref{apfjnepfiunfpweifunwpin}).
To estimate the probability of the union of the events in the last line,
we shall embed it into a specially structured collection.

Let $r' := f_{\ione}(m_{s+2})$, where $f_{\cdot}(\cdot)$
was defined in \eqref{akejnfpifunfpi3unfp3i}. 
We have
\begin{align*}
\bigcup_{r \in [r_{\ione+1}, m_{s+2}] } &\Event^c(r,\ione+1)
\subset \bigcup_{r\in[m_s,r']}\{ \sigma_n(K_r(A)) < n^{-\alpha/2} \}\\
&\cup\;
\bigcup_{r \in [r_{\ione+1}, m_{s+2}] } \Big(\Event^c(r,\ione+1)\cap \Big(
\bigcap_{r\in[m_s,r']}\{ \sigma_n(K_r(A)) \geq n^{-\alpha/2} \}\Big)\Big).
\end{align*}
To be able to apply a recursive bound from the last lemma,
we use the bounds $r_i\leq f_i( \tilde r ) \leq m_{s+2}$, $ \tilde r \in [r_{i+1}, m_{s+2}]$,
$i \in [\izero+1,\ione]$, to
write
\begin{align*}
&\bigcup_{r \in [r_{\ione+1}, m_{s+2}] } \Big(\Event^c(r,\ione+1)\cap \Big(
\bigcap_{r\in[m_s,r']}\{ \sigma_n(K_r(A)) \geq n^{-\alpha/2} \}\Big)\Big)
\subset \bigcup_{i \in [\izero]}\bigcup_{r \in [m_s,m_{s+2}]}
		\Event^c(r,i)\\
&\hspace{1.5cm}\cup\bigcup_{ i \in [\izero,\ione] } \bigcup_{ \tilde{r} \in [r_{i+1}, m_{s+2} ] }
	\Big( \Event^c( \tilde r , i+1 ) \cap \Event( f_i( \tilde r ) , [i] ) \cap 
		\Big(
\bigcap_{r\in[m_s,r']}\{ \sigma_n(K_r(A)) \geq n^{-\alpha/2} \}\Big)\Big).
\end{align*}
Thus, using that $f_i( \tilde r )\leq r'$ for $i \in [\izero+1,\ione]$ and
that $m_{s+2} - r' \ge L $, we get
\begin{align}
	\Prob&\Big(
		\bigcup_{r \in [r_{\ione+1}, m_{s+2}] } \Event^c(r,\ione+1) 	\Big) \nonumber \\
&\le  
	\sum_{ r \in [m_s,r'] } \Prob\{ \sigma_n(K_r(A)) < n^{-\alpha/2} \} 
	+ \sum_{ i \in [\izero] } \sum_{ r \in [m_s,m_{s+2}] } 
	\Prob ( \Event^c(r,i) )  \nonumber \\
    &\hspace{1cm} + \sum_{ i \in [\izero,\ione] } \sum_{ \tilde{r} \in [r_{i+1}, m_{s+2} ] }
	\Prob \Big( \Event^c( \tilde r , i+1 ) \cap \Event( f_i( \tilde r ) , [i] ) \cap 
		\{ \sigma(K_{ f_i( \tilde r )}) \ge n^{-\alpha/2} \} \Big) \nonumber \\
&\le  
	\underbrace{n\cdot n^{-\alpha L/4}}_{\mbox{ by Proposition \ref{prop: polyMain}}}
	+ \underbrace{n^2\exp \big( - \alpha L \log(n) \big)}_{\mbox{ by Lemma \ref{lem: sec5Base}}}
	+ \underbrace{ n^2 \exp \Big( - \frac{1}{4} \alpha \log(n) \frac{L}{\tilde \varepsilon} \Big)}_{\mbox{ by Lemma~\ref{lem: sec5Ind}}}
\le  
	n^{- c(\tilde \varepsilon )\alpha L }
\end{align}
for some $c( \tilde \varepsilon )>0$. 
It remains to note that in view of Lemma~\ref{fekjfnwefkjnfpwijfnwelkfjn},
$g_s(\ione+1)\geq n^{-C(\tilde\varepsilon)\alpha}$ for some $C(\tilde\varepsilon)$.
\end{proof}

\begin{proof}[Proof of Corollary~\ref{cor: sminTall}]
For brevity, we denote $k_0:=k_0(p)$.
We fix $r \in [k_0+1,n-2k_0]$ and 
let $F_{r} \subset \R^{I_r}$ be the right singular subspace of the matrix
$(A_{I_r,[r]})^\top$ corresponding to $k_0$ smallest singular values of $(A_{I_r,[r]})^\top$ 
(since almost everywhere on the probability space $I_r$ is unambiguously determined, and the singular values of $(A_{I,[r]})^\top$ are distinct for every
$I\subset[n]$ with $|I|=r$, $F_r$ is uniquely defined).
Now, let us define the event $\tilde\Event_r(\beta)$ that 
\begin{align}
s_{\min}\big((A_{I_r,[r+1,r+2k_0]})^\top \big\vert_{F_r} \big) &\ge  n^{-\beta/(40p)}&
&\mbox{and}& 
\big\| (A_{I_r,[r+1,r+2k_0]})^\top \big\vert_{F_r^\perp} \big\| &\le 3\sqrt{\beta n},
\end{align}
where 
$(A_{I_r,[r+1,r+2k_0]})^\top \big\vert_{F_r} $ and $(A_{I_r,[r+1,r+2k_0]})^\top \big\vert_{F_r^\perp}$
are linear operators
obtained by restricting the domain of $(A_{I_r,[r+1,2k_0]})^\top$
to $F_r$ and $F_r^\perp$, respectively. 
Then, conditioned on the intersection
$\tilde\Event_r(\beta) \cap \Event_{\rm is}(p,k_0,\beta)$,
for any $v \in \R^{I_r} \backslash \{0\}$,  
\begin{align*}
    \big\| (A_{I_r,[r+2k_0]})^\top \,v\big\|_2 &= \Bigg\| \Bigg( 
   \begin{matrix}
      (A_{I_r,[r]})^\top v \\  (A_{I_r,[r+1,r+2k_0]})^\top v 
   \end{matrix} 
   \Bigg) \Bigg\|_2\\
&\ge 
\max \Big\{ 
n^{-\beta/(50p)}\,\big\|P_{F_r^\perp}\,v\big\|_2   ,\, 
n^{-\beta/(40p)} \,\big\|P_{F_r}\,v\big\|_2  - 3\sqrt{\beta n} \,\big\|P_{F_r^\perp}\,v\big\|_2
\Big\},
\end{align*}
where $P_{F_r}$ and $P_{F_r^\perp}$ are orthogonal projections onto $F_r$ and $F_r^\perp$, respectively. 
Consider two cases.
\begin{itemize}

\item Suppose 
$ \|P_{F_r^\perp}v\|_2 \ge \frac{1}{4} \frac{n^{-\beta/(40p)}}{ 3\sqrt{\beta n }  } \|P_{F_r}v\|_2$. Then, 
\begin{align*}
\|v\|_2 &=  \sqrt{ \|P_{F_r^\perp}v\|_2^2 + \|P_{F_r} v\|_2^2 } \le  \|P_{F_r^\perp}v\|_2 
\sqrt{1^2 + \Big( 
   \frac{1}{4} \frac{n^{-\beta/(40p)}}{ 3\sqrt{\beta n }  }\Big)^{-2} 
   } \\
&\le O(\sqrt{\beta})\,n^{\beta/(40p) + \frac{1}{2}}\,\|P_{F_r^\perp}v\|_2,
\end{align*}
which implies 
\begin{align*}
   n^{-\beta /(50p)} \|P_{F_r^\perp}v\|_2  
   \ge O(\beta^{-1/2})\,n^{-\beta/ (50p) - \beta/(40p) - \frac{1}{2}}\,\|P_{F_r^\perp}v\|_2
   \ge n^{-\beta/(20p)} \|v\|_2,
\end{align*}
where the last inequality holds because $\beta \ge \beta_0(p) \ge  300p$ and since $n$ is sufficiently large depending on $p$. 

\item In the case $ \|P_{F_r^\perp}v\|_2 < \frac{1}{4} \frac{n^{-\beta/(40p)}}{ 3\sqrt{ \beta n }  } \|P_{F_r}v\|_2$, we have 
$$
n^{-\beta/(40p)} \big\|P_{F_r}v\|_2  - 3\sqrt{\beta n}  \big\|P_{F_r^\perp}v\big\|_2
\ge \frac{3}{4} n^{-\beta/(40p)} \| P_{F_r}v\|_2 
\ge n^{-\beta/(20p)} \|v\|_2.
$$
\end{itemize}
Since the above estimate holds for all $v \in \R^{I_r}\backslash \{0\}$,
we conclude that everywhere on the intersection $\tilde\Event_r(\beta) \cap \Event_{\rm is}(p,k_0,\beta)$,
\begin{align*}
    s_{\min}\big( (A_{I_r,[r+2k_0]})^\top \big) \ge n^{-\beta/(20p)}.  
\end{align*}
Therefore, for $p\ge 1$ and $\beta\ge \beta_0(p)$,
$$
\Event_{\rm is}(p,k_0,\beta) \cap \Big( \bigcap_{r \in [k_0+1, n-2k_0]} \tilde\Event_r(\beta) \Big)\subset
\Event_{\rm rec}(p,k_0,\beta),
$$
and thus 
\begin{align*}
\Prob \big( \Event_{\rm rec}(p,k_0,\beta)^c \big) \le 
\Prob \big( \Event_{\rm is}(p,k_0,\beta)^c \big)+
\sum_{r \in [k_0+1, n-2k_0]} \Prob( \tilde \Event_r(\beta)^c ). 
\end{align*}
Since in view of Proposition~\ref{prop: sec5Main},
$\Prob \big( \Event_{\rm is}(p,k_0,\beta)^c \big)
\le n^{-2\beta +o_n(1)}$, the corollary will follow
if we show that $\Prob( \tilde \Event_r(\beta)^c )\le n^{-2\beta -1 +o_n(1)}$. 

\medskip

{\it From now on, we fix
$r \in [k_0+1, n-2k_0]$ and
condition on a realization of $A_{[n],[r]}$
such that the set $I_r$ and the space $F_r$ are uniquely determined.
We will write $\tilde\Prob$ and $\tilde\Exp$
to denote the corresponding conditional probability and conditional expectation.}

\medskip

The independence of the entries of the matrix $A$ implies that 
$Q:=(A_{I_r,[r+1,r+2k_0]})^\top \big\vert_{F_r}$ and 
$W:=(A_{I_r,[r+1,r+2k_0]})^\top \big\vert_{F_r^\perp}$
are (standard) Gaussian linear operators from $F_r$
to $\R^{2k_0}$ and from $F_r^\perp$ to $\R^{2k_0}$, respectively. 
For the purpose of estimating the operator norm and least singular values, 
we can view
$W$ and $Q$ as matrices with i.i.d $N(0,1)$ entries of dimensions
$2k_0\times (r-k_0)$ and $2k_0\times k_0$, respectively;
more specifically, we can define standard Gaussian matrices $\tilde W$
and $\tilde Q$ of dimensions $2k_0\times (r-k_0)$ and $2k_0\times k_0$ such that
everywhere on the probability space the singular spectrum of $W$ and $\tilde W$,
and of $Q$ and $\tilde Q$, agree.

It is well known
that the expected operator norm of any standard
Gaussian matrix is no more than the sum
of square roots of its dimensions (see, for example, \cite[Section~7.3]{Vershynin}).
Hence,
\begin{align*}
\tilde\Exp\, \|W\|=\tilde\Exp\, \|\tilde W\| 
\le  \sqrt{2k_0} + \sqrt{r-k_0} \le \sqrt{2n}.
\end{align*}
Since the spectral norm is $1$--Lipschitz,
the standard Gaussian concentration inequality (see, for example,
\cite[Section~5.2]{Vershynin}) implies
\begin{align} \label{eq: normW}
    \tilde\Prob \big\{ \|W\| \ge 3\sqrt{\beta n} \big\} 
\le \tilde\Prob \big\{ \|W\| \ge \tilde\Exp\,\|W\| + \sqrt{\beta n } \big\}
\le 2\exp\Big( - \frac{ ( \sqrt{\beta n })^2}{2}\Big) 
= 2\exp(-\beta n/2).
\end{align}
Next, we derive an estimate for $s_{\min}(Q)=s_{\min}(\tilde Q)$. 
For $i \in [k_0]$,
let $P_i:\R^{k_0}\to\R^{k_0}$ be the orthogonal projection to the subspace which is orthogonal to the columns vectors $\tilde Q_{[2k_0],j}$ for $j \in [k_0]\backslash \{i\}$. Then, 
\begin{align*}
s_{\min}(Q) &= \min_{v \in S^{k_0-1}} \|\tilde Qv\| 
\ge \min_{v \in S^{k_0-1}} \max_{i\in [k_0]} \|P_i\tilde Qv\|_2
= \min_{v \in S^{k_0-1}} \max_{i\in [k_0]}  \|P_i(\tilde Q_{[2k_0],i})\|_2 |v_i|\\
&\geq \frac{\min_{j \in [k_0]}\|P_j(\tilde Q_{[2k_0],j})\|_2}{\sqrt{k_0}}.
\end{align*}

Since $P_j$ and $\tilde Q_{[2k_0],j}$ are independent, the norm
$\|P_j(\tilde Q_{[2k_0],j})\|_2$ is equidistributed
with that of a $2k_0-(k_0-1)=(k_0+1)$--dimensional standard Gaussian vector.
Since the probability density function of a $(k_0+1)$--dimensional Gaussian vector is bounded above by $(2\pi)^{-(k_0+1)/2}$, we obtain 
$$
\tilde\Prob \big\{ \| P_j (\tilde Q_{[2k_0],j}) \|_2 \le t \big\} \le   \Big(\frac{t}{\sqrt{2\pi}}\Big)^{k_0+1} |B_2^{k_0+1}|,\quad t>0,
$$
where $|B_2^{k_0+1}|$ is the Lebesgue measure of the unit Euclidean ball $B_2^{k_0+1}$ in $\R^{k_0+1}$. Therefore, in view of the previous computations,
\begin{align*}
    \tilde\Prob \big\{ s_{\min}(Q) \le t \big\} \le k_0\Big(\frac{t\sqrt{k_0}}{\sqrt{2\pi}}\Big)^{k_0+1} |B_2^{k_0+1}|,   \quad t>0.
\end{align*}
Applying the bound $(|B_2^{k_0+1}|)^{1/(k_0+1)}
= O( k_0^{-1/2})$, we get that
there exists a universal constant $C_{\rm b}\ge 1$ so that 
$$
\tilde\Prob \big\{ s_{\min}(Q) \le t \big\} \le (C_{\rm b} \,t)^{k_0},\quad t>0. 
$$
Now, setting $t := n^{-\beta/(40p)}$, we get
\begin{align} \label{eq: sminQ}
   \tilde\Prob \Big\{ s_{\min}(Q) \le  n^{-\beta/(40p)} \Big\} \le 
   n^{ - (1-o_n(1))\beta k_0 /(40p) }
   \le n^{-2\beta-1+o_n(1)},
\end{align}
where the last inequality holds since $k_0=k_0(p)\ge 120p$. 

\bigskip

As a final step of the proof, rewriting 
\eqref{eq: normW} and \eqref{eq: sminQ} on the entire probability space, we get
\begin{align*}
&\Prob \big\{ \|W\| \ge 3\sqrt{\beta n}\;\big|\;A_{[n],[r]} 
\big\} \leq 2\exp(-\beta n/2)\quad a.s;\\
&\Prob \Big\{ s_{\min}(Q) \le  n^{-\beta/(40p)}\;\big|\;A_{[n],[r]} \Big\} 
   \le n^{-2\beta-1+o_n(1)}\quad a.s.
\end{align*}
We conclude that $ \Prob( \Event_r(\beta)^c ) \le n^{-2\beta -1 +o_n(1)}$, and the result follows.
\end{proof} 

\section{The smallest singular value and the growth factor in exact arithmetic}

\subsection{Distance to subspaces}

Recall that by $i_t=i_t(A)$, $1\leq t\leq n$, we denote
the indices of the pivot rows in the GEPP process (see Section~\ref{fepijnfeofiniqwjnpij}).

\begin{defi}[Subspaces generated by row vectors of submatrices of $A$] 
For $x,r \in [n]$ with $1\leq x \le r$, let 
$$
H_{r,x} \subset \R^r \mbox{ be the random subspace spanned by } (A_{ i_t, [r] })^\top
\mbox{ for } t\in[x],
$$
and let $H_{r,0}:=\{0\}$.
Additionally, for $s \in [x]$, let 
$$ H_{r,x,s} \subset \R^r
\mbox{ be the random subspace spanned by }
  (A_{ i_t, [r] })^\top \mbox{ for }  t \in [x]\backslash \{s\},
$$
where we set $H_{r,1,1}:=\{0\}$.
\end{defi}

\begin{defi}
For $\beta >0$, let $\Event_{\row}(r,\beta)$ be the event that 
\begin{align} \label{eq: defEventRow}
	\dist\big( (A_{ i_r, [r]})^\top,\, H_{r,r-1} \big)  \ge \sqrt{2/ \pi }\, n^{- 4(1+\beta / (n-r)) }\;\;
	\mbox{ and }\;\; \| A_{ i_r, [n]} \|_2 \le  \sqrt{n}+ 3 \sqrt{ \beta \log(n)}
\end{align}
and set
$$ \Event_{\rm row}(\beta) := \bigcap_{r \in [n-1 ]} \Event_{\row}(r,\beta).$$
Further, let $\Event_{\dist}(\beta)$ be the event that 
\begin{align*}
		\forall\; r,k,s \in [n-1]  &\mbox{ with } s \le r-k \le r,\\
			&\dist \big(  (A_{i_s, [r]})^\top , H_{r,r-k,s} \big) 
			\le  \exp\Big( 6k \Big( 1 + \frac{\beta}{n-r}\Big)\log n\Big)
			\dist\big( (A_{i_s, [r]})^\top, H_{r,r,s}\big) .
	\end{align*}
\end{defi}

The goal in this section is to prove
\begin{prop} \label{prop: sec6Main}
There exists $\beta_{\text{\tiny\ref{prop: sec6Main}}} \ge 2$
so that the following holds. For $\beta \ge \beta_{\text{\tiny\ref{prop: sec6Main}}}$, we have
$\Event_{\dist}(\beta) \supset \Event_{\row}(\beta)$, and 
	\begin{align*}
		\Prob ( \Event_{\dist}(\beta)^c ) \le  
		\Prob ( \Event_{\row}(\beta)^c ) \le
		n^{-\beta}.
	\end{align*}
\end{prop}

The statement is obtained as a combination of Lemmas~\ref{fpnuafpiufnpifqnpi}
and~\ref{apijnfpoeiufnofiunfpi} below.
First, we consider two simple facts from Euclidean geometry.

\begin{lemma}\label{lem: projDist}
Let $u \in \R^r$ and let
$H \subset \R^r$ be a subspace. Then for any orthogonal projection $P$ in $\R^r$, we have 
$$
\dist( u,H) \ge \dist(Pu,PH). 
$$
\end{lemma}
\begin{proof}
The statement follows immediately by observing that $P$ is a contraction.
\end{proof}

\begin{lemma}\label{lem: distInductive}
Let $F$ be a subspace of $\R^k$, and let
$v_1, v_2 \in \R^k$ be vectors such that
$$\dim\spn(F,v_1,v_2)=\dim(F)+2.$$
For $i \in [2]$, let $F_i$ be the linear span of $F$ and $v_i$. Then, 
	$$
		{\rm dist}(v_1, F) \le \frac{{\rm dist}(v_1,F_2)\,\|v_2\|_2}{{\rm dist}(v_2,F_1) }.
	$$
\end{lemma}
\begin{proof}
For any subspace $E$, we let $P_E$ be the orthogonal projection onto $E$.
Let $u_i := \frac{P_{F^\perp}v_i }{ \| P_{F^\perp}v_i \|_2}$. Observe that, 
	$$
		{\rm dist}(v_2,F_1) = \|P_{F_1^\perp}v_2\|_2 = 
		\|  P_{F^\perp}v_2 - \langle P_{F^\perp}v_2, u_1 \rangle u_1 \|_2
		=\|  P_{F^\perp}v_2\|_2\, \|u_2 - \langle u_2,u_1\rangle u_1 \|_2,
	$$
whence
$$
\|u_2 - \langle u_2,u_1\rangle u_1 \|_2 = \frac{ {\rm dist}(v_2,F_1) }{ \|P_{F^\perp}v_2\|} \ge \frac{ \dist(v_2,F_1) }{ \|v_2\|_2}.
$$
On the other hand, 
$$
		\|u_2 - \langle u_2,u_1\rangle u_1\| _2
		= \sqrt{ 1 -  \langle u_2,u_1\rangle^2 }  
		= \|u_1 - \langle u_1,u_2 \rangle u_2\|_2,
$$
and therefore
	$$
		\dist(v_1,F) = \| P_{F^\perp} v_1\|_2
		= \frac{ \| P_{F^\perp} v_1\|_2\, \|u_1 - \langle u_1,u_2 \rangle u_2\|_2 }{
		 \|u_1 - \langle u_1,u_2 \rangle u_2\|_2} =  
		 \frac{\dist(v_1,F_2)}{\|u_1 - \langle u_1,u_2 \rangle u_2\|_2}
		 \le  \frac{\dist(v_1,F_2)\,\|v_2\|_2}{\dist(v_2,F_1) }.
	$$
\end{proof}

\begin{lemma}\label{fpnuafpiufnpifqnpi}
Let $s,k,r \in [n-1]$ such that $s\le r-k < r$. Fix a realization of $A$ such that
the event $\Event_{\row}(\beta)$ holds. Then, 
	\begin{align*}
	\dist \big(  (A_{ i_s, [r] })^\top, {H}_{r,r-k,s}  \big)  
	\le \exp\Big(  6k\Big( 1 + \frac{\beta  }{n-r} \Big)\log(n) \Big) \,
	\dist \big(   (A_{ i_s, [r] })^\top, {H}_{r,r,s}  \big)  .
	\end{align*}
	Thus, $\Event_{\dist}(\beta) \supset \Event_{\row}(\beta)$.
\end{lemma}
\begin{proof}
First, we note that for every $t \in [2,r]$, 
$$
	H_{t,t-1} = P_t (H_{r,t-1})
$$
where $P_t: \R^r \mapsto \R^t$ is the coordinate
projection onto the first $t$ components. Applying Lemma~\ref{lem: projDist} for
every $2\leq t\leq r$, we obtain 		
\begin{align*}
		\dist\big( (A_{i_t , [r]})^\top,\, H_{r,t-1}   \big) \ge 
		\dist\big(  (A_{i_t , [t]})^\top,\, H_{t,t-1} \big) 
		\ge  \sqrt{2/\pi }\, n^{-4(1+ \beta/ (n-t))},
	\end{align*}
where in the last inequality we used the definition of $\Event_{\row}(\beta)$.
Further, for $t \in [r-k+1,r]$, we will apply Lemma~\ref{lem: distInductive} with 
$$
	F := {H}_{r,t-1,s}, \quad v_1 := (A_{i_s, [r] })^\top ,\quad \mbox{and } 
	v_2 :=  (A_{ i_t, [r] })^\top,
$$
so that $F_1 = H_{r,t-1}$ and $F_2 = {H}_{r,t,s}$, and from the previous inequality
and the definition of $\Event_{\row}(\beta)$ we have
$$
		\dist(v_2,\, F_1)  
		\ge \sqrt{2/\pi }\, n^{-4(1+ \beta/ (n-r))}\;
		\mbox{ and }\; \|v_2\|_2 \le \|A_{ i_t, [n] } \|_2 
		\le \sqrt{n} + 3 \sqrt{ \beta \log n }.
$$
Lemma~\ref{lem: distInductive} implies
\begin{align*}
	{\dist}(  (A_{ i_s, [r]})^\top, {H}_{r,t-1,s}  ) \le 
		\sqrt{\pi /2 }\, n^{4 (1 + \beta / (n-t))} ( \sqrt{n} + 3 \sqrt{ \beta \log n } )\,
		\dist(  (A_{ i_s, [r] })^\top, {H}_{r,t,s} )
	\end{align*}
(it is easy to see that
in the case $\dim\spn\{F,v_1,v_2\}<\dim(F)+2$ when the lemma cannot be applied, the
above inequality holds as well).
Together with the inequality $ \sqrt{n} + 3 \sqrt{ \beta \log n } \le 2n^{1+ \beta / (n-t) }$ for $\beta >0$,  
\begin{align*}
	{\dist}\big( {H}_{r,t-1,s} ,\, (A_{ i_s, [r]})^\top \big) \le 
		\exp \bigg( \Big( 1 + \frac{\beta}{n-t} \Big) 6\log n \bigg)
		{\dist}\big(  {H}_{r,t,s},\, (A_{ i_s, [r]})^\top \big),\quad
		t \in [r-k+1,r].
	\end{align*}
Finally, applying the above inequality inductively for $t$ from $r-k+1$ to $r$ we obtain  
\begin{align*}
		{\dist}\big( {H}_{r,r-k,s},\, (A_{ i_s, [r]})^\top  \big) \le 
		\exp\bigg( 6k \Big( 1+ \frac{\beta}{n-r} \Big) \log n \bigg)\,
		{\dist}\big(  {H}_{r,r,s},\, (A_{ i_s, [r]})^\top \big).  
	\end{align*}
\end{proof}

\begin{lemma}\label{apijnfpoeiufnofiunfpi}
For $\beta \ge 2$ and $1\leq r\leq n$, the following probability estimate holds:
$$
	\Prob ( \Event_{\row}(r,\beta)^c ) \le (1+o_n(1))n^{-2\beta}
$$
and 
$$
	\Prob ( \Event_{\row}(\beta)^c ) \le n^{-\beta}.
$$
\end{lemma}
\begin{proof}
First, in view of Corollary \ref{cor: polyDist}, we have 
\begin{align*}
	\Prob \Big\{ \dist\big((A_{ i_r, [r]})^\top,\, H_{r,r-1} \big)  \le  
	\sqrt{ \frac{\pi}{2} }\, n^{-4(1+  \beta/ (n-r))}\Big\}
	\leq n^{-2(1+  \beta/ (n-r))(n-r)}
	= n^{-2(n-r+\beta)} \leq n^{-2\beta}. 
\end{align*}
Next, for each $i \in [n]$, applying the standard concentration inequality for Lipschitz functions of Gaussian variables,  
$$
\Prob \big \{ \|A_{ i,[n] } \|_2 \ge \Exp\,\|A_{i,[n]}\|_2 + t \big\} \le 2\exp( -t^2/2),
\quad t>0.
$$
With $\Exp\|A_{i,[n]}\|_2 \le (\Exp\|A_{i,[n]}\|_2^2)^{1/2}\le \sqrt{n}$, by taking
$t := 3\sqrt{ \beta \log n }$ we have 
$$
	\Prob \big \{ \|A_{ i, [n] } \|_2 \ge \sqrt{n} + 3\sqrt{ \beta \log n } \big\} 
	\le  2\, n^{-9\beta/2}. 
$$
Taking the union bound over $i\in[n]$ and taking into account the condition $\beta\geq 2$,
we get the first assertion of the lemma.

The second assertion follows from another application of 
the union bound. 
\end{proof}

\subsection{The smallest singular value of $A_{I_r,[r]}$}

\begin{defi}
For $k \in [n]$, $\beta,p \ge 1$, let $\Event_{\rm sq}(p,k,\beta)$ be the event that 
$$
\forall r \in [k,n-k],\quad s_{r}(A_{I_r,[r]}) \ge n^{-\beta/(6p)}. 
$$
\end{defi}

\begin{prop}\label{smin prop efwf}
There is a universal constant $C>0$ with the following property.
For any $p \ge 1$, there exist $n_0(p)$, $k_1(p)\leq Cp^2$ and $
\beta_{\text{\tiny\ref{prop: sec6Main}}}\leq
\beta_1(p)\leq Cp^2$
such that for $ n \ge n_0(p), \beta \ge \beta_1(p)$, and $k_1(p)$ we have 
$$
\Event_{\rm sq}(p,k_1(p),\beta) \supset \Event_{\row}(\beta) \cap 
 \Event_{\rm rec}(p,k_0(p),\beta),
$$
where $k_0(p)$ is
taken from Proposition~\ref{prop: sec5Main}, $\beta_{\text{\tiny\ref{prop: sec6Main}}}$ is defined in Proposition~\ref{prop: sec6Main}, and $\Event_{\rm rec}(\cdot)$
is taken from Definition~\ref{def: eventIntSing}.
Moreover,
$$
\Prob( \Event_{\rm sq}(p,k_1(p),\beta)^c ) \le 2n^{-\beta}.
$$
\end{prop}
\begin{proof}
Note that if the events' inclusion above holds then the
second assertion of the proposition follows immediately by
combining the bounds
$\Prob( \Event_{\rm rec}(p,k_0(p),\beta)^c ) \le n^{-2\beta +o_n(1)}$ from Corollary \ref{cor: sminTall} and 
$\Prob( \Event_{\rm row}(\beta)^c)\le n^{-\beta}$ from Lemma~\ref{apijnfpoeiufnofiunfpi}.
Thus, we can focus on proving the first assertion.

\medskip

Let $k_1(p)=\beta_1(p) \ge 400 k_0(p)\,p$ where $k_0(p)$ is
taken from Proposition~\ref{prop: sec5Main}.

Consider an argument by contradiction.
Fix any realization of $A$ such that both $\Event_{\row}(\beta)$
and $\Event_{\rm rec}(p, k_0(p),\beta)$ hold,
and such that for some
$r \in [k_1(p), n-k_1(p)]$, $s_{r}(A_{I_r,[r]}) \le n^{-\beta/(6p)}$, that is, there exists a unit vector $u \in \R^{I_r}$ such that 
$$
	\|(A_{I_r,[r]})^\top u \|_2 \le n^{-\beta/(6p)}
$$
(we assume here that the column of the matrix $(A_{I_r,[r]})^\top$ are indexed over the set $I_r$).
Since $\|u\|_2=1$, there is an index
$s \in [r]$ such that $ |u_{i_s}|\ge r^{-1/2} \ge n^{-1/2} $, whence
$$
\|(A_{I_r,[r]})^\top u \|_2 = \Big\| \sum_{t \in [r]} (A_{\{i_t\},[r]})^\top u_{i_t} \Big\|_2
\ge n^{-1/2}\, \dist\big( (A_{i_s,[r]})^\top ,\, H_{r,r,s} \big).
$$
Thus, our realization of $A$ and our choice of $s$ satisfy
\begin{align}\label{fpfoinqpfiq3unfpq3ifnq3p}
\dist\big( (A_{i_s,[r]})^\top , H_{r,r,s} \big) \le \exp( -\beta \log (n) /(6p) + \log(n) /2).
\end{align}
Set $k := \min\{ 2\,k_0(p), r-s\}$.

Assume first that $s<r$, i.e $k>0$. In view of the inclusion $ \Event_{\row}(\beta) \subset \Event_{\dist}(\beta)$
(see Proposition~\ref{prop: sec6Main}), we get
\begin{align*}
	\dist( (A_{i_s,[r]})^\top , H_{r,r-k,s} ) \le &
	\exp\Big( -\beta\log n/(6p) + \log(n)/2 + 6k \Big ( 1 + \frac{\beta}{n-r} \Big) \log n \Big).
\end{align*}
Since $n-r \ge k_1(p)$ and  $\beta \ge \beta_1(p)= k_1(p) \ge 400 k_0(p)\,p \ge 200kp$, 
we have
\begin{align*}
   \frac{1}{2}+ 6k\Big( 1 + \frac{\beta}{n-r}\Big)
 \le 7k\Big( 1 + \frac{\beta}{n-r}\Big)
 \le 7 \cdot \frac{k_1(p)}{200 p} \Big( 1 + \frac{\beta}{n-r}\Big)
 \le 7\cdot \frac{k_1(p)}{200 p} + 7\cdot \frac{\beta}{200p} 
 \le \frac{\beta}{12p},
\end{align*}
whence
\begin{align} \label{eq: secDist00}
	\dist \big( (A_{i_s,[r]})^\top, H_{r,r-k,s} \big) 
\le   n^{-\beta/(12p)}.      
\end{align}

Further, in the situation when $k=0$, the inequality \eqref{eq: secDist00} is still true as can be
immediately seen from \eqref{fpfoinqpfiq3unfpq3ifnq3p}.

\medskip

Next, we will show that \eqref{eq: secDist00} leads to contradiction.
The argument depends on whether $k=r-s$ or not.

{\bf Case 1: $k=r-s$.} 
By the definition of the event $\Event_{\row}(\beta)$, we have 
\begin{align*}
	\dist\big( (A_{i_s,[r]})^\top , \, H_{r,s,s} \big)
\ge 
	\dist \big( (A_{i_s,[s]})^\top, H_{s,s,s} \big)
	=
	\dist \big( (A_{i_s,[s]})^\top, H_{s,s-1} \big)
	&\ge \sqrt{2/\pi}\, n^{-4(1+ \frac{\beta}{n-s})}\\
&\ge \sqrt{2/\pi}\,n^{-4-\beta/(50k_0(p)p)},
\end{align*} 
where we used that
$n-s \ge n-r \ge k_1(p)\ge 200k_0(p)\,p$.
In view of the condition $\beta \ge 200k_0(p)\,p\ge 400p$,
$$
\sqrt{2/\pi}\,n^{-4-\beta/(50k_0(p)p)} \ge n^{-5-\beta/(50p)}
\ge n^{ - \beta/(25p) } > n^{- \beta/(12p)},
$$
which contradicts \eqref{eq: secDist00}. 

{\bf Case 2: $k=2k_0(p)<r-s$.} 
In this case,  $(A_{i_s,[r]})^\top$ is a column vector of $(A_{I_{r-k}, [r]})^\top$ and $H_{r,r-k,s}$ is the span of every other column vector $(A_{i_{s'},[r]})^\top$ for $s' \in [r-k]\backslash \{s\}$.
Hence, in view of \eqref{eq: secDist00},
$$
s_{\min}\big((A_{I_{r-k}, [r]})^\top\big)\leq
\dist \big( (A_{i_s,[r]})^\top, H_{r,r-k,s} \big)
    \le n^{-\beta/12p}.
$$
However, this contradicts the definition of the
event $\Event_{\rm sq}(p, k_0(p),\beta)$:
$$
\forall r' \in [k_0(p)+1,n-2k_0], \, s_{\min}\big( (A_{I_{r'},[r'+2k_0]})^\top \big) \ge n^{-\beta/(20p)}. 
$$
The result follows.
\end{proof}

The next simple lemma will be used to show that with high probability
indices of the pivot rows obtained in exact arithmetic coincide
with the results of the floating point computations.
\begin{lemma}\label{fakjnfapifjnewfpiqjwnfpijn}
There is a universal constant $C>0$ and a number $n_0\in\N$
such that, assuming $n\geq n_0$,
$$
\Prob\big\{
s_{\min}(A_{I_r,[r]})\leq t\,n^{-C}\mbox{ for some $1\leq r\leq n-1$}
\big\}\leq t,\quad t>0.
$$
\end{lemma}
\begin{proof}
In view of Proposition~\ref{smin prop efwf} (say, applied with $p=1$),
there are constants $C_1,n_0>0$ such that, assuming $n\geq n_0$,
$$
\Prob\big\{
s_{\min}(A_{I_r,[r]})\leq t\,n^{-C}\mbox{ for some $k_1(1)\leq r\leq n-k_1(1)$}
\big\}\leq t,\quad t>0.
$$
For indices $r<k_1(1)$, we use the trivial union bound:
\begin{align*}
\Prob&\big\{
s_{\min}(A_{I_r,[r]})\leq t\;\mbox{ for some $1\leq r<k_1(1)$}\big\}\\
&\leq \sum_{I\subset [n],\,1\leq|I|<k_1(1)}
\Prob\big\{
s_{\min}(A_{I,[|I|]})\leq t\big\}
\leq n^{k_1(1)}\,t,\quad t>0,
\end{align*}
where in the last line we used the standard bound on the smallest singular
value of a square Gaussian random matrix \cite{Edelman,Szarek}.
Similarly, we get
$$
\Prob\big\{
s_{\min}(A_{I_r,[r]})\leq t\;\mbox{ for some $n-k_1(1)<r\leq n-1$}\big\}
\leq n^{k_1(1)}\,t,\quad t>0.
$$
Combining the three estimates above, we get the result.
\end{proof}

\subsection{Estimating the growth factor in exact arithmetic}

\begin{defi}
For $\beta >1$, let $\Event_{\rm col}(\beta)$ be the event that 
\begin{align} \label{eq: defEventCol}
	\forall j\in[n],\, 
	\| A_{ [n],j} \|_2 \le  \sqrt{n}+ 3 \sqrt{ \beta \log(n)}
\end{align}
and for $\tau>1$, let $\Event_{\rm entry}(\tau)$ be the event that 
\begin{align} \label{eq: defEventEntry}
	\max_{i,j\in[n]}
	| A_{ i,j} | \ge  n^{-\tau}.
\end{align}
\end{defi}
\begin{lemma}\label{efojnefpifnqpifjnpiwjn}
For any $\beta \ge 2$, we have
$$
\Prob( \Event_{\rm col}(\beta)^c ) \le n^{-\beta};
$$
furthermore, for every $\tau\geq 1$,
$$
\Prob( \Event_{\rm entry}(\tau)^c ) \le n^{-\tau\,n^2}.
$$
\end{lemma}
\begin{proof}
The upper bound on $\Prob( \Event_{\rm col}(\beta)^c)$
can be derived exactly the same way as in the argument for $\Event_{\rm row}(\beta)$
(see the proof of Lemma~\ref{apijnfpoeiufnofiunfpi}), so we skip the discussion.

To estimate the complement of $\Event_{\rm entry}(\tau)$, we write
\begin{align*}
    \Prob(\Event_{\rm entry}(\tau)^c) \le
    \Prob \big\{ |A_{i,j}| < n^{-\tau}\mbox{ for all $i,j$}\big\}
    \leq n^{-\tau\,n^2},
\end{align*}
where in the last inequality we used that
the probability density function of the standard Gaussian random variable is bounded by $\frac{1}{\sqrt{2\pi}}$.
\end{proof}

At this point, we are ready to prove the
``exact arithmetic'' counterpart of the
main statement of the paper:

\begin{prop}\label{npafjnepfijnfpijqnpijn}
There is a universal constant $C>1$ and a function
$\tilde n:[1,\infty)\to\N$ with the following property. 
Let $p\geq 1$, and let $n\geq \tilde n(p)$.
Then
$$
\Prob\bigg\{\frac{\max_{i,j,\ell}|A^{(\ell)}_{i,j}|}{\max_{i,j}|A_{i,j}|}
\geq n^{t}\bigg\} \leq 5n^{-p t},\quad t\geq Cp^2.
$$
\end{prop}
\begin{proof}
Recall that the parameter $k_1(p)=O(p^2)$ was defined in Proposition~\ref{smin prop efwf}.
We can take a universal constant $C_1>0$ large enough so that
$C_1p^3 \ge 600 p\,k_1(p)$ for all $p \ge 1$. Fix $\beta \ge C_1p^3$, set
$\tau:=\beta/(100p)$, and assume $n\geq \sqrt{100p}$.
In view of the assertions of Lemma~\ref{apijnfpoeiufnofiunfpi}, Proposition~\ref{smin prop efwf},
and Lemma~\ref{efojnefpifnqpifjnpiwjn}, in order to show that
$$
\Prob\bigg\{\frac{\max_{i,j,\ell}|A^{(\ell)}_{i,j}|}{\max_{i,j}|A_{i,j}|}
> n^{\beta/(3p)} \bigg\} \le 5n^{-\beta}
$$
(which would imply the statement), it is sufficient to verify that everywhere on the intersection
$$
\Event_{\rm entry}(\tau)\cap \Event_{\rm row}(\beta) \cap \Event_{\rm col}(\beta) 
\cap \Event_{\rm sq}(p,k_1(p),\beta),
$$
we have
$$
\frac{\max_{i,j,\ell}|A^{(\ell)}_{i,j}|}{\max_{i,j}|A_{i,j}|} \le n^{\beta/(3p)}. 
$$
In what follows, we use the notation introduced at the beginning of Section~\ref{peofjnwqepifunfpijn}; in particular, we work with matrices
$\prA^{(\ell)}$, $0\leq \ell\leq n-1$, defined in
\eqref{nalfkjnfo34iuhoqfjj}.

Recall that
\begin{align} \label{eq: gAtMost2}
\forall r \in [n],\, \quad\frac{\max_{i,j} |\prA_{i,j}^{(r)}| }{ \max_{i,j}|\prA^{(r-1)}_{i,j}|} \le 2.
\end{align}
For $r \in [k_1(p)]$,  we simply use the bound above to get
$$
\frac{\max_{i,j} |\prA_{i,j}^{(r)}| }{ \max_{i,j}|A_{i,j}|} \le 2^{k_1(p)},
\quad r \in [k_1(p)].
$$
Further, for $r \in (k_1(p),n-k_1(p)]$, we write
\begin{align*}
	\frac{ \max_{i,j}|\prA^{(r)}_{i,j}|}{ \max_{i,j}|A_{i,j}|} 
= \max\bigg(
\frac{ \max\limits_{s\in [r],j\geq s}|\prA^{(s-1)}_{i_s,j}|}{ \max_{i,j}|A_{i,j}|} ,\, 
	\frac{\max\limits_{i \in [n]\backslash I_{r},j>r}|\prA^{(r)}_{i,j}|}{\max_{i,j}|A_{i,j}|}
\bigg).
\end{align*}
In view of formula \eqref{kejfnpfiunf9un} and our conditioning on the
event $\Event_{\rm sq}(p,k_1(p),\beta)$,
for $s\in [k_1(p),n-k_1(p)]$, $i \in [n]\backslash I_s$ and $j >s$, we have
\begin{align*}
	|\prA^{(s)}_{i,j}| &\le |A_{i,j}| 
	+ \|A_{i,[s]}\|_2\, \frac{1}{s_{\min}(A_{I_s,[s]})}\,
\|A_{[s],j}\|_2\\
&\le  \sqrt{n}+3\sqrt{\beta\log n}
+ (\sqrt{n}+3\sqrt{\beta\log n}) n^{\beta/(6p)} (\sqrt{n}+3\sqrt{\beta\log n}) \\
&<  2n^{\beta/(6p)}(\sqrt{n}+3\sqrt{\beta\log n})^2,
\end{align*}
and thus 
\begin{align*}
\frac{\max\limits_{i \in [n]\backslash I_{r},j>r}|\prA^{(r)}_{i,j}|}{\max_{i,j}|A_{i,j}|}
\le 2n^{\beta/(6p)+ \beta/(100p)}(\sqrt{n}+3\sqrt{\beta\log n})^2,
\end{align*}
and for every $s\in[r]$ with $s>k_1(p)$,
$$
\frac{\max\limits_{j\geq s}|\prA^{(s-1)}_{i_s,j}|}{ \max_{i,j}|A_{i,j}|}
\leq 2n^{\beta/(6p)+ \beta/(100p)}(\sqrt{n}+3\sqrt{\beta\log n})^2.
$$
By our earlier observation,
$$
\frac{\max\limits_{s\in [k_1(p)],j\geq s}|\prA^{(s-1)}_{i_s,j}|}{ \max_{i,j}|A_{i,j}|}
\leq 2^{k_1(p)}.
$$
Combining the estimates together, we conclude that for all $r \in [n-k_1(p)]$,
\begin{align*}
	\frac{ \max_{i,j}|\prA^{(r)}_{i,j}|}{ \max_{i,j}|A_{i,j}|} 
\le
	\max \big( 2^{k_1(p)},\, 
	2n^{\beta/(6p)+ \beta/(100p)}(\sqrt{n}+3\sqrt{\beta\log n})^2
	\big).
\end{align*}
For the ``last'' $k_1(p)$ admissible values of $r$, we rely on \eqref{eq: gAtMost2} again to get 
\begin{align*}
  \forall r \in (n-k_1(p),n],\, \quad
	\frac{\max_{i,j}|\prA_{i,j}^{(r)}|}{\max_{i,j}|A_{i,j}|}
\le \frac{\max_{i,j}|\prA_{i,j}^{(n-k_1(p))}|}{\max_{i,j}|A_{i,j}|}\,2^{k_1(p)}.
\end{align*}
In the end, we make use of our bound $\beta/(6p) \ge 100 k_1(p)$ to conclude that
for all large $n$,  
$$
\frac{\max_{i,j,\ell}|A^{(\ell)}_{i,j}|}{\max_{i,j}|A_{i,j}|}
=\frac{ \max_{r,i,j}|\prA^{(r)}_{i,j}|}{ \max_{i,j}|A_{i,j}|} 
	\le n^{\frac{\beta}{3p}}.
$$
This completes the proof.
\end{proof}

\section{GEPP in floating point arithmetic}\label{faonfpafjnpifjnpijn}

In this section we transfer the statement of Proposition~\ref{npafjnepfijnfpijqnpijn}
into the proper context of the floating point arithmetic.
We expect a part of the argument in this section (specifically, in the proof of 
Lemma~\ref{alfkjnapfjnpijnpijqn}) to be rather standard for experts
in numerical analysis. Still, we prefer to provide all the details
to make the paper self-contained.

\begin{lemma}\label{akdjfnapfjnwpfijfqpijn}
Let $A$ be an $n\times n$ Gaussian matrix and 
$A=\prA^{(0)},\prA^{(1)},\dots, \prA^{(n-1)}$ be the sequence of matrices generated by the GEPP in exact arithmetic (see \eqref{nalfkjnfo34iuhoqfjj}). Then, for every $1\leq k\leq n-1$,
\begin{align*}
\forall \delta \ge0, \quad
\Prob\big\{ (1-\delta)|\prA^{(k-1)}_{i_{k},k}| < \max_{i \notin I_{k}} 
|\prA^{(k-1)}_{i,k}| \big\}  \le \delta (n-k+1).
\end{align*}
\end{lemma}
\begin{proof}
Fix any $1\leq k\in n-1$.
With the vector $v_{k}(A)$ defined at the beginning of Section~\ref{fepijnfeofiniqwjnpij}
and in view of \eqref{kejfnpfiunf9un}, for every 
$i \notin I_{k-1}(A)$ we have 
\begin{align*}
|\prA^{(k-1)}_{i,k}| = |\langle v_{k}(A), (\prA_{i,[n]})^\top \rangle|.
\end{align*}

Fix any subset $I\subset[n]$ of cardinality $k-1$ and any $(k-1)
\times k$ matrix $M$,
and condition on the realizations $I_{k-1}(A)=I$ and $A_{I_{k-1},[k]}=M$.
In what follows, we denote the conditional probability measure by $\tilde\Prob$.
Under this conditioning, $v_{k}(A)$ and the polytope $K:=K_{k-1}(A)$ (see Section~\ref{fepijnfeofiniqwjnpij}; here we adopt the convention $K_0(A):=\R^n$)
are fixed. 
For $i \notin I$, let
$$
X_i := |\langle v_{k}(A), (A_{i,[n]})^{\top} \rangle |.
$$
By Lemma \ref{lem: secPolyBasic}, under the conditioning the vectors
$A_{i,[n]}$ for $i\notin I$ are i.i.d., with the probability density function 
$$
\rho(y)= {\bf 1}_{K}(y)\, \frac{\exp( - \|y\|_2^2/2) } {
\int_{K} \exp( - \|y'\|_2^2/2)\, {\rm d}y'},\quad y\in \R^n,
$$
which is symmetric and log-concave (i.e
$ y \mapsto \log(\rho(y))$ is a concave function). Since log-concavity is preserved under taking marginals, the random variable
$ \langle \frac{v_{k}(A)}{\|v_{k}(A)\|_2},(A_{i,[n]})^{\top}\rangle$ is also log-concave and symmetric under the conditioning.
This implies, in particular, that the probability density function $\rho_X$
of $X_i$'s ($i \notin I$) is non-increasing on the positive semi-axis. 

Now, since
$X_{i_{k}} = \max_{i \notin I}X_i$, we have
$$
\tilde\Prob\big\{ (1-\delta) X_{i_{k}} \ge \max_{i \notin I_{k}}X_i\big\}
=(n-k+1)\,\int\limits_0^\infty 
\Big(\int_0^{(1-\delta)r} \rho_X(t)\,dt \Big)^{n-k}\,\rho_X(r)\,dr,
$$
whereas 
$$
(n-k+1)\,\int\limits_0^\infty 
\Big(\int_0^{r} \rho_X(t)\,dt \Big)^{n-k}\,\rho_X(r)\,dr=1.
$$
Combining the two identities and using the monotonicity of $\rho_X$, we get
$$
\tilde\Prob\big\{ (1-\delta) X_{i_{k}} \ge \max_{i \notin I_{k}}X_i\big\}
=\Bigg(
\frac{\int\limits_0^\infty 
\Big(\int_0^{(1-\delta)r} \rho_X(t)\,dt \Big)^{n-k}\,\rho_X(r)\,dr}
{\int\limits_0^\infty 
\Big(\int_0^{r} \rho_X(t)\,dt \Big)^{n-k}\,\rho_X(r)\,dr}
\Bigg)
\geq (1-\delta)^{n-k}
\ge 1- \delta(n-k+1).
$$
The result follows by applying Fubini's theorem. 
\end{proof}

\begin{lemma}\label{alfkjnapfjnpijnpijqn}
Let $M$ an $n\times n$ invertible matrix and $\hat M:= {\rm fl}(M)$. 
Let $PM=LU$ be the $PLU$-factorization of $M$ in exact arithmetic,
assume that $P=\Id_n$, and let
$M=M^{(0)},M^{(1)},\dots,$ $M^{(n-1)}=U$ 
be the sequence of matrices obtained during the elimination process.
Let $\delta \in ({\bf u},1/3)$ be a parameter and assume that the matrix $M$ and the unit roundoff ${\bf u}$ satisfy
$$
8n^2\,{\bf u}\,\max\limits_{i,j,\ell}|M^{(\ell)}_{i,j}|
\leq \frac{ 1}{12}\,\frac{s_{\min}(M_{[k],[k]})^3}{\|M\|^2} \frac{\delta}{3}  ,\quad k=1,\dots,n-1,
$$
and 
\begin{equation}\label{apfjnapfiuwenfpewipfnq}
\forall k \in [n-1], \quad
	\frac{\max_{i\in [k+1,n]}|M^{(k-1)}_{i,k}|}{|M^{(k-1)}_{k,k}|} \leq 1-\delta. 
\end{equation}
Then GEPP in floating point arithmetic succeeds
for $\hat M$; the computed permutation matrix $\hat P=\Id_n$, and, denoting by
$\hat M = \hat M^{(0)}, \hat M^{(1)},\dots, \hat M^{(n-1)}$
the sequence of matrices obtained during the elimination process,
for every $k=0, 1,\dots,n-1$,
$$
\max\limits_{i,j}|\hat M^{(k)}_{i,j}|\leq 2\,\max\limits_{i,j,\ell}|M^{(\ell)}_{i,j}|.
$$
\end{lemma}
\begin{proof}
We will prove the statement by induction. Fix any $k\in[n-1]$.
Assume that all of the following holds
\begin{itemize}
\item[(a)] The computed matrix $\hat M^{(k-1)}$
has been produced by taking indices
of the first $k-1$ pivot rows to be $1,2,\dots,k-1$, and
$
|\hat M^{(k-1)}_{k,k}| >  \max_{i \in [k+1,n] }|\hat M^{(k-1)}_{i,k}|$,
so that the index of the $k$-th computed pivot row is $k$.
\item[(b)] $\hat M^{(k-1)} = G_{k-1}\cdots G_1(M+\tilde E^{(k-1)})$,
where $G_i$
is the Gauss transformation to eliminate $i$-th row of $\hat M^{(i-1)}$,
$1\leq i\leq k-1$, and where the error matrix $\tilde E^{(k-1)}$ satisfies
$$
\| \tilde E^{(k-1)} \| \le 8kn\,{\bf u}\, \max_{i,j,\ell}|M_{i,j}^{(\ell)}|.
$$
\item[(c)] $\max\limits_{i,j}|\hat M^{(v)}_{i,j}|\leq
2\,\max\limits_{i,j,\ell}|M^{(\ell)}_{i,j}|$ for all $0\leq v\leq k-1$.
\end{itemize}
Note that, by the assumptions on the matrix $M$, the induction hypothesis for 
the base case $k-1=0$
is satisfied.

Let $G_k=\Id_n-\tilde\tau^{(k)}\,e_k^\top$ be the Gauss transformation which eliminates entries $\hat M^{(k-1)}_{i,k}$, $i=k+1,\dots,n$,
so that in exact arithmetic we have
$$
\tilde \tau^{(k)}_{i}=
\frac{ \hat M^{(k-1)}_{i,k}}{ \hat M^{(k-1)}_{k,k}},\;\;
(G_k\hat M^{(k-1)})_{i,k}=0,\quad i=k+1,\dots,n.
$$
The computed matrix $\hat M^{(k)}$ can be explicitly written as
\begin{align*}
	\hat M^{(k)}_{i,j} & = \begin{cases}
	 0, & \mbox{if } 	i\in [k+1,n] \mbox{ and } j=k, \\
{\rm fl} 
 \Big(  \hat M^{(k-1)}_{i,j} - 
 {\rm fl}\Big({\rm fl}\Big(\frac{ \hat M^{(k-1)}_{i,k}}{ \hat M^{(k-1)}_{k,k}}\Big)\,
 \hat M^{(k-1)}_{k,j}\Big) \Big), &\mbox{if } i,j \in [k+1,n], \\
 \hat M^{(k-1)}_{i,j}, & \mbox{otherwise}
	\end{cases}
 \end{align*}
(note that
we ``force'' $\hat M^{(k)}_{i,k}$ to be $0$ for $i\in [k+1,n]$ whereas
the f.p.\ expression
${\rm fl}\big({\rm fl}\big(\frac{ \hat M^{(k-1)}_{i,k}}{ \hat M^{(k-1)}_{k,k}}\big)\,
 \hat M^{(k-1)}_{k,k}\big)$ is not necessarily equal to $\hat M^{(k-1)}_{i,k}$). 
Denote 
$$ 
E^{(k)} := \hat M^{(k)}- G_k\hat M^{(k-1)}.
$$
Since the first $k$ rows of $E^{(k)}$ are $0$, for every $i \in [k]$ we have 
$G_i E^{(k)}=E^{(k)}$, so $\hat M^{(k)}$ can be expressed in the form
$$
 \hat M^{(k)} = G_k \big(\hat M^{(k-1)} + G_{k-1}\cdots G_1 E^{(k)}\big). 
$$
Applying the above equality together with the induction hypothesis, we obtain 
\begin{align}
\hat M^{(k)} 
= G_k G_{k-1}\cdots G_1(M+\tilde E^{(k)}),
\end{align}
where $\tilde E^{(k)}:=\tilde E^{(k-1)}+E^{(k)}$. 
Note that non-zero entries of $E^{(k)}$
are all contained within the bottom right $(n-k)\times(n-k)$ submatrix of $E^{(k)}$,
and for every $i,j \in [k+1,n]$ we have
\begin{align*}
|E^{(k)}_{i,j}|
&= \bigg| {\rm fl} 
 \Big(  \hat M^{(k-1)}_{i,j} - 
 {\rm fl}\Big({\rm fl}\Big(\frac{ \hat M^{(k-1)}_{i,k}}{ \hat M^{(k-1)}_{k,k}}\Big)\,
 \hat M^{(k-1)}_{k,j}\Big) \Big)
-\bigg(\hat M^{(k-1)}_{i,j} - \frac{ \hat M^{(k-1)}_{i,k}}{ \hat M^{(k-1)}_{k,k}} \hat M^{(k-1)}_{k,j}  \bigg)\bigg|\\
&\le  3({\bf u} + O({\bf u}^2)) 
\max \bigg\{ 
|\hat M^{(k-1)}_{i,j}| ,\, \bigg|\frac{ \hat M^{(k-1)}_{i,k}}{ \hat M^{(k-1)}_{k,k}} \hat M^{(k-1)}_{k,j}\bigg|
\bigg\} 
\le 4{\bf u} \max\big\{ |\hat M^{(k-1)}_{i,j}|,
| \hat M^{(k-1)}_{k,j}|
\big\}  \\
&\le 4 {\bf u} \max_{i',j'} |\hat M_{i',j'}^{(k-1)}|
\le 8 {\bf u} \max_{i',j',\ell} |M_{i',j'}^{(\ell)}|,
\end{align*}
since there are $3$ floating point operations and $
\frac{ |\hat M^{(k-1)}_{i,k}|}{ |\hat M^{(k-1)}_{k,k}|}\le 1$,
and where in the last inequality we used the induction assumption (c).
Thus, 
$$
\| \tilde E^{(k)}\| \le  \|\tilde E^{(k-1)}\|+ 
\| E^{(k)}\| \le 
8kn {\bf u} \max_{i,j,\ell} |M_{i,j}^{(\ell)}| +
n \max_{i,j}|E^{(k)}_{i,j}|
\le 8(k+1)n {\bf u} \max_{i,j,\ell} |M_{i,j}^{(\ell)}|,
$$
confirming condition (b) on the $k$-th step. Moreover, in view of the assumptions
on $M$ we then have
\begin{equation}\label{fakejfnpweifjnwepfijwen}
\| \tilde E^{(k)}\| \le  8n^2 {\bf u} \max_{i,j,\ell} |M_{i,j}^{(\ell)}|
\leq \frac{ 1}{12}\,\frac{\min_{1\leq v\leq n-1}\,s_{\min}(M_{[v],[v]})^3}{\|M\|^2} \frac{\delta}{3}.
\end{equation}

Further, by the assumptions on $s_{\min}(M_{[k],[k]})$ and in view of 
the bound on the norm of $E^{(k)}$, the matrix $(M+\tilde E^{(k)})_{[k],[k]}$
is invertible. Hence, for every $i\in[k+1,n]$ there is a {\it unique}
linear combination $L_i$
of the first $k$ rows of $M+\tilde E^{(k)}$ such that the vector
$\row_i(M+\tilde E^{(k)})-L_i$ has first $k$ components equal zero.
We conclude that necessarily the matrices $G_1,G_2,\dots,G_k$
are Gauss transformations for $M+\tilde E^{(k)}$, whence 
for every $j \in [k+1,n]$ we have 
\begin{align}
\label{eq: hatM}	
\hat M^{(k)}_{j,[k+1,n]}
= (M+\tilde E^{(k)})_{j,[k+1,n]}  -
(M+\tilde E^{(k)})_{j,[k]} \big((M+\tilde E^{(k)})_{[k],[k]} \big)^{-1} (M+\tilde E^{(k)})_{[k], [k+1,n]},
\end{align}
whereas
\begin{align} \label{eq: M}
M^{(k)}_{j,[k+1,n]}
= M_{j,[k+1,n]}  -
M_{j,[k]} \big(M_{[k],[k]} \big)^{-1} M_{[k], [k+1,n]}.
\end{align}
We will rely on formulas
\eqref{eq: hatM} and \eqref{eq: M} to show that
$\hat M^{(k)}$ and $M^{(k)}$
are sufficiently close entry-wise. 

In view of \eqref{eq: hatM} and \eqref{eq: M}, for every $j\in [k+1,n]$ we have
\begin{align*}
\|\hat M^{(k)}_{j,[k+1,n]}-M^{(k)}_{j,[k+1,n]}\|\leq
\|\tilde E^{(k)}_{j,[k+1,n]}\|&+
2\|\tilde E^{(k)}\|\,\|M\|\,
\big\|\big((M+\tilde E^{(k)})_{[k],[k]} \big)^{-1}\big\|\\
&+\big\|\tilde E^{(k)}\|^2\,\|\big((M+\tilde E^{(k)})_{[k],[k]} \big)^{-1}\big\|\\
&+\|M\|^2\, \big\|\big((M+\tilde E^{(k)})_{[k],[k]} \big)^{-1}-\big(M_{[k],[k]} \big)^{-1}\big\|.
\end{align*}
Note that
$$
\big((M+\tilde E^{(k)})_{[k],[k]} \big)^{-1}-\big(M_{[k],[k]} \big)^{-1}
=-\big(M_{[k],[k]} \big)^{-1}\,\tilde E^{(k)}_{[k],[k]}\,\big((M+\tilde E^{(k)})_{[k],[k]} \big)^{-1},
$$
and that the bound $2\|\tilde E^{(k)}\|\leq s_{\min}(M_{[k],[k]})$ implies
$$
\|\big((M+\tilde E^{(k)})_{[k],[k]} \big)^{-1}\big\|
\leq 2\big\|(M_{[k],[k]})^{-1}\big\|.
$$
Thus, applying \eqref{fakejfnpweifjnwepfijwen}, for every $i=k+1,\dots,n$ we get
\begin{align}
\|&\hat M^{(k)}_{i,[k+1,n]}-M^{(k)}_{i,[k+1,n]}\|\nonumber\\
&\leq
\|\tilde E^{(k)}_{i,[k+1,n]}\|+
6\|\tilde E^{(k)}\|\,\|M\|\,\big\|(M_{[k],[k]})^{-1}\big\|
+2\|M\|^2\, \big\|\big(M_{[k],[k]} \big)^{-1}\big\|^2\,
\|\tilde E^{(k)}\|\nonumber\\ 
&\leq 
\frac{ 1}{12}\,\frac{\min_{1\leq v\leq n-1}
s_{\min}(M_{[v],[v]})^3}{\|M\|^2} \frac{\delta}{3}
\bigg(
1+\frac{6\|M\|}{s_{\min}(M_{[k],[k]})}
+\frac{2\|M\|^2}{s_{\min}(M_{[k],[k]})^2}
\bigg)\nonumber\\
&\leq (\delta/3)\,  
\min_{1\leq v\leq n-1} s_{\min}(M_{[v],[v]}).\label{apojnfpqurnp4iuqnpinpfq}
\end{align}
Immediately we have 
$$
	\max_{i,j \in [k+1,n]} | \hat M^{(k)}_{i,j } | \le 
	\max_{i,j \in [k+1,n]} | M^{(k)}_{i,j}| + s_{\min}(M_{[k],[k]}) \le 2\max_{i,j,\ell}|M^{(\ell)}_{i,j}|.
$$
By the nature of the Gaussian Elimination process, the first $k$ rows
of $\hat M^{(k)}$ coincide with those of $\hat M^{(k-1)}$,
and the bottom left $(n-k)\times k$ submatrix of $\hat M^{(k)}$ is zero. And thus
$$
\max_{i,j} |\hat M^{(k)}_{i,j}| \le 2 \max_{i,j,\ell}|M^{(\ell)}_{i,j}|,
$$
confirming the condition (c) for the $k$-th step.

It remains to check the condition (a). Note that we only need to consider the case
$k\leq n-2$.
Using the definition of vectors $v_k(\cdot)$ from the beginning of Section~\ref{fepijnfeofiniqwjnpij},
we can write
\begin{align*}
	|M^{(k)}_{k+1,k+1}| &=   M_{k+1,k+1} -
	M_{k+1,[k]} M^{-1}_{[k],[k]} M_{[k],k+1} \\
	&=  | \langle v_k(M), M_{k+1,[k+1]}^\top \rangle | \\
	&=  \|v_k(M)\|_2 \cdot {\rm dist}(H, M_{k+1,[k+1]}^\top)
	\ge 1 \cdot s_{\min}(M_{[k+1],[k+1]}),
\end{align*}
where $H \subset \mathbb{R}^{k+1}$ is the subspace spanned by the first $k$ rows of $M_{[k+1],[k+1]}$. Applying \eqref{apojnfpqurnp4iuqnpinpfq}
and \eqref{apfjnapfiuwenfpewipfnq}, we conclude that 
\begin{align*}
  |\hat M^{(k)}_{k+1,k+1}|  &\ge  (1-\delta/3 )|M^{(k)}_{k+1,k+1}| >  \max_{i\in[k+1,n]}|M^{(k)}_{i,k+1}| + (\delta / 3) 
|M^{(k)}_{k+1,k+1}| \\
&\ge  \max_{i\in[k+1,n]}|M^{(k)}_{i,k+1}| + (\delta/3) s_{\min}(M_{[k+1],[k+1]}) 
\ge  \max_{i \in [k+1,n]} |\hat M^{(k)}_{i,k+1}|,
\end{align*}
and the result follows.
\end{proof}

\begin{proof}[Proof of Theorem~\ref{thm: main}]
In view of Lemma~\ref{fakjnfapifjnewfpiqjwnfpijn}, there
are $C',n_0>0$
such that, assuming $n\geq n_0$,
$$
\Prob\big\{
s_{\min}(A_{I_r,[r]})\leq t\,n^{-C'}\mbox{ for some $1\leq r\leq n-1$}
\big\}\leq t,\quad t>0.
$$
On the other hand, standard concentration estimates for
the spectral norm of Gaussian matrices (see, for example,
\cite[Chapter~4]{Vershynin})
implies that, assuming $n$ is sufficiently large,
$$
\Prob\big\{\|A\|\geq C''\sqrt{n}\big\}\leq 2^{-n}.
$$
Let $\delta\in({\bf u},1/3)$ be a parameter to be chosen later, and
define the events
\begin{align*}
\Event_1(\delta)
&:=\bigg\{
8n^2\,{\bf u}^{1/2}
\leq \frac{ 1}{12}\,\frac{s_{\min}(A_{I_k,[k]})^3}{\|A\|^3} \frac{\delta}{3},
\quad k=1,\dots,n-1\bigg\},\\
\Event_2(\delta)
&:=\bigg\{
\frac{\max\limits_{i,j,\ell}|A^{(\ell)}_{i,j}|}{\|A\|}\leq {\bf u}^{-1/2}
\bigg\},\\
\Event_3(\delta)
&:=\bigg\{
\frac{\max_{i\notin I_k}|\prA^{(k-1)}_{i,k}|}
{|\prA^{(k-1)}_{i_k,k}|} \leq 1-\delta,\quad k=1,\dots,n-1
\bigg\}.
\end{align*}
The above observations on the spectral norm and the smallest singular
values imply
$$
\Prob(\Event_1(\delta))
\geq 1-2^{-n}-n^{C'}\cdot \big(7C'' n^{7/6}{\bf u}^{1/6}\delta^{-1/3}\big).
$$
Further, Proposition~\ref{npafjnepfijnfpijqnpijn} (applied, say, with
$p=2$)
yields for all sufficiently large $n$,
$$
\Prob(\Event_2(\delta))\geq 1-n^{C'''}{\bf u}
$$
for some universal constant $C'''>0$.
Finally, in view of Lemma~\ref{akdjfnapfjnwpfijfqpijn}, we have
$$
\Prob(\Event_3(\delta))\geq 1-\delta\,n^2.
$$
Thus, the intersection of the events
$\Event_1(\delta)\cap \Event_2(\delta)\cap \Event_3(\delta)$
has probability at least
$$
1-2^{-n}-n^{C'}\cdot \big(7C'' n^{7/6}{\bf u}^{1/6}\delta^{-1/3}\big)-n^{C'''}{\bf u}
-\delta\,n^2.
$$
Taking $\delta:={\bf u}^{1/8}$, we get that for any large enough $n$,
$$
\Prob\big(\Event_1({\bf u}^{1/8})\cap \Event_2({\bf u}^{1/8})
\cap \Event_3({\bf u}^{1/8})\big)\geq 1-n^{\tilde C}{\bf u}^{1/8},
$$
where $\tilde C>0$ is a universal constant.
It remains to note that, in view of Lemma~\ref{alfkjnapfjnpijnpijqn}, everywhere on
the intersection $\Event_1({\bf u}^{1/8})\cap \Event_2({\bf u}^{1/8})
\cap \Event_3({\bf u}^{1/8})$
the GEPP in floating point arithmetic succeeds
for $\hat A^{(0)}={\rm fl}(A)$; the computed permutation matrix $\hat P$ coincides
with the matrix $P$ from the $PLU$--factorization of $A$ in exact arithmetic, and
$$
{\bf g}_{\rm GEPP}(A)\leq \frac{4\,\max\limits_{i,j,\ell}|A^{(\ell)}_{i,j}|}
{\max\limits_{i,j}|A_{i,j}|}.
$$
A second application of Proposition~\ref{npafjnepfijnfpijqnpijn}, now to 
bound ${\bf g}_{\rm GEPP}(A)$ conditioned on the intersection
$\Event_1({\bf u}^{1/8})\cap \Event_2({\bf u}^{1/8})
\cap \Event_3({\bf u}^{1/8})$, completes the proof.
\end{proof}

\section{Further questions}\label{fepfojnfpwieufnqwpifjqn}

In this section, we mention some open questions related to
the probabilistic analysis of the Gaussian Elimination with Partial Pivoting.

\medskip

{\bf Sharp estimate of the growth factor.}
Our main result shows that with probability close to one, the growth factor
of GEPP is at most polynomial in the matrix dimension, ${\bf g}_{\rm GEPP}(A)
\leq n^C$.
Our analysis leaves the constant $C>0$ unspecified, and it would be of interest
to obtain an estimate with a reasonable (single digit) explicit constant.
Furthermore, as we mentioned in the introduction, it was suggested in \cite{EM}
based on numerical simulations that for large $n$, ${\bf g}_{\rm GEPP}(A)
=O(n^{1/2+o_n(1)})$ with probability close to one.
The problem of finding the optimal constant power of $n$
in the growth factor estimate seems to require essential new ideas.
At the same time, it is natural to expect that recurrent
estimates of the singular
spectrum of submatrices obtained in the GEPP process should
remain a key element of the future refinements of our result.

\medskip

{\bf The probability that the Gaussian Elimination with Partial Pivoting
succeeds in the floating point arithmetic.}
Our main result states that, under the assumption that the dimension $n$
is sufficiently large, with probability at least
$1-{\bf u}^{1/8}\,n^{\tilde C}$
the GEPP in f.p.\ arithmetic succeeds for ${\rm fl}(A)$,
and the computed permutation matrix agrees with that obtained in exact arithmetic.
We expect the probability estimate to be much stronger, perhaps of the form
$1-{\bf u}^{1-o_n(1)}\,n^{\tilde C}$, and leave this as an open problem.



\medskip

{\bf Smoothed analysis of the growth factor.}
Our proof does not extend to the {\it smoothed analysis} setting
without incurring significant losses in the upper estimate for the growth factor.
In fact, our treatment of the partially random block matrices $B$
in Section~\ref{s: poienfpq49f3p9-}
heavily relies on the assumption that the norm of a submatrix
within the ``random part'' of $B$ is typically of the same order as the square
root of that submatrix' larger dimension. Establishing a polynomial
upper bound on the growth factor in the presence of a non-random shift
(of polynomial operator norm) is an interesting and challenging problem.

\appendix

\section{Proof of Proposition~\ref{p: small sing shifted}}

In this section, we provide a proof of Proposition~\ref{p: small sing shifted}
which is based on the {\it restricted invertibility} argument of Nguyen \cite{Nguyen}.

The classical restricted invertibility theorem
was derived by Bourgain--Tzafriri \cite{BT};
see, in particular, \cite{V restricted,SS,NY} for extensions and a comprehensive discussion.
In what follows, we will use the estimate from paper \cite{SS}:
\begin{theorem}[{\cite{SS}}]\label{restrInv}
Let $B$ be an $u\times t$ ($t\geq u$) matrix. Then for any $\varepsilon\in(0,1)$ with $\frac{\varepsilon^2 \|B\|_{HS}^2}{\|B\|^2}\geq 1$
there exists a subset $J\subset[t]$ such that
$$
|J|\geq \bigg\lfloor\frac{\varepsilon^2 \|B\|_{HS}^2}{\|B\|^2}\bigg\rfloor \quad\quad  \mbox{ and } \quad \quad
s_{|J|} (B_{[u]\times J})\geq \frac{(1-\varepsilon)\|B\|_{HS}}{\sqrt{t}}.
$$
\end{theorem}

Further, we recall an anti-concentration estimate for orthogonal projections
of random vectors with independent components of bounded density, obtained in \cite{RV} (see also \cite{LPP}
for improved constants).
\begin{theorem}[{\cite{RV}}]\label{densityth}
Let $X$ be a random vector in $\R^\ell$ with independent components having the distribution
densities uniformly bounded by $\rho>0$. Then for every $d\leq \ell$ and every $d$--dimensional
subspace $H\subset \R^\ell$, the distribution density of the orthogonal projection of $X$ onto $H$ is bounded above by $(C\rho)^d$, where $C>0$ is a universal constant.
\end{theorem}

Now, we are ready to prove Proposition~\ref{p: small sing shifted}.
Although the proof essentially repeats the one in \cite{Nguyen}, we provide the full argument for completeness.
The argument works for matrices $M$ with i.i.d
entries with continuous distribution density uniformly
bounded above by $\rho>0$. 

\begin{proof}[Proof of Proposition~\ref{p: small sing shifted}]
In view of the interlacing properties of singular values (see, for example, \cite[Theorem~1.4]{Djalil}),
we can assume without loss of generality that $t=u$.
We fix an index $1\leq i\leq u-1$, and $s\in(0,1]$.
Denote by $\Event$ the event
$$
\bigg\{s_{u-i}(M)\leq \frac{ci\,s}{\sqrt{u}}\bigg\},
$$
where $c=c(\rho)>0$ will be chosen later.
Let $Z^\top$ be the $u\times i$ random orthogonal matrix measurable w.r.t $\sigma(M)$
and such that $\|M\col_j(Z^\top)\|_2\leq \frac{ci\,s}{\sqrt{u}}$, $1\leq j\leq i$, everywhere on $\Event$ (one may take $Z^\top$ as the matrix whose columns are the normalized right singular vectors of $M$
corresponding to $i$ smallest singular values of $M$).

In view of Theorem~\ref{restrInv}, there is a (random) subset $J\subset [u]$
such that
$$
 |J| = \lfloor i/2 \big\rfloor\quad \quad \mbox{ and } \quad \quad
 s_{|J|} (Z_{[i]\times J})\geq \big(1-2^{-1/2} \big) \sqrt{i/u}
$$
everywhere on the probability space.

In particular, since $s_{|J|} (Z_{[i]\times J})>0$, there exists an $i\times |J|$ matrix $Y$ such that $(Z_{[i]\times J})^\top Y=I_i$ and hence $\|Y\|\leq \big(1-2^{-1/2} \big)^{-1} \sqrt{u/i}$ everywhere.

Define a subspace
$$
 H_J:=\spn\{ \col_j(M)\}_{j\in J^c},
$$
and let $P$ be the orthogonal projection onto $H_J^\perp$. We then have
$$
PMZ^\top=
P\big(M_{[u]\times J} (Z_{[i]\times J})^\top + M_{[u]\times J^c} (Z_{[i]\times J^c})^\top\big)=PM_{[u]\times J} (Z_{[i]\times J})^\top,
$$
whence on the event $\Event$,
\begin{align*}
i\,\bigg(\frac{ci\,s}{\sqrt{u}}\bigg)^2\cdot \big(1-2^{-1/2} \big)^{-2}\, \frac{u}{i}
&\geq
\|MZ^\top Y\|_{HS}^2\\
&\geq
\|P MZ^\top Y\|_{HS}^2\\
&=\|P M_{[u]\times J}\|_{HS}^2\\
&=\sum _{j\in J} \dist ^2 (\col_j(M), H_J).
\end{align*}
So, we can write
\begin{align*}
\Prob(\Event)&\leq \sum_{I\subset[u]:\;|I|=\lfloor i/2 \big\rfloor}
\Prob\bigg\{\sum _{j\in I} \dist ^2 (\col_j(M), H_I)
\leq \big(ci\,s\big)^2\cdot \big(1-2^{-1/2} \big)^{-2}\bigg\}\\
&\leq \sum_{I\subset[u]:\;|I|=\lfloor i/2 \big\rfloor}
\sum_{I'\subset I,\,|I'|\geq |I|/2}
\Prob\bigg\{\dist ^2 (\col_j(M), H_I)
\leq \frac{2}{|I|}\big(ci\,s\big)^2\cdot \big(1-2^{-1/2} \big)^{-2},\mbox{ $j\in I'$}\bigg\}.
\end{align*}
To estimate the probabilities on the right hand side of the inequality, we apply
Theorem~\ref{densityth}.
Observe that, in view of Theorem~\ref{densityth},
$$
\Prob\big\{\dist (\col_j(M), H_I)
\leq \sqrt{8}\big(c\sqrt{i}\,s\big)\cdot \big(1-2^{-1/2} \big)^{-1}\big\}
\leq (C'c\rho s)^{|I|},
$$
for some universal constant $C'>0$.
Thus,
$$
\Prob(\Event)\leq (2u)^{\lfloor i/2\rfloor}  (C'c\rho s)^{(\lfloor i/2\rfloor)^2/2}.
$$
The result follows by choosing a sufficiently small $c=c(\rho)$.
\end{proof}

\section{Proofs of the auxiliary results in Section~\ref{akjdfnoifuwbfoiubfo}}

\begin{proof}[Proof of Lemma~\ref{lem: sec5i0i1bound}]
As a verification of \eqref{eq: sec5i1DefBound}, we have
\begin{align*}
	 	r - \lfloor ( 1 - ( 1 + \tilde \varepsilon)^{-\ione-1})r \rfloor 	
		\le &
		(1 + \tilde \varepsilon )^{-\ione-1} r +1  
	\le ( 1+ \tilde \varepsilon)^{-\ione-1}\cdot 4m_s + 1
	\underbrace{<}_{\mbox{ by \eqref{eq: i_1} }} \frac{9L}{ \tilde \varepsilon }. 
	\end{align*}
For the upper bound of $\ione$, we have 
$$
   ( 1+ \tilde \varepsilon)^{-\ione}m_s  \ge \tilde \varepsilon( 1+ \tilde \varepsilon)^{-\ione}m_s  
   \ge L \ge 1,
$$
which implies 
$$
      \ione\log( 1+ \tilde \varepsilon)  \le \log(m_s) \quad \Rightarrow\quad \ione \le \frac{ 2 \log(m_s)}{\tilde \varepsilon},
$$
where we used that $ \log(1+t) \ge \frac{t}{2}$ for $0\le t \le \frac{1}{4}$. 
\end{proof}

\begin{proof}[Proof of Lemma~\ref{lem: irxepsilon}]
First, from the definition of $f_i$ and since $r=f_i(\tilde r)$ we have 
\begin{gather}\label{eq: sec5par0}
\begin{matrix}
& r & \le & \tilde r / \big( 1 + ( 1 + \tilde \varepsilon)^{-i}\big)  & \le&  r + 1 \\
  \Rightarrow& \big(1+(1+ \tilde \varepsilon)^{-i}\big) r  &\le& \tilde r & \le& \big(1+(1+ \tilde \varepsilon)^{-i}\big) r + 2 \\
  \Leftrightarrow&( 1 + \tilde \varepsilon)^{-i}r &\le&  x &\le&  ( 1 + \tilde \varepsilon)^{-i}r + 2
\end{matrix}.
\end{gather}
Thus, the lower bound for $x$ in \eqref{eq: xRange} is obtained. Further, in view of the definition of $\ione$,
and since $r\geq r_i\geq m_s$,
\begin{align*}	
	2 \le L \le  \tilde \varepsilon  & ( 1 + \tilde \varepsilon)^{-
	\ione} m_s 
	\le \frac{1}{20} ( 1 + \tilde \varepsilon)^{-i} m_s 
	 \le \frac{1}{20} ( 1 + \tilde \varepsilon)^{-i} r_i
	\le \frac{1}{20} ( 1 + \tilde \varepsilon)^{-i} r.
\end{align*}
Combining this with the last inequality in \eqref{eq: sec5par0}, we obtain 
$$
x
\le ( 1 + \tilde \varepsilon)^{-i}r + 2 \le \frac{21}{20} ( 1 + \tilde \varepsilon )^{-i} r. 
$$

\medskip

Next, we will verify the conditions on the parameters which appear in Proposition~\ref{prop: sec4Main}.
For convenience, we itemize the rest of the proof according to the inequalities we wish to check:
\begin{enumerate}
	\item $\bf r - \lfloor (1 - ( 1 + \tilde \varepsilon)^{-i}) r \rfloor 
\le x \le r.$
First, the inequality $( 1 + \tilde \varepsilon)^{-i}r \le x$ together with $x$ is an integer implies 
$ x \ge \lceil ( 1 + \tilde \varepsilon)^{-i} r  \rceil$. Since
$$
\forall t \in [0,r],\, 
\lfloor r-t \rfloor + \lceil t \rceil = r,
$$
we have
$$ r - \lfloor (1 - ( 1 + \tilde \varepsilon)^{-i}) r \rfloor
	=\lceil ( 1 + \tilde \varepsilon)^{-i} r  \rceil
	\le  x.
$$
For the upper bound on $x$, by the definition of $\izero$ we have $ ( 1 + \tilde \varepsilon)^{-\izero}m_s  \le  m_s/8 $, which implies 
$$
 x \le \frac{21}{20} ( 1 + \tilde \varepsilon )^{-i} r
\le \frac{21}{20} (1 + \tilde \varepsilon)^{-\izero} \cdot 4m_s
\le m_s \le r.
$$

\item {\bf $\bf \tilde \varepsilon x \ge 4$ and $\bf \tilde \varepsilon ( 1 + \tilde \varepsilon )^{-i}r \ge 2$.}
This follows immediately as
$$
	\tilde \varepsilon x \ge \tilde \varepsilon ( 1 + \tilde \varepsilon)^{-i} r
	\ge \tilde \varepsilon ( 1 + \tilde \varepsilon)^{- \ione} m_s
	\ge L \ge 4. 
$$

\item $\bf 3 ( 1 + \tilde \varepsilon )^{-i-1} r - ( 1+ \tilde \varepsilon )^{-i}r 
	\ge x + 1 + 11 \tilde \varepsilon x.$
Since $ \tilde \varepsilon \le \frac{1}{28}$, we have, in view of the above,
$$ 3 ( 1 + \tilde \varepsilon )^{-i-1} r - ( 1+ \tilde \varepsilon )^{-i}r 
\ge \frac{3}{2} ( 1+ \tilde \varepsilon )^{-i}r \ge \frac{10}{7} x,$$
whereas $x + 1 + 11 \tilde \varepsilon x 
	\le x + 12 \tilde \varepsilon x \le  \frac{10}{7} x$. 
\end{enumerate}
\end{proof}

\begin{proof}[Proof of Lemma~\ref{lem: r_iBound}]
First, by the definition of $r_i$'s we have $r_i= m_s$ for $ i \in [\izero]$. 
Next, assume that the statement is true for all indices in an interval $[i]$, for some $i \in [\izero,\ione]$.
We will show it holds for $i+1$. Fix for a moment any $i'\in[\izero,i]$.
By \eqref{eq: xRange} (applied with $x=r_{i'+1}-r_{i'}= r_{i'+1} - f_{i'}(r_{i'+1})$) we have 
\begin{align*}
r_{i'+1} -r_{i'} \le&
\frac{21}{20} ( 1+ \tilde \varepsilon)^{-i'}r_{i'}
\le \frac{21}{20} ( 1+ \tilde \varepsilon)^{- \izero}\cdot 2m_s.
\end{align*}
The definition of $\izero$ and the conditions on $\tilde\varepsilon$ imply
$\frac{21}{20}(1+ \tilde \varepsilon)^{-\izero}\cdot 2m_s \le \frac{ \tilde \varepsilon}{2} m_s$, and we obtain  
\begin{align*}
r_{i'+1} -r_{i'} \le \frac{ \tilde \varepsilon m_s}{2} (1+ \tilde \varepsilon)^{-(i'-\izero)}.
\end{align*}
Finally, summing up over all $i'\in[\izero,i]$, we obtain
$$
r_{i+1}- r_{\izero}=\sum_{i'\in[\izero,i]}\big(r_{i'+1}-r_{i'}\big)
\leq \sum_{i'\in[\izero,i]}\frac{ \tilde \varepsilon m_s}{2} (1+ \tilde \varepsilon)^{-(i'-\izero)}
\leq m_s.
$$
The result follows.
\end{proof}

\begin{proof}[Proof of Lemma~\ref{lem: Probhx}]
For the first summand, from the definition of $h_s(i)$ and since
$x \ge ( 1 + \tilde \varepsilon )^{-i} r \ge ( 1 + \tilde \varepsilon )^{-i}m_s$ (see Lemma~\ref{lem: irxepsilon}), we get
\begin{align} \label{eq: Probhx0}
		x\, h_s(i)^{ \tilde \varepsilon x / 64 } 
	\le
		x\exp \bigg( - \frac{128\, \log n }{ \tilde \varepsilon ^2 
				( 1 + \tilde\varepsilon )^{-i} m_s } 
			\frac{\tilde \varepsilon x}{ 64 } \bigg)
	\le
		x\exp ( - 2 \log n )
	= n^{-1},
	\end{align}
where the last inequality follows as $x \le \tilde{r} \le m_{s+2}\le n$. Thus, 
\begin{align} \label{eq: lemHiEq0}
		2x^{ \tilde \varepsilon x/ 2}
		h_s(i)^{(\tilde\varepsilon x)^2/64}
	=
		2 \big( x h_s(i)^{ \tilde \varepsilon  x / 32 } 
		\big)^{ \tilde \varepsilon x /2 } 
	\underbrace{\le}_{\mbox{by \eqref{eq: Probhx0}}}
    2 n^{- \tilde \varepsilon x / 2}
    h_s(i)^{ ( \tilde \varepsilon x )^2 /128 }.
\end{align}
As $ \tilde \varepsilon x/2 \ge 2 $ (see Lemma~\ref{lem: irxepsilon}), for $n \ge 2$ we have 
$$
2x^{ \tilde \varepsilon x/ 2}
		h_s(i)^{(\tilde\varepsilon x)^2/64} \le
	 	\frac{1}{2} h_s(i)^{ ( \tilde \varepsilon x )^2 /128 }. 
$$
For the second and third summands, since $ ( 1 + \tilde \varepsilon )^{-i} r \ge \frac{ 20 }{ 21 } x$,
we obtain
\begin{align} \label{eq: lemHiEq1}
	 4\exp\big(-cx^2\,\tilde\varepsilon/h_s(i)^2\big) &
	+\tilde C\exp\big(-c\tilde\varepsilon^2 (1+\tilde\varepsilon)^{-i}r x/h_s(i)^2\big) \nonumber \\
\le 
	& (4+ \tilde C) \exp \Big( - \frac{20}{21}  \frac{c}{ h_s(i)^2} (\tilde \varepsilon x)^2  \Big) .
\end{align}
Now, we impose first constraint on $h_s(i)$:
$$
 4 + \tilde C \le \frac{1}{2} \exp\Big( \frac{c}{h_s(i)^2} \Big),
$$
which can be achieved by taking $C_h$ large enough depending only on $\tilde C$ and $c$.
Then, since $ \tilde \varepsilon x \ge 4$,
$$
(4+ \tilde C) \exp \Big( - \frac{20}{21}  \frac{c}{ h_s(i)^2} (\tilde \varepsilon x)^2  \Big) 
\le 
\frac{1}{2}
\exp \Big(  \Big(\frac{1}{16} - \frac{20}{21} \Big)   \frac{c}{ h_s(i)^2} (\tilde \varepsilon x)^2  \Big)
\le 
\frac{1}{2}
\exp \Big(  -  \frac{c}{ 2h_s(i)^2} (\tilde \varepsilon x)^2  \Big).
$$
Next, as the function $t \mapsto t^2 \log(1/t)$ converges to $0$ when $t \searrow 0$, we can suppose by taking $C_h$ large enough that
$$
    h_s(i)^2 \log( 1/h_s(i)) \le 64c, 
$$
which implies 
$$
\frac{1}{2}\exp \Big( -   \frac{c }{ 2h_s(i)^2} (\tilde \varepsilon x)^2 \Big)
\le \frac{1}{2} \Big(  - \frac{\log( 1/h_s(i)) }{128} (\tilde \varepsilon x)^2\Big) 
= \frac{1}{2}(h_s(i))^{ (\tilde \varepsilon x)^2/128 }.
$$
Finally, combining the last estimate with \eqref{eq: lemHiEq0}, we obtain 
$$
		2x^{\tilde\varepsilon x/2}\,
h_s(i)^{(\tilde\varepsilon x)^2/64}
+4\exp\big(-cx^2\,\tilde\varepsilon/h_s(i)^2\big)
+\tilde C\exp\big(-c\tilde\varepsilon^2 (1+\tilde\varepsilon)^{-i}r x/h_s(i)^2\big)
	\le h_s(i)^{ ( \tilde \varepsilon x)^2/128 },
$$
and the lemma follows since 
$ h_s(i) \le 
	\exp  \big( - \frac{128 \alpha  \log n }{
		\tilde \varepsilon^2 ( 1+ \tilde \varepsilon )^{-i} m_s } \big)
$ and $ x \ge ( 1 + \tilde \varepsilon )^{-i}r \ge ( 1 + \tilde \varepsilon )^{-i} m_s$. 
\end{proof}

\begin{proof}[Proof of Lemma~\ref{fekjfnwefkjnfpwijfnwelkfjn}]
  First, as $\izero \le 2\log(m_s)/ \tilde \varepsilon \le 2 \log(n) / \tilde \varepsilon$ (see \eqref{eq: sec5i0i1bound}) and in view of the definition of $g_s(i)$, we have
\begin{align} \label{eq: sec5g(i0)}
	g_s( \izero ) 
		= &
			\frac{ c' }{2 \sqrt{m_s}} 16^{-\izero} m_s n^{-\alpha}
		\ge \frac{ c' }{2}n^{-1/2} n^{-2\log(16)/ \tilde \varepsilon} n^{-\alpha} 
		\ge \frac{ c' }{2}n^{-7/\tilde \varepsilon - \alpha}.
	\end{align}
To estimate $g_s(\ione+1)$, we write
\begin{align} \label{eq: gEstimate}
		\frac{ g_s(\ione+1)} {g_s(\izero)}
 	= 
		\prod_{i=\izero}^{\ione} \frac{ g_s(i+1)}{g_s(i)}
		\ge 	
		\prod_{i=\izero}^{\ione} 
		\Big(\frac{c' \tilde \varepsilon}{32}
		 h_s(i)^5\Big).
	\end{align}
By the definition of $ h_s(i)$, 	
\begin{align}
		\forall i \in [\izero,\ione], \, \quad
		h_s(i) \ge  &
		\exp  \Big( - 
			  \frac{128 \alpha \log n }{
		\tilde \varepsilon^2 ( 1+ \tilde \varepsilon )^{-i} m_s }
		 - C_h
	\Big).
	\end{align}
Returning to \eqref{eq: gEstimate}, we obtain 
\begin{align*}
	\frac{ g_s(\ione+1)} {g_s(\izero)}
	\ge &
	\Big(\frac{c' \tilde \varepsilon}{32}  e^{- 5C_h } \Big)
	^{ \ione} 
	\exp\Big( 
		-  \frac{ 5 \cdot 128  \alpha\log n }{
		\tilde \varepsilon \cdot  \tilde \varepsilon ( 1+ \tilde \varepsilon )^{-\ione} m_s }
		 \;\sum_{j=0}^{\infty} (1+ \tilde \varepsilon)^{-j} 
		\Big) \\
		\ge &
		\Big(\frac{c' \tilde \varepsilon}{32}  e^{- 5C_h } \Big)^{ 2 \log(n) / \tilde \varepsilon}
		\exp\Big( 
		- \frac{ 5 \cdot 128  \alpha \log n }{
		\tilde \varepsilon }
		  \frac{1+ \tilde \varepsilon}{ \tilde \varepsilon} 
		\Big), \\
	\end{align*}
	where we used $\ione \le 2 \log(m_s)/ \tilde \varepsilon \le 2 \log(n) / \tilde \varepsilon $, $ \tilde \varepsilon ( 1+ \tilde \varepsilon)^{-\ione}m_s \ge L \ge 1$, and $ \sum_{j=0}^\infty ( 1 + \tilde \varepsilon)^{-j} 
	= \frac{1+ \tilde \varepsilon }{ \tilde \varepsilon }$.  
	Together with \eqref{eq: sec5g(i0)} and since $\alpha \ge 1$, we can simplify the bound to have the form 
	$$
		g_s(\ione+1) \ge n^{-C( \tilde \varepsilon) \alpha }.
	$$
\end{proof}

\end{document}